%% file: main.tex
\documentclass[journal]{AIAA}
\input{package.tex}
\input{notation.tex}

\input{color.tex}
\newboolean{includefigures}
\newboolean{revisionmarkups}

\usepackage[normalem]{ulem}
\useunder{\uline}{\ul}{}

\graphicspath{{"Figures"}}
\setboolean{includefigures}{true}
\setboolean{revisionmarkups}{false}

\ifthenelse{\boolean{revisionmarkups}}
{
    \usepackage[authormarkup=none, commentmarkup=todo, todonotes={textsize=tiny}]{changes}
    \definechangesauthor[name=Amlan Sinha]{AS}
}{
    \usepackage[final]{changes}
}

\newcommand{\orcid}[1]{\href{https://orcid.org/#1}{\includegraphics[scale=0.25]{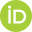}}}

\title{Bi-Level Optimal Control Framework For Missed-Thrust-Design With First-Order Bounds On Maximum Missed-Thrust-Duration
\footnote{Builds on previous work which was presented as Paper AAS 25-697 at the AAS/AIAA Astrodynamics Specialist Conference, Boston, MA, August 10-14 2025.}
}

\author{
Amlan Sinha 
\orcid{0009-0001-4572-7583} 
\footnote{Corresponding Author: \href{mailto:amlans@princeton.edu}{amlans@princeton.edu}} 
\footnote{
Ph.D. Candidate, Mechanical and Aerospace Engineering, Princeton University, Princeton, NJ 08540.
Member AIAA.
} 
and 
Ryne Beeson 
\orcid{0000-0003-2176-0976} 
\footnote{
Assistant Professor, Department of Mechanical and Aerospace Engineering, Princeton University, Princeton, NJ 08540.
Member AIAA.
}
}

\affil{Princeton University, Princeton, NJ 08840}

\ifthenelse{\boolean{revisionmarkups}}
{
\usepackage{datetime}
\newcommand{\asnotes}[1]{\textcolor{DodgerBlue}{[#1]}}
\newcommand{\astodos}[1]{\textcolor{Orange}{[AS on \today @ \currenttime: #1]}}
\newcommand{\asmarker}[1]{\textcolor{Fuchsia}{[AS on \today @ \currenttime: *** #1 ***]}}
}
{
\newcommand{\asnotes}[1]{}
\newcommand{\astodos}[1]{}
\newcommand{\asmarker}[1]{}
}

\begin{document}

\maketitle

% \linenumbers

\input{sections/abstract}
\input{sections/nomenclature}
\input{sections/introduction}
\input{sections/problem_formulation}
\input{sections/methodology}
\input{sections/theoretical_results}
\input{sections/experimental_results}
\input{sections/conclusion}
\input{sections/acknowledgement}

\bibliography{references}

\newpage
\input{sections/appendix}

\end{document}

%% file: package.tex
\usepackage[utf8]{inputenc}
\usepackage{graphicx,float,placeins,tikz,onimage,caption,subcaption,xcolor}
\usepackage{physics,amsmath,mathtools,amsfonts,mathrsfs,bm}

\let\oldopenbox\openbox
\let\openbox\relax
\usepackage{amsthm}
\let\openbox\oldopenbox

\let\oldBbbk\Bbbk
\let\Bbbk\relax
\usepackage{amssymb}
\let\Bbbk\oldBbbk

\usepackage{longtable,tabularx,booktabs,multirow}
\usepackage[version=4]{mhchem}
\usepackage{siunitx}
\usepackage{textcomp}
\usepackage[pagewise]{lineno}
\usepackage{wasysym}
\usepackage{ifthen}
\usepackage{enumitem}
\usepackage{url}
\PassOptionsToPackage{colorinlistoftodos,prependcaption,textsize=normalsize}{todonotes}
\usepackage{todonotes}

\usepackage{listings}
\lstdefinestyle{codestyle}{
    commentstyle=\color{codegreen},
    keywordstyle=\color{blue},
    numberstyle=\tiny\color{gray},
    stringstyle=\color{codepurple},
    basicstyle=\fontfamily{pcr}\selectfont\footnotesize,
    breakatwhitespace=false,
    breaklines=true,
    captionpos=b,
    keepspaces=true,
    numbers=left,
    numbersep=5pt,
    showspaces=false,
    showstringspaces=false,
    showtabs=false,
    tabsize=2
}

%% file: notation.tex
\newtheorem{theorem}{Theorem}

\newtheorem{lemma}{Lemma}

\newcommand{\reals}{\mathbb{R}}

\newcommand{\integers}{\mathbb{Z}}

\newcommand{\mF}{\mathcal{F}}
\renewcommand{\P}[1][P]{\mathbb{#1}}

\DeclareMathOperator*{\argminn}{arg\,min}

%% file: color.tex
\definecolor{White}{RGB}{255, 255, 255}
\definecolor{Silver}{RGB}{192, 192, 192}
\definecolor{Grey}{RGB}{128, 128, 128}
\definecolor{Black}{RGB}{0, 0, 0}
\definecolor{Tomato}{RGB}{255, 99, 71}
\definecolor{Red}{RGB}{255, 0, 0}
\definecolor{Maroon}{RGB}{128, 0, 0}
\definecolor{Magenta}{RGB}{172, 87, 156}
\definecolor{Yellow}{RGB}{255, 255, 0}
\definecolor{Tan}{RGB}{205, 181, 145}
\definecolor{Olive}{RGB}{128, 128, 0}
\definecolor{Lime}{RGB}{0, 255, 0}
\definecolor{Green}{RGB}{0, 128, 0}
\definecolor{Chartreuse}{RGB}{160, 252, 78}
\definecolor{Aqua}{RGB}{0, 255, 255}
\definecolor{SkyBlue}{RGB}{83, 194, 236}
\definecolor{Teal}{RGB}{0, 128, 128}
\definecolor{Blue}{RGB}{0, 0, 255}
\definecolor{DodgerBlue}{RGB}{85, 154, 248}
\definecolor{RoyalBlue}{RGB}{65, 105, 225}
\definecolor{NavyBlue}{RGB}{25, 25, 107}
\definecolor{TealBlue}{RGB}{84, 113, 171}
\definecolor{Navy}{RGB}{0, 0, 128}
\definecolor{Fuchsia}{RGB}{232, 51, 244}
\definecolor{Purple}{RGB}{128, 0, 128}
\definecolor{Orange}{RGB}{242, 132, 52}
\definecolor{Charteuse}{RGB}{204, 255, 0}

\definecolor{colunstablethreefour}{RGB}{248, 130, 83}
\definecolor{colstablethreefour}{RGB}{136, 183, 213}
\definecolor{colunstablefivesix}{RGB}{166, 135, 186}
\definecolor{colstablefivesix}{RGB}{130, 186, 126}

\definecolor{codegreen}{RGB}{0, 153, 0}
\definecolor{codepurple}{RGB}{148, 0, 209}

\definecolor{python_deep_blue}{RGB}{27, 106, 148}
\definecolor{python_deep_orange}{RGB}{228, 116, 50}

%% file: sections/abstract.tex
\begin{abstract} % 199 words
In this paper, we present a bi-level optimal control framework for designing low-thrust spacecraft trajectories with robustness against missed-thrust-events. 
The \emph{upper-level} (UL) problem generates a nominal trajectory assuming full control authority, while each \emph{lower-level} (LL) problem computes the optimal recovery maneuver following a missed-thrust-event along the nominal solution. 
Under suitable regularity conditions ensuring uniqueness and smoothness of the LL response, the hierarchy admits a single-level reformulation by embedding the LL first-order optimality conditions within the UL constraints.
We further establish a robustness certificate, which provides an upper bound on the maximum admissible missed-thrust-duration for which the structural assumptions remain valid for the LL problem. 
The bound depends explicitly on precomputable dynamical quantities along the nominal solution, enabling rapid evaluation over large ensembles without iterative solves. 
Numerical experiments show that while the certificate identifies when modeling assumptions are valid, it does not fully characterize recoverability after missed-thrust-events.
A finite-horizon controllability-energy analysis is therefore used to interpret recovery beyond the theoretical bounds.
Collectively, these results provide a deterministic, certifiable approach for incorporating robustness directly into trajectory design, replacing post-hoc margin allocation techniques with formal guarantees.
\astodos{200 Words}
\asnotes{If you want to remove the notes, please recompile the `main.tex' file with the boolean \texttt{revisionmarkups} turned off}
\end{abstract}

%% file: sections/nomenclature.tex
\section*{Nomenclature}

{\renewcommand\arraystretch{1.0}
\noindent\begin{longtable*}{@{}l @{\quad=\quad} l@{}}

% \multicolumn{2}{@{}l}{\textbf{Probability Space and Uncertainty}}\\
$\Omega$                    & sample space of uncertainty realizations \\
$\mF$                       & $\sigma$-algebra of measurable events \\
$(\mF_t)_{t \ge 0}$         & filtration representing information over time \\
$\P$                        & probability measure on $(\Omega,\mF)$ \\
$\omega$                    & uncertainty realization / index for a missed thrust event \\
$\tau_1(\omega)$            & start time of MTE under realization $\omega$ \\
$\tau_2(\omega)$            & end time of MTE under realization $\omega$ \\
$\Delta \tau(\omega)$       & missed thrust duration for realization $\omega$ \\

% \multicolumn{2}{@{}l}{\textbf{Dynamics, States, Controls}}\\
$\bm{\xi}$                  & spacecraft state (reference or realization) \\
$\bm{\xi}^\dagger$          & reference (nominal) state trajectory \\
$\bm{\xi}^\omega$           & realization state trajectory \\
$\tilde{\bm{\xi}}$          & state deviation $\tilde{\bm{\xi}}=\bm{\xi}^\omega - \bm{\xi}^\dagger$ \\
$\bm{u}$                    & control input \\
$\bm{u}^\dagger$            & reference (nominal) control input \\
$\bm{u}^\omega$             & realization control input \\
$\tilde{\bm{u}}$            & control deviation $\tilde{\bm{u}}=\bm{u}^\omega - \bm{u}^\dagger$ \\
$\bm{f}$                    & uncontrolled (natural) dynamics \\
$\bm{g}$                    & controlled dynamics \\
$A_t$                       & Jacobian $\nabla_{\bm{\xi}} \bm{f}$ along reference trajectory at time $t$ \\
$B_t$                       & Jacobian $\nabla_{\bm{u}} \bm{f}$ along reference trajectory at time $t$ \\
$r(t,\bm{\xi})$             & second-order Taylor remainder of nonlinear dynamics \\
$\Phi(t,s)$                 & state-transition matrix of linearized dynamics \\

% \multicolumn{2}{@{}l}{\textbf{Times and Segmentation}}\\
$t_0$                       & initial time \\
$t_f$                       & final time \\
$T$                         & time of flight \\

% \multicolumn{2}{@{}l}{\textbf{Optimization Quantities}}\\
$J$                         & objective function (UL cost) \\
$\phi$                      & terminal cost \\
$\mathcal{L}$               & running cost \\
$\bm{x}$                    & NLP decision vector \\
$\bm{c}$                    & NLP constraint vector \\
$\mathcal{E}$               & index set of equality constraints \\
$\mathcal{I}$               & index set of inequality constraints \\

% \multicolumn{2}{@{}l}{\textbf{Bilevel and LL (LQR) Quantities}}\\
$\bm{\mu}$                  & LL costate vector (adjoint variable) \\
$Q,\,Q_f$                   & state and terminal cost weighting matrices \\
$R$                         & control penalty matrix \\
$O_\omega$                  & outage interval(s) for realization $\omega$ \\
$A_\omega$                  & control-available interval(s) for realization $\omega$ \\

% \multicolumn{2}{@{}l}{\textbf{Constants and Bounds (Theoretical Analysis)}}\\
$\alpha$                    & uniform bound on $\|\nabla_{\bm{\xi}} \bm{f}\|$ along reference \\
$\beta$                     & uniform bound on $\|\nabla_{\bm{u}} \bm{f}\|$ along reference \\
$H$                         & uniform bound on $\|\nabla^2_{\bm{\xi}} \bm{f}\|$ \\
$u^\dagger_{\min}$, $u^\dagger_{\max}$ & lower/upper bounds on nominal thrust magnitude \\
$f_{u^\dagger}^{\min}$, $f_{u^\dagger}^{\max}$ & lower/upper bounds on effective forcing amplitude $\beta u^\dagger_{\min/\max}$ \\
$\delta$                    & admissible safe state-space error radius \\
$\delta\tau_{\max}$         & maximum allowable missed thrust duration ensuring linear-approximation validity \\

\multicolumn{2}{@{}l}{\textbf{Superscripts}}\\
$\dagger$                   & reference (nominal) solution \\
$\omega$                    & realization solution \\

\multicolumn{2}{@{}l}{\textbf{Subscripts}}\\
$i,j$                       & indices for segments, constraints, nodes, etc. \\
$t$                         & time variable or time index \\

\end{longtable*}
}

%% file: sections/introduction.tex
\section{Introduction} \label{sec:introduction}

\lettrine{U}{niquely} high effective mass ratios render \emph{low-thrust} (LT) propulsion systems invaluable to many large-scale flagship scientific missions (e.g., Hayabusa, Dawn, Bepi-Colombo, Psyche) as well as to many small-scale technology demonstration missions (e.g., NEA Scout, Mars Cube One, Lunar Flashlight). 
Although these systems offer substantial advantages, their benefits are tempered by their susceptibility to temporary losses in control authority, which, given the inherent limitations in the attainable maximum thrust, can significantly degrade mission performance.
When operational anomalies (e.g., mechanical failure, power systems anomalies, or collision with space debris) arise, a spacecraft typically enters a protective safe mode wherein all thruster operations are shut down.
If the resulting downtime intersects with a scheduled thrust arc, the resulting phenomenon is classified as a \emph{missed thrust event (MTE)}.
Because LT solutions typically exhibit long-duration thrust arcs, there is a higher likelihood of such interruptions occurring, thereby making them relatively common in historical LT missions~\cite{imken_modeling_2018}.
Unless specifically accounted for during preliminary mission design, such interruptions can significantly degrade mission performance or even result in mission failure in some cases, especially when critical maneuvers (e.g., those occurring during flybys) become impossible. 
These risks underscore the importance of incorporating robustness against such unforeseen thruster outages directly into the trajectory optimization process during the preliminary mission design phase.
\asnotes{Introducing the missed thrust problem}

The robust optimal control paradigm seeks control policies which yield feasible, or near-optimal solutions, even in the presence of various uncertainties (e.g., \emph{aleatoric uncertainties} such as navigational errors, \emph{epistemic uncertainties} such as imperfect parameter knowledge, and \emph{maneuver execution errors}). 
Although the topic is relevant to the broader controls community in general, it is particularly crucial in aerospace applications, where even the slightest deviations can jeopardize mission success. 
In current practice, robustness is often imposed a posteriori by adding empirical margins, lowering duty cycles, or inserting coast arcs at dynamically sensitive locations along the trajectory during the preliminary mission design phase, frequently at the expense of optimality.
Recent research has thus increasingly focused on incorporating uncertainty directly within the optimization problem itself. 
Existing approaches in this direction can be broadly classified into two main categories: \emph{deterministic} and \emph{stochastic}. 
While deterministic approaches typically consider a specific uncertainty realization (e.g., worst case), stochastic approaches incorporate a probabilistic representation (e.g., chance constraints).
In either case, however, the resulting framework enables systematic balancing between robustness and performance to varying degree of fidelity and success.
Notable stochastic methods in the literature include stochastic differential dynamic programming, which leverages linear feedback control policies and employs unscented transform methods for uncertainty propagation~\cite{ozaki_stochastic_2018}; tube stochastic optimal control, which generalizes the prior method to nonlinear dynamics and constraints by sequentially approximating pertinent stochastic processes using tractable surrogates (e.g., Gaussian)~\cite{ozaki_tube_2020}; and convex optimization frameworks which have been successfully applied to both impulsive and LT trajectory design problems~\cite{oguri_robust_2021,oguri_stochastic_2022}. 
Covariance steering methodologies~\cite{okamoto_optimal_2018,yi_nonlinear_2020} and distribution agnostic belief optimal control approaches~\cite{greco_direct_2020,greco_robust_2022} further enrich the stochastic toolbox.
While these methods are quite effective in handling aleatoric and epistemic uncertainties, they fall short in addressing maneuver execution errors, which constitute a critical challenge in LT missions. 
Consequently, this study adopts the term `robust problems' to denote a particular category of trajectory optimization problems where the goal is to construct solutions with explicit robustness against maneuver execution errors.
\asnotes{Describing how the missed thrust problem fits within the general robust optimal control problem paradigm}

Depending on the respective approaches, two principal schools of thought emerge in existing research: \emph{missed thrust analysis} and \emph{missed thrust design}.
In missed thrust analysis, a nominal solution is first constructed without any consideration for MTEs, after which any necessary robustness adjustments are incorporated within the baseline trajectory.
In practice, such adjustments include redesigning the nominal solution with carefully chosen coast phases, or heuristically lower duty cycles, as exemplified by the Dawn and Psyche missions which incorporated such measures to mitigate performance degradations~\cite{rayman_coupling_2007,oh_analysis_2008,madni_missed_2020}. 
Building on this idea, Laipert et al. investigated how multiple missed thrust occurrences can impact a nominal solution through extensive Monte Carlo simulations introducing introducing simple yet useful metrics such as lateness and propellant margin to quantify performance degradation~\cite{laipert_automated_2015,laipert_monte_2018}.
However, these retrospective adjustments inherently decouple the uncertainty quantification process from the trajectory optimization process, potentially shifting the sensitivity to a different location along the trajectory rather than fully eliminating risk.
In contrast, missed thrust design approaches directly integrate a robustness criterion into the optimization problem, either as part of its objective or constraint functions. 
The state-of-the-art deterministic missed thrust design strategies extend the missed thrust problem into a higher dimensional space, in which a \emph{reference} trajectory and multiple \emph{realization} trajectories (each representing a potential recovery path from a discrete missed thrust occurrence) are simultaneously optimized~\cite{mccarty_missed_2020,venigalla_low_2020}, with explicit constraints on relevant performance metrics e.g., missed thrust recovery margins.
Although these approaches tend to be quite effective, their computational complexity scales rapidly with the number of missed thrust occurrences.
Accurately capturing the intrinsic stochasticity would require considering an infinite set of possible realizations, rendering such approaches computationally prohibitive.
Therefore, with such approaches, there is an important tradeoff between \emph{modeling fidelity}, i.e., how thoroughly we explore failure modes, and \emph{computational efficiency}, i.e., how quickly we can solve these problems, highlighting the need for an efficient computational framework to solve missed thrust design problems using similar deterministic approaches. 
Notable stochastic missed thrust design strategies, on the other hand, include two-level stochastic optimal control formulations treating missed thrust occurrences as purely stochastic processes~\cite{olympio_designing_2010}, but remain relatively underexamined within the research community.
Recent research efforts have also explored data-driven methods, including neural networks~\cite{rubinsztejn_neural_2020,izzo_real-time_2021} and reinforcement learning controllers~\cite{miller_low-thrust_2019,zavoli_reinforcement_2021}, to learn optimal recovery maneuvers. 
However, such approaches typically offer only local validity and may not generalize well to chaotic multibody gravitational environments.
An alternative and promising research direction in robust trajectory design involves leveraging dynamical structures inherent to multibody gravitational environments. 
In an early study by Alizadeh and Villac, it was shown that by rewriting the objective function to penalize the deviation of solutions from the natural dynamics, it was possible to induce robustness in the resulting solutions~\cite{alizadeh_sensitivity_2013}.
Building on this idea, Sinha and Beeson recently explored the statistical correlation between robust solutions and invariant manifolds in three body problems, establishing a quantitative relationship between the two and highlighting their potential in robust trajectory design~\cite{amlans_dynamical_2025j}.
Irrespective of robustness considerations, efficient exploration of the solution space remains a critical challenge in robust trajectory design due to their inherently high dimensionality. 
The efficacy of gradient-based solvers depends on providing initial guesses which lie within basins of attraction of the local optima.
In practice, determining effective distributions for these initial guesses poses considerable difficulty especially when mission parameters may evolve on a rapid cadence during the preliminary mission design phase. 
Leveraging insights into the geometric structure within robust solutions, Sinha and Beeson introduced a novel initial-guess generation strategy that projects solutions from simpler surrogate robust problems onto the more complex robust problem in consideration, preserving inherent dynamical relationships while significantly enhancing convergence speed and solution diversity~\cite{amlans_initial_2025j}. 
Nevertheless, each method above suffers from fundamental challenges, including limited scalability, reliance on local optimization, or overly restrictive assumptions regarding uncertainty characterization.
Building on the theoretical and algorithmic foundations in the bi-level optimal control (BOC) literature, this study aims to at least partially address these limitations by proposing a novel algorithmic framework which enables efficient solutions to missed thrust problems using a deterministic missed thrust design approach.
\asnotes{Distinguishing between missed thrust analysis and missed thrust design, and (briefly) describing our own work}

The BOC problem is a hierarchical mathematical problem in which an \emph{upper level} (UL, the \emph{leader}) problem specifies the overarching objectives and constraints (e.g., mission objectives, operational constraints, and an admissible class of control policies), while a \emph{lower level} (LL, the \emph{follower}) solves another optimal control problem conditional on the UL decision and, in the most general case, subject to additional constraints specific to the LL.
A solution is admissible only if there exists an optimal response at the LL which, together with the UL decision, satisfies all constraints at both levels.
The versatility of this framework is evidenced by its applications across various dynamical systems, including inverse optimal control problems \asnotes{(insert appropriate citations)}, robotic locomotion problems \asnotes{(insert appropriate citations)}, and hierarchical control tasks within complex engineering systems \asnotes{(insert appropriate citations)}.
\asnotes{Introducing bi-level optimal control paradigm}

Foundational results on existence, uniqueness and algorithmic techniques are well established in the bi-level optimization literature~\cite{dempe_2002_foundations,dempe_2020_bilevel}.
Broadly, the algorithmic approaches can be classified into two categories: \emph{explicit} or \emph{implicit} methods~\cite{dempe_2002_foundations}. 
Explicit methods reformulate the bi-level optimization problem into a single-level problem by embedding the LL optimality conditions as explicit constraints in the UL problem, thereby allowing the use of off-the-shelf numerical solvers. 
Common examples include \emph{Karush-Kuhn-Tucker} (KKT) or \emph{Fritz-John} (FJ) reformulations, value function approaches, and penalty or complementarity relaxations. 
However, as explored in more depth in \S~\ref{subsection:methodology:single-level reformulation ll optimality conditions}, such reformulations can introduce computational complexities due to their inherently non-smooth nature.
Conversely, implicit approaches describe the LL optimal solution as an implicit function of the UL decision variables, so that relevant gradients can be computed via sensitivity analyses or surrogate models. 
Examples include \emph{implicit function theorem} (IFT) differentiation, adjoint or variational methods for optimal control problems, surrogate modeling, etc. 
When the LL sub-problem is well-posed, these approaches typically yield smoother numerical behavior, but their effectiveness depends critically on the availability of a smooth implicit representation of the LL solution, which can be challenging when the LL sub-problem itself is an optimal control problem.
\asnotes{Describing some methods to solve bi-level optimal control problems}

Regardless of the algorithmic approach, however, a bi-level framework offers a natural and flexible way to handle the missed thrust problem. 
Within this framework, the UL problem optimizes the nominal spacecraft trajectory assuming full control authority (i.e., without any thruster interruptions) with mission objectives (e.g., propellant usage, flight time, etc.) and operational constraints (e.g., path constraints, boundary conditions, etc.).
On the other hand, the LL sub-problem seeks a recourse after a missed thrust scenario by solving a subordinate optimal control problem which returns the best recovery action conditional on the nominal solution by the UL problem. 
By embedding either a value function which bounds worst-case performance loss or a recovery margin which quantifies distance to recourse infeasibility within the UL objective and/or constraints, it is possible to guarantee robustness by construction.
The bi-level framework, therefore, supports verifiable robustness certificates, replacing post-hoc empirical checks with ex-ante guarantees.

The bi-level framework confers several algorithmic advantages with direct computational benefits. 
First, it offers substantial modeling flexibility which can help balance physical fidelity and computational tractability e.g., if there are LL sub-problem instances which represent minor missed thrust occurrences, they can be adaptively pruned to improve computational efficiency without any significant change to the overarching problem structure.
Second, a direct transcription of the bi-level formulation naturally produces sparse nonlinear programs that are well-suited to first- and second-order methods (e.g., sequential quadratic programming), thereby enabling systematic exploitation of sparsity patterns and block structure.
Third, when admissible, single-level reformulations permit the direct use of standard off-the-shelf nonlinear programming solvers, thereby reducing implementation effort and leveraging well-established numerical reliability.
Collectively, this paradigm delivers formal robustness guarantees along with a scalable algorithmic framework to enable efficient solutions to the missed thrust design problems via deterministic approaches.
\asnotes{Motivating the use of bi-level optimal control to solve the missed thrust design problem due to the potential theoretical and algorithmic benefits in closing existing research gaps}

By posing the missed thrust design problem as a bi-level optimal control problem, we establish an efficient framework with formal robustness guarantees. 
The principal contributions of this work are as follows:
\begin{enumerate}[label=\arabic*.]
    \item 
    We formulate the missed thrust design problem as a BOC problem, where the UL problem generates the nominal reference trajectory under full control authority, and each LL sub-problem characterizes the optimal recovery maneuver following an MTE. 
    Under suitable regularity conditions, we further show that this BOC formulation admits a single-level representation by augmenting the UL constraints with the first-order optimality conditions of the LL sub-problems.
    \item We introduce a robustness certification in two complementary phases.
    First, we establish a theoretical upper bound on the maximum admissible missed thrust duration for which the single-level reformulation remains appropriate, thereby providing a rigorous robustness certificate that directly connects the severity of the MTE with the structural properties of the underlying optimal control problem.
    Second, we introduce a finite-horizon controllability energy analysis that serves as a diagnostic tool to interpret recovery feasibility after the MTE, particularly for regimes where the theoretical conditions for linear model are no longer satisfied.
    \item We validate the theoretical results on a nonlinear relative motion scenario involving a chief-deputy spacecraft pair, demonstrating the applicability of the proposed framework and verifying the tightness of the derived robustness bounds. \asnotes{The three main contributions of the paper}
\end{enumerate}

The remainder of the paper is organized as follows.
In \S~\ref{sec:problem formulation}, we introduce the general missed thrust design problem which is central to this study.
In \S~\ref{sec:methodology}, we detail the methodology for posing the missed thrust design problem as a bi-level problem and demonstrate how, under suitable regularity conditions, it can be expressed as a single-level optimal control problem.
In \S~\ref{sec:theoretical bounds on maximum missed thrust duration}, we present a critical theoretical result which establishes a theoretical upper bound on the maximum missed thrust duration under which these regularity conditions remain valid.
In \S~\ref{sec:experimental results relative spacecraft motion}, we present a numerical experiment in which this framework is applied to a rendezvous problem between two spacecrafts.
Finally, in \S~\ref{sec:conclusion}, we summarize the main contributions of this work and outline potential directions for future research.
\asnotes{The organization of paper}
\astodos{Need to revise}

%% file: sections/problem_formulation.tex
\section{Problem Formulation} \label{sec:problem formulation}

In this section, we formulate the missed thrust design problem as a robust optimal control problem. 
A full derivation from the general robust optimal control framework can be found in Sinha and Beeson~\cite{amlans_dynamical_2025j}. 
For the reader's convenience, however, the specific problem formulation relevant to the present study is summarized below.

Let $(\Omega, \mF, \P)$ be a probability space and $\omega \in \Omega$ a random sample.
We aim to find an extremal control solution $\bm{u}^* \in \mathcal{U}^\Omega$, with $\mathcal{U}^\Omega$ an admissible control set, to minimize the Bolza-type cost functional,
\begin{align} \label{eq:robust MTD}
    \min_{\bm{u} \in \mathcal{U}^\Omega} \{ J(\bm{u}^\dagger) \equiv \phi(\bm{\xi}^\dagger_1) + & \int_0^1 \mathcal{L}(s, \bm{\bm{\xi}^\dagger_s}, \bm{u}^\dagger_s) ds \},
\end{align}
such that the dynamical constraints with deterministic boundary conditions, $\Xi_0, \Xi_1$ given by Eq.~\eqref{eq:robust MTD reference dynamics equation} is satisfied for the reference solution, 
\begin{align} \label{eq:robust MTD reference dynamics equation}
    \bm{\xi}^\dagger_t = \bm{\xi}^\dagger_0 + \int_0^t \bm{f}(s, \bm{\xi}^\dagger_s) ds &+ \int_0^t \bm{g}(s, \bm{\xi}^\dagger_s, \bm{u}^\dagger_s) ds \quad \forall t \in [0, 1], \quad \xi^\dagger_0 \in \Xi_0, \quad \xi^\dagger_1 \in \Xi_1. 
\end{align}
and by Eq.~\eqref{eq:robust MTD realization dynamics equation} is satisfied for the realization solution(s),
\begin{align} \label{eq:robust MTD realization dynamics equation}
    \bm{\xi}^\omega_t = \bm{\xi}^\dagger_0 + \int_0^t \bm{f}(s, \bm{\xi}^\omega_s) ds &+ \int_0^{\tau_1(\omega) \wedge t} \bm{g}(s, \bm{\xi}^\omega_s, \bm{u}^\dagger_s) ds + \sum_{i \in \integers_+} \int_{\tau_{2i}(\omega) \wedge t}^{\tau_{2i + 1}(\omega) \wedge t} \bm{g}(s, \bm{\xi}^\omega_s, \bm{u}^\omega_s) ds, \nonumber \\
    & \quad \forall t \in [0, 1], \quad \forall \omega \in \Omega, \quad \xi^\dagger_0 \in \Xi_0, \quad \xi^\omega_1 \in \Xi_1. 
\end{align}
In the equations above, the reference trajectory $\bm{\xi}^\dagger$, generated by applying the reference control solution $\bm{u}^\dagger$, represents the nominal path given by Eq.~\eqref{eq:robust MTD reference dynamics equation}. 
While these terms define the total nominal cost, that cost is intrinsically coupled to a family of realization solutions $\bm{\xi}^\omega$, each induced by a specific missed thrust occurrence.
The realization solutions evolve according to Eq.~\eqref{eq:robust MTD realization dynamics equation}, and satisfy the same boundary conditions as the reference solution. 
The system dynamics for both reference and realization solutions are governed by the functions $\bm{f}$ and $\bm{g}$, where $\bm{f}$ captures the natural dynamics and $\bm{g}$ captures the forcing dynamics. 
The optimal control problem is posed over the normalized time domain $[0, 1]$.
Figure~\ref{fig:problem_setup} illustrates the problem setup schematically, showing the nominal reference trajectory $\xi^\dagger$ and a representative realization trajectory $\xi^\omega$ which experiences an MTE beginning at time $t=\tau_1(\omega)$ and ending at time $t=\tau_2(\omega)$, together with their respective throttle profiles.
\asnotes{Describe the problem setup}

\ifthenelse{\boolean{includefigures}}
{
    \begin{figure}[!htb]
        \centering
        \begin{subfigure}[b]{0.45\textwidth}
            \centering
            \includegraphics[keepaspectratio, width=\textwidth]{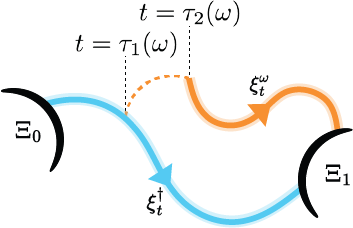}
            \caption{
            Schematic of the restricted robust problem showing \textcolor{SkyBlue}{reference} and \textcolor{Orange}{realization} trajectories
            }
            \label{fig:problem_setup_trajectory}
        \end{subfigure}
        \hfill
        \begin{subfigure}[b]{0.45\textwidth}
            \centering
            \includegraphics[keepaspectratio, width=\textwidth]{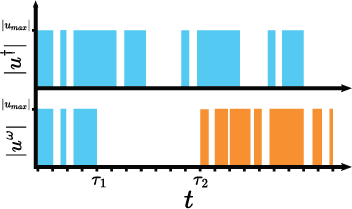}
            \caption{
            Schematic of corresponding throttle profiles with thruster outage during interval $t \in [\tau_1, \tau_2]$
            }
            \label{fig:problem_setup_throttle}
        \end{subfigure}
        \caption{
        Schematic of the restricted robust problem with one allowable MTE initiation location
        }
        \label{fig:problem_setup}
    \end{figure}
}
{
    % Do nothing
}

Uncertainty is introduced through the random variable $\tau : \Omega \rightarrow [0, 1]$, which encodes when an MTE begins and when it ends. 
Formally, $\tau \equiv \{\tau_i(\omega) \in \mathbb{R}_+ \ | \ \tau_i < \tau_{i+1}, \ \forall i \in \mathbb{Z}_+, \ \omega \in \Omega\}$ defines a strictly increasing sequence of event times. 
The sample space is identified with the unit circle, $\Omega \simeq S^1 \simeq [0, 1]$, and $\wedge$ denotes the minimum operator, $a \wedge b \equiv \min(a, b)$. 
For a given realization $\omega \in \Omega$, the corresponding MTE has start time $\tau_{2i+1}(\omega)$ and end time $\tau_{2(i+1)}(\omega)$, and it occurs whenever $\tau_{2i+1}(\omega) < 1$.
Therefore, the duration of the MTE is given by $\Delta\tau_i(\omega) \equiv \tau_{2(i+1)}(\omega) - \tau_{2i+1}(\omega) > 0$.
We define the admissible control space as $\mathcal{U} \equiv PC([0, 1]; \mathbb{R}^n)$, the space of piecewise continuous functions over $[0, 1]$. 
To incorporate uncertainty, we extend this to $\mathcal{U}^\Omega \equiv \mathcal{U}^{S^1} \simeq \prod_{S^1} \mathcal{U}$, representing control functions indexed by $\omega \in \Omega$. 
We assume $\tau_1(0) = \tau_1(1) > 1$, so that $\omega \in \{0,1\}$ corresponds to the reference solution, with control and state denoted $\bm{u}^\dagger$ and $\bm{\xi}^\dagger$. 
For all other $\omega \in (0,1)$, the control is denoted $\bm{u}^\omega$, and the corresponding state as $\bm{\xi}^\omega$.
\asnotes{Describe the uncertainty characterization}

With an eye toward keeping the problem computationally tractable, we impose the following additional simplifying assumptions.
First, for each realization $\omega \in \Omega$, we consider exactly one MTE.
In other words, each realization $\omega \in \Omega$ satisfies $\tau_3(\omega) > 1$.
Second, we restrict the admissible start times to at most one, aligned with the start of a thrust segment. 
Specifically, we partition the space $(0,1) \subset S^1 = \Omega$ into $N$ intervals $\{E_i\}_{i=1}^N$, and require that all realizations within the same interval share the same initiation time, i.e., $\tau_1(\omega_0) = \tau_1(\omega_1)$ for all $\omega_0,\omega_1 \in E_i$.
Finally, we assume that each event disables an integer number of thrust segments, resulting in a quantized event duration. 
Accordingly, we further partition each interval $E_i$ into $M$ sub-intervals $\{E_{i,j}\}_{j=1}^M$, where $E_{i,j}$ corresponds to an event disabling exactly $j$ thrust segments.
Hence, for any $\omega \in E_{i,j}$, the durations $\Delta \tau(\omega) = \tau_2(\omega) - \tau_1(\omega) = j (\tau_S^{\omega} / N^{\omega} )$, where $\tau_S^{\omega}$ and $N^{\omega}$ denotes the shooting time and the number of segments respectively for the realization $\omega$.
\asnotes{Describe the simplifying assumptions}

%% file: sections/methodology.tex
\section{Methodology} \label{sec:methodology}

In \S~\ref{subsection:methodology:missed thrust design a bi-level optimal control formulation}, we first cast the missed thrust design problem into its bi-level form, where the UL selects a nominal solution and, for each realization $\omega \in \Omega$ corresponding to a distinct missed thrust occurrence, the LL computes an optimal recovery solution $(\bm{\xi^\omega},\bm{u^\omega})$ conditioned on the nominal solution $(\bm{\xi^\dagger},\bm{u^\dagger})$. 
In \S~\ref{subsection:methodology:single-level reformulation ll optimality conditions}, we demonstrate that it is possible to collapse the bi-level structure into a single-level structure under certain mathematical conditions by appending the LL first order optimality conditions to the UL constraints, as is done in this work.
Finally, in \S~\ref{subsection:methodology:transcription into a nonlinear program}, we transcribe the infinite dimensional continuous time optimal control problem to a finite-dimensional nonlinear program in a general setting (with more details to follow in \S~\ref{subsec:experimental results relative spacecraft motion:transcription into a nonlinear program} for the specific transcription configuration in our study), which can then solved by an off-the-shelf numerical solver.

\subsection{Missed Thrust Design: A Bi-level Optimal Control Approach} \label{subsection:methodology:missed thrust design a bi-level optimal control formulation}

A general bi-level optimization problem comprises an UL problem and (single/multiple) LL sub-problems. 
Its hallmark characteristic is a hierarchical structure wherein the admissible set for the UL problem is characterized not only by explicit constraints on the UL decision variables but also by implicit constraints induced by the optimal solution set of the LL sub-problem(s):
\begin{equation}
\label{eq:general bi-level optimization problem}
    \begin{aligned}
        \min_{\bm{x}\in \mathscr{X}, \bm{y}\in \mathscr{Y}}\quad & \{\mathscr{F}^{\text{UL}}(\bm{x},\bm{y}) \}\\
        \text{s.t.} \quad & \mathscr{G}^{\text{UL}}_i(\bm{x},\bm{y}) \le 0, \quad i=1, \dots, I^\text{UL}, \\
        & \mathscr{H}^{\text{UL}}_j(\bm{x},\bm{y}) = 0, \quad j=1, \dots, J^\text{UL}, \\
        & \bm{y} \in \argminn_{\bm{\tilde{y}} \in \mathscr{Y}(\bm{x})}\bigl\{\mathscr{F}^{\text{LL}}(\bm{x},\bm{\tilde{y}}) \ \mid \ \mathscr{G}^{\text{LL}}_{i'}(\bm{x},\bm{\tilde{y}}) \le 0, i'=1,\dots,I^\text{LL}; \quad \mathscr{H}^{\text{LL}}_{j'}(\bm{x},\bm{\tilde{y}})=0, j'=1,\dots,J^\text{LL}\bigr\}.
    \end{aligned}
\end{equation}
Here, $\bm{x} \in \mathscr{X} \subseteq \mathbb{R}^{n_x}$ represents the UL decision variables, and $\bm{y} \in \mathscr{Y} \subseteq \mathbb{R}^{n_y}$ represents the LL decision variables. 
The function $\mathscr{F}^{\text{UL}} : \mathbb{R}^{n_x} \times \mathbb{R}^{n_y} \to \mathbb{R}$ describes the UL objective function. 
The inequality constraints $\mathscr{G}^{\text{UL}}_i(\bm{x},\bm{y}) \leq 0$ for $i=1,\dots,I^\text{UL}$ and equality constraints $\mathscr{H}^{\text{UL}}_j(\bm{x},\bm{y})=0$ for $j=1,\dots,J^\text{UL}$ restrict the feasible region for the problem in Eq.~\ref{eq:general bi-level optimization problem}. 
However, unlike a conventional single-level optimization problem, the feasibility of the problem is contingent on the existence of a solution to the LL sub-problem which satisfies its own set of feasibility and optimality conditions.
For each fixed $\bm{x}\in \mathscr{X}$, the LL sub-problem determines the optimal reaction $\bm{y}^\star \in \mathscr{Y}(\bm{x})$ which minimizes $\mathscr{F}^{\mathrm{LL}}$ subject to constraint set comprising the inequality constraints $\mathscr{G}^{\text{LL}}_{i'}(\bm{x},\bm{y}) \leq 0$ for $i'=1,\dots,I^\text{LL}$ and equality constraints $\mathscr{H}^{\text{LL}}_{j'}(\bm{x},\bm{y})=0$ for $j'=1,\dots,J^\text{LL}$.
This nested dependency creates a leader-follower structure where the UL must anticipate the optimal response from the LL, while the feasible set for the LL in turn depends parametrically on the decision variables from the UL.
Consequently, bi-level optimization problems are inherently nonconvex and non-smooth, even when both levels are convex individually.
These properties make bi-level optimization problems significantly more challenging than classical single-level optimization problems, motivating specialized reformulations such as value functions, penalty methods, or single-level reformulation approaches.
\asnotes{General bi-level optimization problem}

The missed thrust design problem naturally admits a bi-level representation in which two coupled optimal control problems are solved in a hierarchical leader-follower structure.
\begin{equation}
\label{eq:mte bi-level optimal control problem}
\begin{aligned}
    \min_{\bm{u^\dagger} \in \mathcal{U}^\dagger \subseteq \mathcal{U}, \bm{u^\omega} \in \mathcal{U}^\omega \subseteq \mathcal{U}} \quad & \{ J^\dagger(\bm{u^\dagger}) \equiv \phi(\bm{\xi^\dagger}_1) + \int_0^1 \mathcal{L}(s, \bm{\xi^\dagger}_s, \bm{u^\dagger}_s) ds \} \\
    \text{s.t.} \quad 
        & \bm{\xi}_t^\dagger = \bm{\xi^\dagger}_0 + \int_0^t \bm{f}(s, \bm{\xi^\dagger}_s) ds + \int_0^t \bm{g}(s, \bm{\xi^\dagger}_s, \bm{u^\dagger}_s) ds, \quad \forall t \in [0, 1], \\
        & \bm{\xi^\dagger}_0 \in \Xi_0, \quad \bm{\xi^\dagger}_1 \in \Xi_1, \\
        & \bm{u^\omega} \equiv \argminn_{\bm{\tilde{u}}^\omega \in \mathcal{U}^\omega} 
        \left\{ 
            \begin{aligned}
                J^\omega(\bm{u^\dagger}, \bm{u^\omega}) & \equiv \phi(\bm{\xi^\dagger}_1, \bm{\xi^\omega}_1) + \int_0^1 \mathcal{L}(s, \bm{\xi^\dagger}_1, \bm{\xi^\omega}_1, \bm{u^\dagger}_s, \bm{u^\omega}_s) ds \\ 
                \text{s.t.} \quad
                & \bm{\xi}_t^\omega = \bm{\xi^\dagger}_0 + \int_0^t \bm{f}(s, \bm{\xi^\omega}_s) ds \\ 
                & \ \quad \qquad + \int_0^{\tau_1(\omega) \wedge t} \bm{g}(s, \bm{\xi^\omega}_s, \bm{u^\dagger}_s) ds \\
                & \ \quad \qquad + \sum_{i \in \mathbb{Z}_+} \int_{\tau_{2i}(\omega) \wedge t}^{\tau_{2i+1}(\omega) \wedge t} \bm{g}(s, \bm{\xi^\omega}_s, \bm{\tilde{u}}_s^\omega) ds, \quad \forall t \in [0, 1]\\
                & \bm{\xi^\omega}_0 \in \Xi_0, \quad \bm{\xi^\omega}_1 \in \Xi_1
            \end{aligned}
        \right\} \forall \omega \in \Omega.
\end{aligned}
\end{equation}
Here, we use the same notation as was done in \S~\ref{sec:problem formulation}.
In this setup, the UL problem, or the \emph{leader} problem, determines the nominal solution $(\bm{\xi}^\dagger,\bm{u}^\dagger)$ under full control authority, optimizing mission level objectives such as propellant consumption, flight time, or a weighted combination thereof.
The LL sub-problem, or the \emph{follower} sub-problem, represents the optimal recovery solution $(\bm{\xi}^\omega,\bm{u}^\omega)$ corresponding to a missed thrust realization characterized by the random variable $\omega\in\Omega$.
Each realization follows the nominal trajectory until the onset of the MTE, propagates under natural dynamics during the missed thrust interval, and subsequently applies an optimal recovery control to re-satisfy the terminal constraints.
Accordingly, the UL variables must be chosen so that, for every missed thrust scenario under consideration, the corresponding LL sub-problem is feasible and satisfies its own optimality conditions, thereby embedding robustness in the design by construction.
Posing the missed thrust design problem as a BOC problem in Eq.~\eqref{eq:mte bi-level optimal control problem} enforces mutual feasibility i.e., any improvement in nominal performance (e.g., reduction in the UL objective function) must not come at the expense of infeasibility under the worst case MTE scenarios.
\asnotes{Writing the missed thrust design problem as a bi-level optimal control problem}

\subsection{Single‑Level Reformulation: LL Optimality Conditions} \label{subsection:methodology:single-level reformulation ll optimality conditions}

The principal difficulty in solving bi-level optimization problems lies in the implicit dependence of the UL solutions on the LL solutions.
Because the feasible set of the UL problem is characterized by the optimal solution to the LL sub-problem, it renders bi-level programs inherently nonconvex and nonsmooth, even when both levels are individually convex and differentiable as was briefly stated before.
In fact, the map from the UL variables to the LL solution is only piecewise smooth, and its derivatives depend on the subset of the LL inequality constraints which are active at the optimal solution.
Whenever this active set changes, i.e., when an inequality constraint switches between active and inactive status, the gradient of the LL solution with respect to the UL variables may experience a discontinuity.
Consequently, the overall mapping is only directionally differentiable, and may fail to be continuously differentiable, thereby complicating sensitivity analysis and numerical optimization.
To address these difficulties, several methodological frameworks have been developed in the literature, ranging from enumerative and penalty approaches to descent and trust-region methods, as well as single-level reformulations, each offering distinct advantages and disadvantages.
\asnotes{Motivating the need for single-level reformulations}

\emph{Enumerative} approaches resolve the hierarchical coupling in bi-level programs by enumerating all possible active set configurations in the LL sub-problem~\cite{papavassilopoulos_algorithms_1982,bialas_two-level_1982,candler_linear_1982,dempe_simple_1987,tuy_global_1993}, and solving the resulting optimization problem for each candidate active set using an iterative scheme, typically a branch-and-bound or branch-and-cut framework~\cite{bard_branch_1990,moore_mixed_1990,liu_solving_1995,shi_extended_2006}.
For small-scale problems, enumerative approaches can converge to the global optimum in theory, however, since the number of possible active set configurations grows combinatorially with the dimensionality of the problem, this is often prohibitive due to exponential computational complexity and runtimes for large-scale problems e.g., continuous time optimal control problems.
For enumerative methods to be computationally tractable, the LL sub-problem must exhibit a finite set of possible active set configurations at optimality, which, in practice, requires the LL sub-problem to be a convex program with a polyhedral feasible region, thereby ensuring a finite number of active set configurations to consider~\cite{vicente_discrete_1996,xu_exact_2014,liu_enhanced_2021}. 
\asnotes{Describing the enumerative approach}

\emph{Penalty} approaches offer an alternative which bypasses the combinatorial burden of enumeration by reformulating the bi-level problem as a sequence of single-level problems.
The key idea is to augment the original objective function with additive penalty terms (with a scalar weight) which quantify violations in the constraints such that these terms go to zero when constraints are satisfied but remain strictly positive otherwise.
By systematically increasing the penalty weights, the solution sequence is driven toward feasibility and, ideally, to satisfaction of the original bi-level structure.
In its original formulation, Aiyoshi and Shimizu proposed replacing the LL sub-problem with a surrogate unconstrained optimization problem, in which the objective function is augmented with additive penalty terms corresponding to residual violations of the constraints~\cite{aiyoshi_hierarchical_1981,aiyoshi_solution_1984}.
However, this approach does not help resolve the hierarchical structure of the bi-level program, as the LL sub-problem must still be solved at each iteration to optimality.
While conceptually simple and provably convergent under certain conditions, this method remains computationally intractable for general nonlinear bi-level programs.
To help mitigate this issue, Ishizuka and Aiyoshi~\cite{ishizuka_double_1992} proposed a \emph{double penalty} approach in which both the UL and LL objective functions are modified by additive penalty terms.
In their formulation, the LL sub-problem is replaced by its stationarity conditions, which are penalized and embedded directly within the UL objective, thereby transforming the bi-level structure into a single-level \emph{Mathematical Program with Equilibrium Constraints} (MPEC) and eliminating the need for repeated LL optimization.
Building on this framework, later methods explicitly encode the full \emph{Karush-Kuhn-Tucker} (KKT) system of the LL sub-problem into the UL-level formulation.
Notably, Bi, Calamai, and Conn~\cite{bi_exact_1989,bi_exact_1991} developed exact penalty formulations that incorporate regularization techniques to promote satisfaction of the LL optimality conditions, particularly the complementarity and stationarity components.
\asnotes{Describing the penalty approach}

An alternative reformulation strategy involves introducing the LL optimal value function as an explicit constraint in the UL problem.
This approach, known as the \emph{value-function} method, defines the LL value function as:
\begin{equation}
    \Phi^{\text{LL}}(\bm{x}) 
    \equiv 
    \min_{\bm{\tilde{y}} \in \mathscr{Y}(\bm{x})}\bigl\{\mathscr{F}^{\text{LL}}(\bm{x},\bm{\tilde{y}}) \ \mid \ \mathscr{G}^{\text{LL}}_{i'}(\bm{x},\bm{\tilde{y}}) \le 0, i'=1,\dots,I^\text{LL}; \quad \mathscr{H}^{\text{LL}}_{j'}(\bm{x},\bm{\tilde{y}})=0, j'=1,\dots,J^\text{LL}\bigr\}
\end{equation}
where $\Phi^{\text{LL}}(\bm{x})$ denotes the minimal LL objective value attainable for a fixed UL decision $\bm{x}$ under the LL constraints.
Using this optimal value function, the LL sub-problem can then be replaced by the constraint 
\begin{equation}
    \mathscr{F}^{\text{LL}} (\bm{x}, \bm{y}) \leq \Phi(\bm{x})
\end{equation}
together with the LL feasibility conditions:
\begin{equation}
    \mathscr{G}^{\text{LL}}_{i'}(\bm{x},\bm{\tilde{y}}) \le 0, i'=1,\dots,I^\text{LL} \quad \text{and} \quad \mathscr{H}^{\text{LL}}_{j'}(\bm{x},\bm{\tilde{y}}) = 0, j'=1,\dots,J^\text{LL}
\end{equation}
ensuring that $\bm{y}$ is an optimal LL response for a given $\bm{x}$~\cite{dempe_2020_bilevel,fischer_semismooth_2022,jolaoso_fresh_2023,outrata_numerical_1990,ye_difference_2022,ye_new_2010}.
While this yields a valid single-level reformulation, the value function $\Phi^{\text{LL}}(\bm{x})$ is not available in closed form in general, making this approach inapplicable for more complex problems.
Evaluating $\Phi^{\text{LL}}(\bm{x})$ typically requires solving the LL sub-problem at each query point $\bm{x}$, thereby embedding a nested optimization problem into every constraint evaluation.
Moreover, the feasible region defined by the inequality $\mathscr{F}^{\text{LL}}(\bm{x},\bm{y}) \le \Phi^{\text{LL}}(\bm{x})$ and the LL constraints $\mathscr{G}^{\text{LL}}_{i'}(\bm{x},\bm{\tilde{y}}) \le 0, i'=1,\dots,I^\text{LL} \quad \text{and} \quad \mathscr{H}^{\text{LL}}_{j'}(\bm{x},\bm{\tilde{y}}) = 0, j'=1,\dots,J^\text{LL}$ is generally nonsmooth, complicating both analysis and solution via standard \emph{nonlinear programming} (NLP) techniques.
Consequently, while conceptually appealing, value-function reformulations are best suited to settings where $\Phi^{\text{LL}}(\bm{x})$ admits tractable approximations or analytical expressions which is rarely plausible in practice.
\asnotes{Describing the value function approach}

\emph{Descent} methods aim to improve the UL objective by computing search directions which reduce the objective value while preserving feasibility with respect to LL optimality.
Originally proposed by Kolstad and Lasdon~\cite{kolstad_derivative_1990} and further developed by Savard and Gauvin~\cite{savard_steepest_1994}, these methods rely on defining the LL solution $\bm{y}^\star(\bm{x})$ as an implicit function of the UL variables $\bm{x}$, which allows one to differentiate the LL optimality conditions with respect to $\bm{x}$.
This yields sensitivity information that can be used to construct a feasible descent direction in the UL space, ensuring first order improvements in the objective while satisfying the LL response at each iteration.
Descent methods offer a principled mechanism for leveraging the coupling between levels, and the descent direction search problem often reduces to a quadratic program derived from IFT.
However, the applicability of these methods hinges on strong regularity conditions: the LL sub-problem must be continuously differentiable, possess a locally unique optimizer, and satisfy constraint qualifications (e.g., LICQ) to guarantee differentiability of the implicit solution map.
In cases where the LL is nonconvex, degenerate, or exhibits multiple optima, the LL solution map $\bm{y}^\star(\bm{x})$ may become nondifferentiable or even discontinuous, invalidating gradient approaches and undermining convergence guarantees.
\asnotes{Describing the descent methods}

\emph{Trust-region} methods decompose the hierarchical structure by iteratively constructing surrogate models which locally approximate the original problem within a neighborhood of the current iterate. 
With this approach, the UL objective is linearized around the current iterate and the LL sub-problem is then replaced by a tractable convex surrogate e.g., a quadratic program or a linear variational inequality over a polyhedral set.
When the model accurately captures local behavior, the trust region expands but otherwise contracts to maintain local validity.
Early applications have considered unconstrained UL problems with strongly convex LL sub-problems with to linear constraints~\cite{liu_trust_1998}. 
The LL sub-problem was replaced by its optimal response map, and each iteration then built a surrogate model using a second order approximation of the LL objective and a first order approximation of the UL objective. 
Marcotte et al. extended this strategy to nonlinear bi-level problems by approximating the LL sub-problem with a strongly monotone linear variational inequality~\cite{marcotte_trust_2001}. 
Colson et al. further generalized the approach by using a quadratic approximation of the LL objective with linearized LL constraints~\cite{colson_trust-region_2005}, and encoding the optimality conditions to reduce the problem to a single-level mixed integer program. 

A particularly powerful reformulation opportunity arises when the LL problem is convex and satisfies the \emph{Mangasarian-Fromovitz Constraint Qualification} (MFCQ), which ensures the existence of Lagrange multipliers for the LL problem~\cite{mangasarian_fritz_1967}. 
Under these conditions, the LL problem can be replaced by its first order KKT conditions, yielding a \emph{Mathematical Program with Complementarity Constraints} (MPCC). 
While this transformation results in an explicit single-level formulation, it introduces complementarity constraints of the form $\lambda_{i'}, \mathscr{G}^{\text{LL}}_{i'}(\bm{x}, \tilde{\bm{y}}) = 0 \ \forall \ i'=1,\dots,I'_g$, where $\lambda_{i'} \ge 0$ are the dual variables and $\mathscr{G}^{\text{LL}}_{i'}(\bm{x}, \tilde{\bm{y}}) \le 0$ are the LL inequality constraints. 
These constraints define a nonconvex feasible region composed of the union of finitely many manifolds, which leads to the violation of MFCQ at all feasible points in the MPCC. 
Specifically, at any feasible point where $\lambda_{i'}=0$ and $\mathscr{G}^{\text{LL}}_{i'}(\bm{x}, \tilde{\bm{y}}) = 0$ simultaneously, the gradients of the active constraints become linearly dependent, thereby failing to span the tangent space required by standard optimization routines.
Classical NLP theory, including convergence guarantees of standard solvers, relies on regularity conditions such as MFCQ. 
To recover meaningful stationarity results, additional conditions become necessary. 
If the LL sub-problem satisfies the \emph{Linear Independence Constraint Qualification} (LICQ), i.e., the gradients of all active LL constraints are linearly independent, then it implies that the LL KKT multipliers are unique.
This uniqueness prevents ambiguity in the embedding of the KKT system and ensures that the MPCC accurately captures the local behavior of the original bi-level problem. 
Furthermore, due to the stricter nature of the LICQ conditions, if LICQ holds, it automatically implies that MFCQ holds as well.
Hence, this reformulation technique remains especially attractive for optimal control applications where LICQ is guaranteed by the structure of the problem. 
\asnotes{Describing the kkt method}

To formalize this connection, consider the bi-level optimization problem in Eq.~\eqref{eq:general bi-level optimization problem}, where, once again, $\bm{x}\in X\subseteq\mathbb{R}^{n_x}$ and $\bm{y}\in Y\subseteq\mathbb{R}^{n_y}$ represent the UL and LL decision variables respectively.  
Recall that the LL sub-problem minimizes $\mathscr{F}^{\mathrm{LL}}(\bm{x},\tilde{\bm{y}})$ over $\tilde{\bm{y}}\in Y(\bm{x})$ subject to inequality constraints $\mathscr{G}^{\mathrm{LL}}_{i'}(\bm{x},\tilde{\bm{y}})\le 0, i'=1,\dots,I^\text{LL}$ and equality constraints $\mathscr{H}^{\mathrm{LL}}_{j'}(\bm{x},\tilde{\bm{y}})=0, j'=1,\dots,J^\text{LL}$.  
Let $\bm{\lambda}\in\mathbb{R}^{I'_g}_{+}$ and $\bm{\mu}\in\mathbb{R}^{J'_h}$ denote the associated LL multipliers, and define the LL Hamiltonian as:
\begin{equation}
\label{eq:general LL hamiltonian:continuous}
    \mathcal{H}_{\mathrm{LL}}(\bm{x},\tilde{\bm{y}},\bm{\lambda},\bm{\mu})
    \equiv 
    \mathscr{F}^{\mathrm{LL}}(\bm{x},\tilde{\bm{y}})
    + \sum_{i'=1}^{I^\text{LL}}\lambda_{i'}\mathscr{G}^{\mathrm{LL}}_{i'}(\bm{x},\tilde{\bm{y}})
    + \sum_{j'=1}^{J^\text{LL}}\mu_{j'}\mathscr{H}^{\mathrm{LL}}_{j'}(\bm{x},\tilde{\bm{y}}).
\end{equation}
for any given UL decision variable $\bm{x}$.
Under standard regularity assumptions, optimality of the LL sub-problem requires that the decision variable set $(\bm{y},\bm{\lambda},\bm{\mu})$ satisfy \emph{stationarity}, \emph{primal} and \emph{dual feasibility}, and \emph{complementarity} conditions.
Substituting these conditions into the UL constraints produces the single-level MPCC formulation below:
\begin{equation} \label{eq:general mpcc}
    \begin{aligned}
        \min_{\bm{x} \in X,\bm{y} \in Y, \bm{\lambda}, \bm{\mu}} 
        & \quad \{\mathscr{F}^{\text{UL}}(\bm{x},\bm{y})\} \\
        \text{s.t.}\quad & \mathscr{G}^{\text{UL}}_i(\bm{x},\bm{y})\le 0,\ i=1,\dots,I_G, \\
        & \mathscr{H}^{\text{UL}}_j(\bm{x},\bm{y})=0,\ j=1,\dots,J_H, \\
        & \nabla_y \mathcal{H}_{\mathrm{LL}}(\bm{x},\bm{y},\bm{\lambda},\bm{\mu}) = \nabla_y \left( \mathscr{F}^{\mathrm{LL}}(\bm{x},\tilde{\bm{y}}) + \sum_{i'=1}^{I^\text{LL}}\lambda_{i'}\mathscr{G}^{\mathrm{LL}}_{i'}(\bm{x},\tilde{\bm{y}}) + \sum_{j'=1}^{J^\text{LL}}\mu_{j'}\mathscr{H}^{\mathrm{LL}}_{j'}(\bm{x},\tilde{\bm{y}}) \right) = \bm{0}, \\
        & \mathscr{G}^{\text{LL}}_{i'}(\bm{x},\bm{y})\le 0,\ i'=1,\dots,I'_g, \\
        & \mathscr{H}^{\text{LL}}_{j'}(\bm{x},\bm{y})=0,\ j'=1,\dots,J'_h, \\
        & \lambda_i \ge 0,\ i'=1,\dots,I'_g, \\
        & \lambda_{i'} \mathscr{G}^{\text{LL}}_{i'}(\bm{x},\bm{y})=0,\ i'=1,\dots,I'_g.
    \end{aligned}
\end{equation}
\asnotes{Describing the single-level reformulation for a general bocp}

The missed thrust design problem admits a direct analog of the above construction.  
To enable this, however, we model the LL sub-problem as a finite horizon \emph{linear quadratic regulator} (LQR) with a strictly convex quadratic objective and linear time-varying dynamics.
Under this structure, the optimal solution and corresponding costates become unique, and LICQ is satisfied automatically due to the linearity of the constraints and the strict convexity of the objective function.
As a result, the continuous-time KKT system, which includes the stationarity, primal and dual feasibility conditions, constitutes both necessary and sufficient conditions for optimality of the LL sub-problem.
Most importantly, because the LL sub-problem is unconstrained, the embedded KKT conditions do not involve complementarity conditions, and the resulting single-level optimal control formulation remains continuously differentiable with unique primal and dual solution pair.
\asnotes{Assumptions we need to make about the structure of the problem so that we can use a single-level reformulation for the mtd problem}

\ifthenelse{\boolean{includefigures}}
{
    \begin{figure}[!htb]
        \centering
        \includegraphics[keepaspectratio, width=0.5\linewidth]{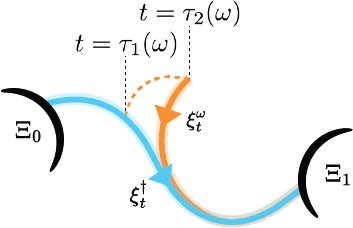}
        \caption{
        Schematic of the restricted robust problem in the bi-level setup showing \textcolor{SkyBlue}{reference} and \textcolor{Orange}{realization} trajectories
        }
        \label{fig:problem_setup_trajectory_regulator}
    \end{figure}
}
{
    % Do nothing
}

To illustrate these properties, we outline the structural features of the missed thrust design problem in its bi-level form below.
The UL problem seeks to determine the reference control input $\bm{u}^\dagger$ which minimizes
\begin{equation}
\label{eq:mte UL cost}
    J^\dagger(\bm{u}^\dagger)
    =\{\phi(\bm{\xi}_1^\dagger)
    +\int_0^1
    \mathcal{L}\left(t,\bm{\xi}_t^\dagger,\bm{u}_t^\dagger\right)dt\},
\end{equation}
subject to 
\begin{equation} \label{eq:mte UL constraint}
    \dot{\bm{\xi}}_t^\dagger
    =\bm{f}(t,\bm{\xi}_t^\dagger)
    +\bm{g}(t,\bm{\xi}_t^\dagger,\bm{u}_t^\dagger), \quad \bm{\xi}_0^\dagger\in\Xi_0, \quad \bm{\xi}_1^\dagger\in\Xi_1.  
\end{equation}
For each realization $\omega\in\Omega$, the LL sub-problem seeks to determine the realization control input $\bm{u}^\omega$ which minimizes
\begin{equation}
\label{eq:mte LL cost}
    J^\omega(\bm{u}^\dagger,\bm{u}^\omega)
    =\{\tilde{\bm{\xi}}_1^T Q_f \tilde{\bm{\xi}}_1
    +
    \int_0^1 
    \left( \tilde{\bm{\xi}}_t^T Q \tilde{\bm{\xi}}_t + \tilde{\bm{u}}_t^T R \tilde{\bm{u}}_t \right) dt\},
\end{equation}
subject to
\begin{equation}
    \dot{\tilde{\bm{\xi}}}_t^\omega
    =
    A_t \tilde{\bm{\xi}} + B_t \tilde{\bm{u}}, \quad \bm{\xi}_0^\omega=\bm{\xi}_0^\dagger, \quad \bm{\xi}_1^\omega\in\Xi_1.
\end{equation} 
where, $A_t \equiv \nabla_{\bm{\xi}} \tilde{\bm{f}}(t,\bm{\xi}_t^\omega, \bm{u}_t^\omega)$, $B_t \equiv \nabla_{\bm{u}} \tilde{\bm{f}}(t,\bm{\xi}_t^\omega, \bm{u}_t^\omega)$, and $\tilde{\bm{f}}(t,\bm{\xi}_t^\omega, \bm{u}_t^\omega) \equiv \bm{f}(t,\bm{\xi}_t^\omega) + \bm{g}(t,\bm{\xi}_t^\omega, \bm{u}_t^\omega)$.
We define the deviation between the reference and the realization states and control as $\tilde{\bm{\xi}}_t \equiv \bm{\xi}_t^\omega - \bm{\xi}_t^\dagger$ and $\tilde{\bm{u}}_t \equiv \bm{u}_t^\omega - \bm{u}_t^\dagger$ respectively.
With the quadratic objective function for the LL sub-problem, the corresponding Hamiltonian can be written as:
\begin{equation}
\label{eq:mte LL hamiltonian:continuous}
    \mathcal{H}^\omega
    (t,\bm{\xi}_t^\omega,\bm{u}_t^\omega,\bm{\lambda}_t^\omega)
    \equiv
    \tilde{\bm{\xi}}_t^T Q \tilde{\bm{\xi}}_t + \tilde{\bm{u}}_t^T R \tilde{\bm{u}}_t
    +\bm{\mu}_t^{\omega\top} \left( A_t \tilde{\bm{\xi}} + B_t \tilde{\bm{u}} - \dot{\tilde{\bm{\xi}}}_t^\omega \right).
\end{equation}
where $Q,Q_f\succeq 0$ and $R\succ 0$ describe the weights describing the relative penalty on the state and control components.  
Here, $\bm{\mu}_t^\omega$ describes the LL costates corresponding to the equality constraints describing the linear dynamics.  
Subsequently, the first order necessary conditions for LL optimality are given by:
\begin{equation}\label{eq:mte kkt:continuous}
\begin{aligned}
    R \tilde{\bm{u}}_t^\omega+B_t^\top\bm{\mu}_t^\omega &= \bm{0},\\
    \dot{\tilde{\bm{\xi}}}_t^\omega &= A_t \tilde{\bm{\xi}} + B_t \tilde{\bm{u}},\\
    -\dot{\bm{\mu}}_t^\omega &= A_t^\top\bm{\mu}_t^\omega+Q \tilde{\bm{\xi}_t},\\
    \bm{\mu}_1^\omega &= Q_f \tilde{\bm{\xi}_1}.
\end{aligned}
\end{equation}
The stationarity condition implies the affine feedback law given by:
\begin{equation}\label{eq:mte feedback_law}
\bm{u}_t^\omega =
\begin{cases}
    \bm{0}, & \forall \ t \in \mathcal{O}^\omega \equiv \bigcup_{i\in\mathbb{Z}_+} [\tau_{2i+1}(\omega), \tau_{2i+2}(\omega)), \\
    -R^{-1} B_t^\top \bm{\mu}_t^\omega, & \forall \ t \in \mathcal{A}^\omega \equiv [0,1] \setminus \mathcal{O}^\omega.
\end{cases}
\end{equation}
where for each $\omega \in \Omega$, the thrust outage periods $\mathcal{O}^\omega$ and availability periods $\mathcal{A}^\omega$ dictate when thrust is suppressed or allowed. 
Substituting the feedback law from Eq.~\eqref{eq:mte feedback_law} into the first order necessary conditions defined by Eq.~\eqref{eq:mte kkt:continuous} produces a two-point boundary value problem in $(\bm{\xi}^\omega,\bm{\mu}^\omega)$ coupled to the nominal trajectory $(\bm{\xi}^\dagger,\bm{u}^\dagger)$.  
The embedded KKT system defines a smooth single-level optimal-control problem that is locally equivalent to the original bi-level formulation, and the absence of complementarity constraints ensures that the resulting single-level formulation remains continuously differentiable, allowing accurate sensitivity analysis and robust convergence of off-the-shelf NLP solvers.

Assembling the necessary necessary elements above, we obtain the following BOC representation of the missed thrust design problem:
\begin{equation} \label{eq:mte mpcc}
    \begin{aligned}
        \min_{\bm{u^\dagger} \in \mathcal{U}^\dagger \subseteq \mathcal{U}, \bm{u^\omega} \in \mathcal{U}^\omega \subseteq \mathcal{U}}
        \quad & \{J^\dagger(\bm{u^\dagger}) \equiv \phi(\bm{\xi^\dagger}_1) + \int_0^1 \mathcal{L}\bigl(t,\bm{\xi}_t^\dagger,\bm{u^\dagger}_t\bigr) dt \} \\
        \text{s.t.}\quad
        & \dot{\bm{\xi}}_t^\dagger = \bm{f}\bigl(t,\bm{\xi}_t^\dagger\bigr) + \bm{g}\bigl(t,\bm{\xi}_t^\dagger,\bm{u^\dagger}_t\bigr); \\
        & \mathcal{C}^\dagger\bigl(\bm{\xi^\dagger}_t,\bm{u^\dagger}_t\bigr)\le \bm{0} \qquad \forall t \in [0,1], \\
        & \bm{\xi^\dagger}_0 \in \Xi_0, \quad \bm{\xi^\dagger}_1 \in \Xi_1, \\
        & \left\{ 
            \begin{aligned}
                \qquad \bm{u^\omega}_t &= \bm{0} \qquad \forall t\in \mathcal{O}^\omega, \\ 
                \qquad R\bm{u^\omega}_t + {B_t}^\top \bm{\mu}_t^\omega &= \bm{0} \qquad \forall t\in \mathcal{A}^\omega, \\ 
                \bm{\xi^\omega}_0 &\in \Xi_0, \\
                \qquad \bm{\xi^\omega}_t &= \bm{\xi^\dagger}_t \qquad \forall t\in[0, \tau_1), \\
                \qquad \dot{\tilde{\bm{\xi}}}_t^\omega &= A_t \tilde{\bm{\xi}} + B_t \tilde{\bm{u}} \qquad \forall  t\in[\tau_1, 1), \\
                \bm{\xi^\omega}_1 &\in \Xi_1, \\
                \qquad -\dot{\bm{\mu}}_t^\omega &= A_t^\top\bm{\mu}_t^\omega+Q \tilde{\bm{\xi}_t} \qquad \forall  t\in\mathcal{A}^\omega, \\
                \bm{\mu}_1^\omega &= Q_f \tilde{\bm{\xi}_1}, \\
                \mathcal{C}^\omega \bigl(\bm{\xi^\omega}_t,\bm{u^\omega}_t\bigr) &\le \bm{0},
            \end{aligned}
        \right\} \forall \omega \in \Omega.
    \end{aligned}
\end{equation}
where $\mathcal{C}^\dagger$ and $\mathcal{C}^\omega$ describe any additional constraints on the reference and realization solutions (see \S~\ref{subsection:methodology:transcription into a nonlinear program}).
\asnotes{Describing the single-level reformulation for the mtd problem}

\subsection{Transcription Into A Nonlinear Program} \label{subsection:methodology:transcription into a nonlinear program}

To numerically solve the bi-level optimal control problem reformulated into a single-level form (see \S~\ref{subsection:methodology:single-level reformulation ll optimality conditions}), we transcribe the infinite-dimensional continuous time formulation into a finite-dimensional NLP.
We adopt a finite burn LT propulsion model, wherein the trajectory is partitioned into multiple burn segments, and within each segment the applied thrust is constant in both magnitude and direction. 
This parametrization provides a discrete representation of the control profile and state evolution over each segment, thereby enabling an efficient direct transcription of the bi-level formulation into a finite-dimensional NLP.
\asnotes{Introducing the transcription method}

Let $\bm{x}^\dagger \in \mathbb{R}^{\overline{N}^\dagger}$ and $\bm{x}^\omega \in \mathbb{R}^{\overline{N}^\omega}$ denote the decision vectors for the nominal and realization trajectories respectively. 
Transcribing the single-level formulation of the missed-thrust design problem in Eq.~\eqref{eq:mte mpcc} yields an NLP of the following general form:
\begin{equation} \label{eq:nonlinear_program}
    \begin{aligned}
        \min_{\bm{x}^\dagger \in \mathbb{R}^{\overline{N}^\dagger}, \bm{x}^\omega \in \mathbb{R}^{\overline{N}^\omega}} 
        & \quad \{ \mathscr{F}^{\text{UL}}(\bm{x},\bm{y}) \} \\
        \text{s.t.} \quad & c^\dagger_e(\bm{x}^\dagger) = 0 \quad \forall \ e \in \mathcal{E}, \\
        & c^\omega_e(\bm{x}^\omega) = 0 \quad \forall \ e \in \mathcal{E}, \\
        & c^\dagger_i(\bm{x}^\dagger) \leq 0 \quad \forall \ i \in \mathcal{I}, \\
        & c^\omega_i(\bm{x}^\omega) \leq 0 \quad \forall \ i \in \mathcal{I}, \\
        & \underline{\bm{x}}^\dagger \le \bm{x}^\dagger \le \overline{\bm{x}}^\dagger, \\
        & \underline{\bm{x}}^\omega \le \bm{x}^\omega \le \overline{\bm{x}}^\omega.
    \end{aligned}
\end{equation}
where $\mathcal{E}$ and $\mathcal{I}$ index the equality and inequality constraint sets respectively, and $(\underline{\bm{x}}^\dagger, \overline{\bm{x}}^\dagger)$ and $(\underline{\bm{x}}^\omega, \overline{\bm{x}}^\omega)$ denote appropriate box constraints imposed on the reference and the realization decision variables. 
\asnotes{General NLP setup}

The reference decision vector $\bm{x}^\dagger$ has the form:
\begin{equation} \label{equation: decision vector (reference)}
    % \bm{x}^\dagger \equiv (T^\dagger, \bm{u}^\dagger_1, \bm{u}^\dagger_2,..., \bm{u}^\dagger_{N^\dagger}, m^\dagger_f),
    \bm{x}^\dagger \equiv (T^\dagger, \bm{u}^\dagger_1, \bm{u}^\dagger_2,..., \bm{u}^\dagger_{N^\dagger}),
\end{equation}
where $T^\dagger$ is the flight time, and $N^\dagger$ is the number of thrust segments each of equal duration $T^\dagger_s / N^\dagger$. 
Each thrust vector $\bm{u}_p^\dagger \in \mathbb{R}^3$ represents the applied thrust during segment $p$ (e.g., a throttle and two directional parameters in a spherical representation).
The realization trajectory $\bm{x}^\omega$ follows a similar parameterization, including the dynamics and constraints consistent with the missed thrust structure:
\begin{equation} \label{equation: decision vector (realization)}
    \bm{x}^\omega \equiv (T^\omega, \bm{u}^\omega_1, \bm{u}^\omega_2,..., \bm{u}^\omega_{N^\omega}).
\end{equation}
The number of segments for the realization $N^\omega$ depend on the location where the MTE occurs, and is given by an \emph{adaptive segmentation strategy} (see Sinha and Beeson \cite{amlans_initial_2025j} for more details).
The resulting NLP is sparse and continuously differentiable, allowing efficient solution using large-scale, gradient-based solvers.
\asnotes{Describing the decision variables}

%% file: sections/theoretical_results.tex
% ===========================================================================
\section{Theoretical Results: Maximum Allowable Missed Thrust Duration}
\label{sec:theoretical bounds on maximum missed thrust duration}

In the single-level reformulation described earlier in \S~\ref{subsection:methodology:single-level reformulation ll optimality conditions}, each LL recovery problem is modeled as a finite-horizon LQR problem, obtained by linearizing the nonlinear dynamics about the nominal UL trajectory and penalizing the deviation between the reference and the realization solutions through a quadratic cost. 
This structural assumption guarantees uniqueness and smoothness of the LL solution and enables the corresponding KKT system to be embedded within the UL constraints. 
However, its validity relies on the deviation between the realization and the reference trajectory remaining within a region where the linearization is a faithful approximation of the true nonlinear dynamics. 
In particular, if the missed thrust interval is too long, the resulting state excursion can invalidate the linear approximation and, consequently, the embedded LL structure.
\asnotes{Motivating the theoretical result}

In this section, we derive an explicit upper bound on the missed thrust duration for which the nonlinear effects remain within an acceptable tolerance with respect to the linear dynamics, thereby providing formal guarantees under which this approach is justifiable. 
In \S\ref{subsec:theoretical bounds on maximum missed thrust duration:mathematical preliminaries}, we first describe the relevant mathematical preliminaries pertinent to the forthcoming theoretical results, and state the regularity assumptions necessary for the corresponding analysis. 
In \S\ref{subsec:theoretical bounds on maximum missed thrust duration:main theorem}, we first present the main theoretical result: an analytical expressions for the maximum allowable missed thrust duration (which depends solely on precomputable quantities along the nominal trajectory) for which the linear dynamics dominate the nonlinear dynamics, and then, in \S\ref{subsec:theoretical bounds on maximum missed thrust duration:supporting lemmas}, we provide three supporting lemmas that establish the key intermediate bounds underpinning the main theorem.
\asnotes{Describing the organization of this section}

% ===========================================================================
\subsection{Mathematical Preliminaries} \label{subsec:theoretical bounds on maximum missed thrust duration:mathematical preliminaries}

% ===========================================================================
\subsubsection{Additional Notation} \label{subsubsec:theoretical bounds on maximum missed thrust duration:mathematical preliminaries:notation}

This subsection consolidates the notation employed in the subsequent analysis. 
While several of the notation listed here were introduced earlier in \S~\ref{sec:problem formulation} and \S~\ref{sec:methodology}, the theoretical results rely on a specific collection of variables and constants, particularly the reference and realization trajectories as well as their corresponding deviation, which are reiterated below. 
These define the error dynamics induced by an MTE and constitute the framework within which the regularity conditions are imposed.
For clarity and to streamline the statement and proof of the main theorem, we collect all relevant notation here in Table~\ref{tab:notation}.

\begin{table}[!htb]
    \centering
    \caption{Symbols appearing in Theorem \ref{thm:maximum_missed_thrust_duration}.}
    \renewcommand{\arraystretch}{1.15}
    \begin{tabular}{@{}p{0.20\linewidth} p{0.74\linewidth}@{}}
        \toprule[1pt]
        Symbol & Meaning / definition \\ \midrule
        $\bm{\xi}$              & State of the nonlinear system (Eq.~\eqref{eq:dynamics}). \\
        $\bm{u}$                & Control of the nonlinear system (Eq.~\eqref{eq:dynamics}). \\
        $\bm{\xi^{\dagger}}$    & Nominal reference state (Eq.~\eqref{eq:reference_dynamics}). \\
        $\bm{u^{\dagger}}$      & Nominal reference control input driving $\bm{\xi^{\dagger}}$ (Eq.~\eqref{eq:reference_dynamics}). \\
        $\bm{\xi^{\omega}}$     & Realization state (Eq.~\eqref{eq:realization_dynamics}). \\
        $\bm{u^{\omega}}$       & Realization control input driving $\bm{\xi^{\omega}}$ (Eq.~\eqref{eq:realization_dynamics}). \\
        $\bm{\tilde{\xi}}$      & Error in state (Eq.~\eqref{eq:error}). \\
        $\tilde{u}$             & Error in control (Eq.~\eqref{eq:error}). \\
        $A_t$                   & Jacobian $\nabla_{\xi}\tilde{f}$ along $(\bm{\xi^{\dagger}}_t, \bm{u^{\dagger}}_t)$ (Eq.~\eqref{eq:jacobians}). \\
        $B_t$                   & Jacobian $\nabla_{u}\tilde{f}$ along $(\bm{\xi^{\dagger}}_t, \bm{u^{\dagger}}_t)$ (Eq.~\eqref{eq:jacobians}). \\
        $\Phi(t,s)$             & State-transition matrix from time $s$ to time $t$. \\
        $\alpha$                & Uniform bound on $\nabla_{\xi}\tilde{f}$ along $(\bm{\xi^{\dagger}}, \bm{u^{\dagger}})$ (Assumption \ref{ass:first_order_system_bounds}). \\
        $\beta$                 & Uniform bound on $\nabla_{u}\tilde{f}$ along $(\bm{\xi^{\dagger}}, \bm{u^{\dagger}})$ (Assumption \ref{ass:first_order_system_bounds}). \\
        $H$                     & Uniform bound on $\nabla^2_{\xi}\tilde{f}$ along $(\bm{\xi^{\dagger}}, \bm{u^{\dagger}})$ (Assumption\ref{ass:second_order_system_bounds}). \\
        $\underline{u}^\dagger/\overline{u}^\dagger$                    & Uniform lower/upper bound on $\|\bm{u^{\dagger}}\|$ (Assumption \ref{ass:control_bounds}). \\
        $f_{\min}^{\bm{u}^\dagger}/f_{\max}^{\bm{u}^\dagger}$          & Effective minimum/maximum forcing amplitude (Assumption \ref{ass:control_bounds}). \\
        % $\varepsilon$           & Prescribed relative-error tolerance (Assumption\ref{ass:relative_error_tolerance}). \\
        $\delta$                & Admissible error radius. \\
        $\delta \tau_{\text{max}}$      & Maximum allowable missed thrust duration. \\
        \bottomrule[1pt]
    \end{tabular}
    \label{tab:notation}
\end{table}
\asnotes{Stating the notation}

Let $\tilde{f}:[t_0, t_f]\times\reals^{n}\times\reals^{m}\to\reals^{n}$ be at least twice continuously differentiable function describing the temporal evolution of a nonlinear dynamical system with control affine dynamics
\begin{equation} \label{eq:dynamics}
    \dot{\xi}_t = \tilde{f} \bigl(\bm{\xi}_t, u_t\bigr) \equiv f \bigl(\bm{\xi}_t\bigr) + u \bigl(\bm{\xi}_t, u_t\bigr) \qquad \forall t \in [t_0, t_f]; \qquad \bm{\xi}_{t_0} = \bm{\Xi}_0 \in \reals^{n},
\end{equation}
where $\bm{\xi}_t$ denotes the state driven by the control $\bm{u}_t$ at time $t$.
A reference pair $(\bm{\xi^{\dagger}}, \bm{u^{\dagger}})$ satisfies
\begin{equation} \label{eq:reference_dynamics}
    \dot{\xi}^{\dagger}_t = \tilde{f} \bigl(\bm{\xi^{\dagger}}_t, \bm{u^{\dagger}}_t\bigr) \qquad \forall t \in [t_0, t_f]; \qquad \bm{\xi^\dagger}_{t_0} = \bm{\Xi}_0\in \reals^{n},
\end{equation}
while the realization pair $(\bm{\xi^{\omega}}, \bm{u^{\omega}})$ satisfies
\begin{equation} \label{eq:realization_dynamics}
    \dot{\xi}^{\omega}_t = \tilde{f}\bigl(\bm{\xi^{\omega}}_t, \bm{u^{\omega}}_t\bigr) \qquad \forall t \in [t_0, t_f]; \qquad \bm{\xi}^\omega_{t_0} = \bm{\Xi}_0\in \reals^{n},
\end{equation}
During an MTE characterized by the random variable $\omega \in \Omega$ beginning at time $\tau_1(\omega)$ and ending at time $\tau_2(\omega) = \tau_1(\omega) + \Delta \tau(\omega)$, where $\Delta \tau > 0$ denotes the missed thrust duration, the control is identically zero i.e.,
\begin{equation}
    \bm{u^{\omega}}_t = 0 \qquad \forall t \in [\tau_1, \tau_1 + \Delta \tau].
\end{equation}
Over this time interval, we define the error state as the deviation between the reference trajectory $\xi^\dagger$ and the realization trajectory $\xi^\omega$ as
\begin{equation} \label{eq:error}
    \bm{\tilde{\xi}}_t \equiv \bm{\xi^{\omega}}_t - \bm{\xi^{\dagger}}_t \qquad \forall t \in [\tau_1, \tau_1 + \Delta \tau]; \qquad \bm{\tilde{\xi}}_{\tau_1} = 0,
\end{equation}
where $\bm{\xi^{\dagger}}$ denotes the solution of Eq.~\eqref{eq:reference_dynamics} and $\bm{\xi^{\omega}}$ denotes the solution of Eq.~\eqref{eq:realization_dynamics}. 
Similarly, we can also define the deviation in the corresponding control as
\begin{equation}
    \bm{\tilde{u}}_t \equiv \bm{u^{\omega}}_t-\bm{u^{\dagger}}_t=-\bm{u^{\dagger}}_t \qquad \forall t \in [\tau_1, \tau_1 + \Delta \tau]
\end{equation}
Linearizing the dynamics along the reference trajectory with respect to the state $\bm{\xi}$ and the control $\bm{u}$ introduces the first order sensitivity information given by:
\begin{equation} \label{eq:jacobians}
    A_t \equiv \nabla_{\xi} \tilde{f} \left(\bm{\xi^{\dagger}}_t, \bm{u^{\dagger}}_t\right) \qquad \text{and} \qquad B_t \equiv \nabla_{u} \tilde{f} \left(\bm{\xi^{\dagger}}_t, \bm{u^{\dagger}}_t\right).
\end{equation}
which defines the temporal evolution of the error dynamics:
\begin{equation}
    \dot{\tilde{\bm{\xi}}}_t^\omega
    =
    A_t \tilde{\bm{\xi}} + B_t \tilde{\bm{u}}, \quad \bm{\xi}_0^\omega=\bm{\xi}_0^\dagger, \quad \bm{\xi}_1^\omega\in\Xi_1.
\end{equation}

% ===========================================================================
\subsubsection{Assumptions} \label{subsubsec:theoretical bounds on maximum missed thrust duration:assumptions}

The analysis that follows requires a set of regularity conditions on the nonlinear dynamics as well as bounds on the reference control solution. 
These assumptions ensure that the linearization along the reference trajectory is well defined, that the associated Taylor remainders admit uniform bounds, and that the forcing induced by the nominal control does not degenerate during the MTE. 
Collectively, these conditions provide the analytical framework within which the error dynamics can be bounded and the maximum allowable missed thrust duration certificate can be derived. 
The assumptions are stated below in a form suitable for direct use in the subsequent lemmas and the main theorem.

\begin{enumerate}[label=\textbf{(A\arabic*)},ref=A\arabic*]
    \item\label{ass:function_is_C2}
    \textbf{Regularity:} 
    $\tilde{f} \in C^{2}$ on a neighborhood $\mathcal T_{\rho} \equiv \{\bm{\xi}_t : \|\bm{\xi}_t-\bm{\xi^{\dagger}}_t\| \le \rho\} \ \forall \ t$.
    \item\label{ass:first_order_system_bounds}
    \textbf{Bounds on first order dynamics: } $\exists$ constants $\alpha,\beta\geq0$ s.t. 
    \begin{equation}
        \|\nabla_{\xi}\tilde{f}(\bm{\xi}^\dagger,\bm{u}^\dagger)\|\le\alpha \qquad \text{and} \qquad \|\nabla_{u}\tilde{f}(\bm{\xi}^\dagger,\bm{u}^\dagger)\|\le\beta \qquad \forall t\in[\tau_1, \tau_1 + \Delta \tau].
    \end{equation}
    \item\label{ass:second_order_system_bounds}
    \textbf{Bounds on second order dynamics:} For some $H>0$,
    \begin{equation}
        \|{\nabla_{ \xi}^{2}\tilde{f}(\bm{\xi}^\dagger,\bm{u}^\dagger)}\| \le H \qquad \forall t\in[\tau_1, \tau_1 + \Delta \tau],
    \end{equation}
    \item\label{ass:control_bounds}
    \textbf{Bounds on control input: } $\exists$ constants $\underline{u}^{\dagger}, \overline{u}^{\dagger}>0$ s.t.  
    \begin{equation}
        0 < \underline{u}^{\dagger}\le\|\bm{u^{\dagger}}_t\|\le\overline{u}^{\dagger} \qquad \forall t\in[\tau_1, \tau_1 + \Delta \tau].
    \end{equation}
    In conjunction with the uniform bounds from Assumption \ref{ass:first_order_system_bounds}, we can therefore define the \emph{minimal} and \emph{maximal} effective forcing amplitude trivially as
    \begin{equation}
        f_{\min}^{\bm{u}^\dagger} \equiv \beta\underline{u}^{\dagger} \qquad \text{and} \qquad f_{\max}^{\bm{u}^\dagger} \equiv \beta\overline{u}^{\dagger},
    \end{equation}
    respectively.
\end{enumerate}
\asnotes{Stating the assumptions}

Assumptions \ref{ass:function_is_C2}-\ref{ass:control_bounds} specify the regularity conditions and control bounds under which the LL linearization is well posed. 
The smoothness and uniform bounds on the first and second order derivatives ensure that the Taylor remainder can be controlled in terms of the state deviation, while the control magnitude bounds guarantee a non-degenerate forcing amplitude along the reference trajectory. 
These conditions collectively provide the analytical foundation required to establish a finite safe error radius and, subsequently, an explicit upper bound on the admissible missed thrust duration in the main theorem. 
For any given reference trajectory $\bm{\xi^{\dagger}}$, the associated constants may be computed a-priori through numerical evaluation of the Jacobians and Hessians along the reference trajectory.
Once obtained, they enter directly into the certification bound in the main theorem presented below.
\asnotes{Describing the assumptions}

% ===========================================================================
\subsection{Main Theorem} \label{subsec:theoretical bounds on maximum missed thrust duration:main theorem}

\ifthenelse{\boolean{includefigures}}
{
    \begin{figure}[!htb]
    \centering
    \includegraphics[keepaspectratio, width=0.7\linewidth]{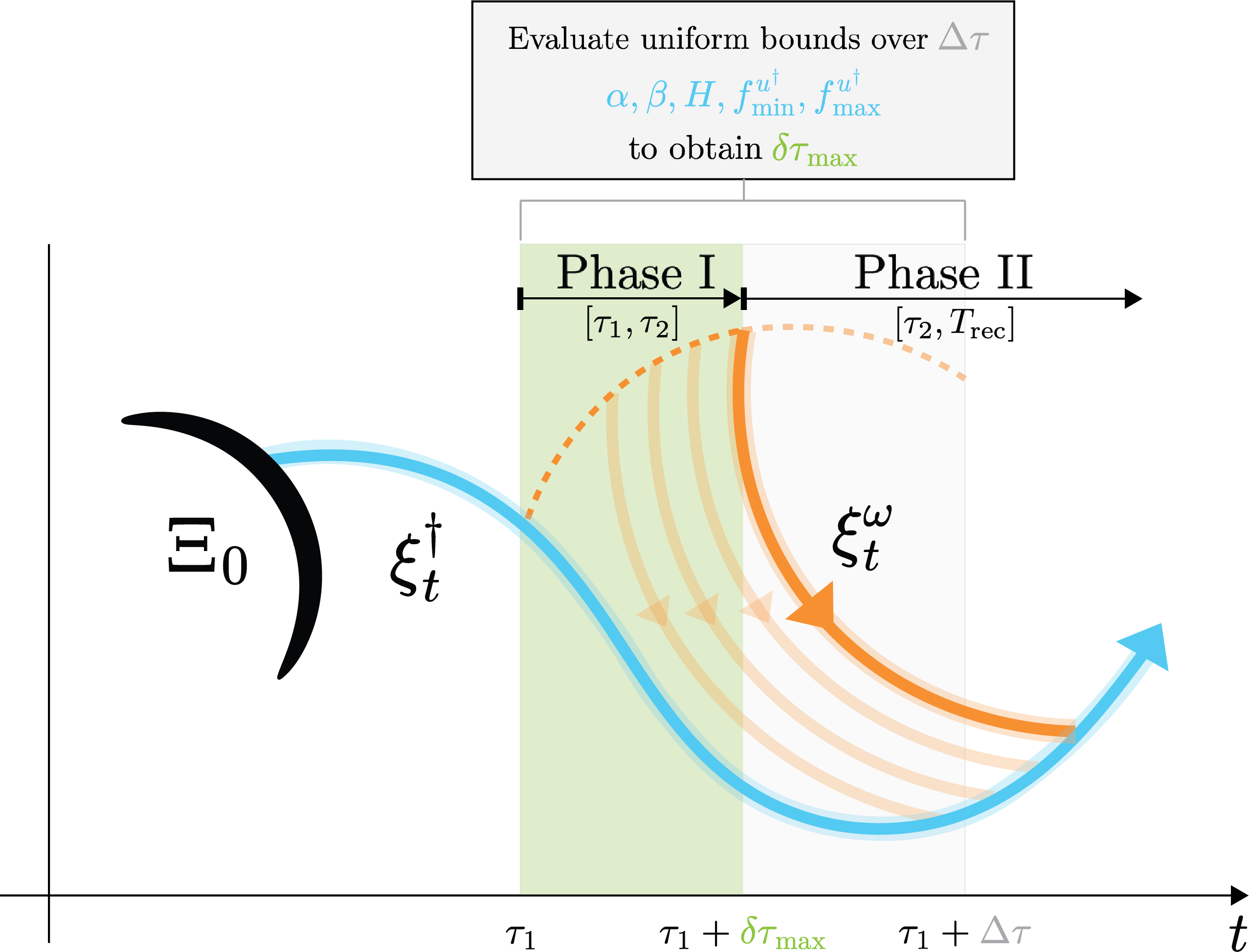}
    \caption{Two phases in the robustness analysis for missed thrust recovery. 
    \textbf{Phase I:} Uniform bounds $(\alpha,\beta,H,f_u^{\min},f_u^{\max})$ are evaluated over the outage interval $[\tau_1,\tau_2]$ to certify a maximum admissible missed thrust duration $\delta\tau_{\max}$ for which the linear model remains appropriate (see \S~\ref{sec:theoretical bounds on maximum missed thrust duration}).
    \textbf{Phase II:} Over the recovery interval $[\tau_2,T_{\mathrm{rec}}]$, recovery feasibility is interpreted using a finite-horizon controllability-energy analysis of the realized deviation dynamics (see \S~\ref{subsec:experimental results relative spacecraft motion:recovery after missed thrust event}).
}
    \label{fig:main_theorem}
    \end{figure}
}
{
    % Do nothing
}

\begin{theorem}[\textbf{Maximum Missed Thrust Duration For Sufficiency Of Linear Approximation Of LL}]
\label{thm:maximum_missed_thrust_duration}
Under assumptions \ref{ass:function_is_C2}-\ref{ass:control_bounds}, given any $\varepsilon \in (0,1)$, define the \emph{safe radius} $\delta=\delta(\varepsilon)$ as the largest $\delta>0$ such that, whenever $\|\bm{\tilde{\xi}}_t\|\leq \delta$, the nonlinear remainder satisfies
$\|r(t,\bm{\tilde{\xi}}_t)\| \leq \varepsilon\|A_t \bm{\tilde{\xi}}_t + B_t \bm{\tilde{u}}_t\|$ for all $t\in[\tau_1,\tau_1+\Delta\tau]$.
Here $r(t,\bm{\tilde{\xi}}_t)$ denotes the second order Taylor remainder of the dynamics
about the reference trajectory $(\bm{\xi}_t^\dagger, \bm{u^\dagger}_t)$, i.e., 
\begin{equation}
    r(t,\bm{\tilde{\xi}}_t)  \equiv  \tilde{f}\big(\bm{\xi}_t^\dagger+\bm{\tilde{\xi}}_t, \bm{u^\dagger}_t+\bm{\tilde{u}}_t\big) - \tilde{f}\big(\bm{\xi}_t^\dagger, \bm{u^\dagger}_t\big) - A_t\bm{\tilde{\xi}}_t - B_t\bm{\tilde{u}}_t,
\end{equation}
with $A_t \equiv \nabla_\xi \tilde{f}(\bm{\xi}_t^\dagger, \bm{u^\dagger}_t)$ and $B_t \equiv \nabla_u \tilde{f}(\bm{\xi}_t^\dagger, \bm{u^\dagger}_t)$ (See Appendix \ref{app:taylor_expansion} for more details).
Then, there exists a unique $\delta\tau_{\max}\in(0,\Delta\tau)$ which defines the maximum missed thrust duration for the sufficiency of the linear approximation of the LL problem.
Writing $\Delta \equiv \alpha^2-2H f^{\dagger}_{u,\max}$, $r_1 \equiv (-\alpha + \sqrt{\Delta})/H$, and $r_2 \equiv (-\alpha - \sqrt{\Delta})/H$, it is possible to derive analytical expressions for $\delta\tau_{\textrm{max}}$ given by:
\begin{enumerate}[label=(\roman*)]
    \item \emph{If $\Delta>0$},
    \begin{equation} \label{eq:Tmax-real}
        \delta\tau_{\max} = \frac{2}{H (r_{1}-r_{2})} \ln \Biggl(\frac{r_{2} \bigl( \delta - r_{1}\bigr)}{r_{1} \bigl( \delta - r_{2}\bigr)}\Biggr).
    \end{equation}
    \item \emph{If $\Delta=0$},
    \begin{equation} \label{eq:Tmax-double}
        \delta\tau_{\max} = \frac{\delta}{\tfrac{\alpha}{2} \delta + \tfrac{\alpha^2}{2H}}.
    \end{equation}
    \item \emph{If $\Delta<0$},
    \begin{equation} \label{eq:Tmax-complex}
        \delta\tau_{\max} = \frac{1}{\gamma} \Bigl(\tan^{-1} \Bigl(\tfrac{\delta + \tfrac{\alpha}{H}}{\gamma}\Bigr) - \phi\Bigr).
    \end{equation}
    where $\gamma \equiv \tfrac{1}{H}\sqrt{2H f_{\max}^{\bm{u}^\dagger} - \alpha^2}$ and $\phi \equiv \tan^{-1}\Bigl(\tfrac{\alpha}{\gamma H}\Bigr)$.
\end{enumerate}
\end{theorem}
\asnotes{Stating the main theorem}

\begin{proof}[Proof]
Let $\varepsilon\in(0,1)$ be fixed.
By Lemma~\ref{lem:lemma_2}, $\delta=\delta(\varepsilon)$ is the largest $\delta>0$ such that, whenever $\|\bm{\tilde{\xi}}_t\|\leq \delta$, the nonlinear remainder satisfies
$\|r(t,\bm{\tilde{\xi}}_t)\| \leq \varepsilon\|A_t \bm{\tilde{\xi}}_t + B_t \bm{\tilde{u}}_t\|$ over the MTE in the (shifted) time interval $t\in[0, \Delta\tau]$.
By Lemma~\ref{lem:lemma_3}, the error norm $\rho_t\equiv\|\tilde{\xi}_t\|$ is bounded from above over this time interval
\begin{equation}
    \rho_t\le \bar\rho_t \ \forall \ t\in[0, \Delta\tau]
\end{equation}
by a scalar envelope $\bar\rho_t$ satisfying
\begin{equation}
    \dot{\bar\rho}_t=\alpha\bar\rho_t+f_{\max}^{\bm{u}^\dagger}+\tfrac{H}{2}\bar\rho_t^2,\qquad \bar\rho_0=0.
\end{equation}
Since $\alpha, f_{\max}^{\bm{u}^\dagger}, H \ge 0$ and $\bar\rho_0=0$, we have $\dot{\bar\rho}_0=f_{\max}^{\bm{u}^\dagger}\ge0$.
The right–hand side remains nonnegative for $\bar\rho\ge 0$, and hence $\bar\rho_t$ is (strictly) increasing over the interval $t\in[0, \Delta\tau]$. 
Therefore, there exists a unique time $\delta\tau_{\max}\in(0,\Delta\tau)$ such that $\bar\rho(\delta\tau_{\max})=\delta$ and for all $t\in[0, \Delta\tau]$ we have $\bar\rho_t\le \delta$ which automatically ensures $\rho_t\le \delta$ as well. Consequently, the relative remainder bound from Lemma~\ref{lem:lemma_2} holds throughout $[0, \Delta\tau]$, which is precisely the desired sufficiency condition for validity of the linear approximation in the LL subproblem.
Deriving explicit expressions for $\delta\tau_{\max}$ is trivial, and can be obtained simply by inverting the closed-form envelope $\bar\rho$ from Lemma~\ref{lem:lemma_3} (i.e., Eqs.~\ref{eq:lemma_3_case_1}, ~\ref{eq:lemma_3_case_2},~\ref{eq:lemma_3_case_3}) at the threshold $\bar\rho(\delta\tau_{\max})=\delta$.
\end{proof}

% ===========================================================================
\subsection{Supporting Lemmas} \label{subsec:theoretical bounds on maximum missed thrust duration:supporting lemmas}

We next present the auxiliary results which are used to prove the main theorem. 
Recall from Eq.~\eqref{eq:error} that the deviation $\bm{\tilde{\xi}}_t$ between the realization and reference trajectories evolves according to
\begin{equation} \label{eq:full_error}
    \dot{\bm{\tilde{\xi}}}_t  =  A_t \bm{\tilde{\xi}}_t + B_t \bm{\tilde{u}}_t + r(t,\bm{\tilde{\xi}}_t) \qquad \forall t \in [\tau_1,\tau_1+\Delta\tau]; \quad \bm{\tilde{\xi}}_{\tau_1}=0,
\end{equation}
where $r(t,\bm{\tilde{\xi}}_t)$ denotes the second order Taylor remainder of the dynamics
about the reference trajectory $(\bm{\xi}_t^\dagger, \bm{u^\dagger}_t)$. 
It is convenient to decompose the exact error trajectory into a linear part and a nonlinear correction,
\begin{equation} \label{eq:full_error_decomposition}
    \bm{\tilde{\xi}}_t = \bm{\tilde{\xi}}^{\mathrm{L}}_t + \bm{\tilde{\xi}}^{\mathrm{NL}}_t.
\end{equation}
The \emph{linear component} $\bm{\tilde{\xi}}^{\mathrm{L}}$ satisfies
\begin{equation} \label{eq:linear_error}
    \dot{\bm{\tilde{\xi}}}^{\mathrm{L}}_t  =  A_t \bm{\tilde{\xi}}^{\mathrm{L}}_t + B_t \bm{\tilde{u}}_t \qquad \forall t \in [\tau_1,\tau_1+\Delta\tau]; \quad \bm{\tilde{\xi}}^{\mathrm{L}}_{\tau_1}=0,
\end{equation}
while the \emph{nonlinear correction} $\bm{\tilde{\xi}}^{\mathrm{NL}}$ evolves according to
\begin{equation} \label{eq:nonlinear_error}
    \dot{\bm{\tilde{\xi}}}^{\mathrm{NL}}_t  =  r(t,\bm{\tilde{\xi}}_t) \qquad \forall t \in [\tau_1,\tau_1+\Delta\tau]; \quad \bm{\tilde{\xi}}^{\mathrm{NL}}_{\tau_1}=0.
\end{equation}
The lemmas that follow provide bounds on each component of this decomposition, as well as on the total error norm, which together underpin the proof of the main result.

% ===========================================================================
\subsubsection{Bounding The Linear Error Term} \label{subsubsec:theoretical bounds on maximum missed thrust duration:supporting lemmas:bounding the linear error term}

As a first step, we isolate the linear component $\bm{\tilde{\xi}}^{\mathrm{L}}$ in the decomposition Eq.~\eqref{eq:full_error_decomposition} and derive an explicit upper bound on its norm over the missed thrust interval $[\tau_1,\tau_1+\Delta\tau]$.

% ---------------------------------------------------------------------------
\begin{lemma}[\textbf{Bounding The Linear Error Term}] \label{lem:lemma_1}
Consider the linear system below governing the temporal evolution of the linear error component $\bm{\tilde{\xi}}^\mathrm{L}$
\begin{equation}
    \dot{\bm{\tilde{\xi}}}^{\mathrm{L}}_t  =  A_t \bm{\tilde{\xi}}^{\mathrm{L}}_t + B_t \bm{\tilde{u}}_t \qquad \forall t \in [\tau_1,\tau_1+\Delta\tau]; \quad \bm{\tilde{\xi}}^{\mathrm{L}}_{\tau_1}=0,
\end{equation}
Under Assumptions \ref{ass:function_is_C2}, \ref{ass:first_order_system_bounds} and \ref{ass:control_bounds},
\begin{equation}
    \|\bm{\tilde{\xi}}^{\mathrm{L}}_t\| \le
    \begin{cases}
        \displaystyle
            \frac{f_{\max}^{\bm{u}^\dagger}}{\alpha}\Bigl(e^{\alpha t}-1\Bigr), & \alpha>0,\\
        \displaystyle
            f_{\max}^{\bm{u}^\dagger}t, & \alpha=0.
    \end{cases}
    \qquad \forall t \in [\tau_1, \tau_1 + \Delta\tau]
\end{equation}
\end{lemma}

% ---------------------------------------------------------------------------
\begin{proof}
Consider the linearized error dynamics given by Eq.~\eqref{eq:linear_error}.
Associated with this system is the \emph{homogeneous system} whose state-transition matrix $\Phi(t,s)$ is defined as the unique solution of
\begin{equation} \label{eq:state_transition_matrix}
    \frac{\partial}{\partial t}\Phi(t,s) = A_t \Phi(t,s) \qquad \forall t, s \in [\tau_1, \tau_1 + \Delta \tau], \qquad \Phi(s,s)=I.
\end{equation}
By the standard variation-of-constants formula,
\begin{equation} \label{eq:variation_of_constants}
    \bm{\tilde{\xi}}^{\mathrm{L}}_t = \int_{\tau_1}^{t}\Phi(t,s)B_s\bm{\tilde{u}}_s \mathrm{d}s \qquad \forall t\in[\tau_1, \tau_1 + \Delta \tau].
\end{equation}
Taking the norm yields
\begin{equation} \label{eq:variation_of_constants_inequality}
    \begin{aligned}
        \| \bm{\tilde{\xi}}^{\mathrm{L}}_t \| &= \| \int_{\tau_1}^{t}\Phi(t,s)B_s\bm{\tilde{u}}_s \mathrm{d}s \| \\
        &\le \int_{\tau_1}^{t} \| \Phi(t,s)B_s\bm{\tilde{u}}_s \| \mathrm{d}s \qquad (\text{Minkowski's inequality}) \\
        &\le \int_{\tau_1}^{t} \| \Phi(t,s) \| \| B_s\bm{\tilde{u}}_s \| \mathrm{d}s \qquad (\text{submultiplicativity}).
    \end{aligned}
\end{equation}

\paragraph{Step 1: Upper bounding $\| \Phi(t,s) \|$}
\leavevmode\\
For $h \ll 1$, we know:
\begin{equation}
    \Phi(t+h,s) = \Phi(t+h,t)\Phi(t,s) = (I+hA_t+o(h))\Phi(t,s).
\end{equation}
We define the scalar function $g_t \equiv \|\Phi(t,s)\| \ \forall t \geq s$ s.t.,
\begin{equation}
    \begin{aligned}
        g_{t+h}-g_t &=\bigl\|\Phi(t+h,s)\bigr\|-\bigl\|\Phi(t,s)\bigr\| \\
        &=\bigl\|\Phi(t+h,t)\Phi(t,s)\bigr\|-\|\Phi(t,s)\| \\
        &\le\|\Phi(t+h,t)\|\|\Phi(t,s)\|-\|\Phi(t,s)\| \qquad (\text{submultiplicativity}) \\
        &=\bigl(\|\Phi(t+h,t)\|-1\bigr)g_t \\
        &=\bigl(\|I + hA_t + o(h)\|-1\bigr)g_t \\
        &\le\bigl(\|I\| + \|hA_t\| + \|o(h)\| - 1 \bigr)g_t \qquad (\text{triangle inequality}) \\
        &\le\bigl(\|hA_t\| + \|o(h)\| \bigr)g_t \qquad (\| I \| = 1 \text{ for any induced operator norm}).
    \end{aligned}
\end{equation}
Divide by $h>0$ and take the limit $h\to0^{+}$:
\begin{equation} \label{eq:gdot_bound}
    \begin{aligned}
        \dot{g}_t &\le \|A_t\|g_t \leq \alpha g_t
    \end{aligned}
\end{equation}
since
\begin{equation}
    \frac{g_{t+h}-g_t}{h}\xrightarrow{h\to0^{+}}\dot{g}_t \qquad \text{and} \qquad \frac{o(h)}{h}\xrightarrow{h\to0^{+}}0.
\end{equation}
By Gr\"{o}nwall lemma, we have
\begin{equation}
    g_t \le e^{\alpha(t-s)}.
\end{equation}
Recalling that $g_t=\|\Phi(t,s)\|$ we have proved:
\begin{equation} \label{eq:state_transition_matrix_norm_bound}
    \|\Phi(t,s)\| \le e^{\alpha(t-s)}, \qquad \tau_1\le s\le t\le \tau_2.
\end{equation}

\paragraph{Step 2: Upper bounding $\| B_s\bm{\tilde{u}}_s \|$}
\leavevmode\\
From Assumption \ref{ass:control_bounds}, we know
\begin{equation} \label{eq:forcing_amplitude_bound}
    \|B_s\bm{\tilde{u}}_s\| \le \beta\bar{u}^\dagger = f_{\max}^{\bm{u}^\dagger}.
\end{equation}

\paragraph{Step 3: Obtaining an explicit bound for $\| \bm{\tilde{\xi}}^{\mathrm{L}}_t \|$}
\leavevmode\\
Inserting the results from Eqs.~\eqref{eq:state_transition_matrix_norm_bound} and~\eqref{eq:forcing_amplitude_bound} in Eq.~\eqref{eq:variation_of_constants_inequality}, we obtain
\begin{equation}
    \|\bm{\tilde{\xi}}^{\mathrm{L}}_t\| \le \int_{\tau_1}^{t}e^{\alpha(t-s)}f_{\max}^{\bm{u}^\dagger}\mathrm{d}s.
\end{equation}
For $\alpha>0$, the integral evaluates to
\begin{equation}
    \int_{\tau_1}^{t}e^{\alpha(t-s)}\mathrm{d}s = e^{\alpha t}\int_{\tau_1}^{t}e^{-\alpha s}\mathrm{d}s = e^{\alpha t}\frac{1-e^{-\alpha t}}{\alpha} = \frac{e^{\alpha t}-1}{\alpha}.
\end{equation}
Hence,
\begin{equation} \label{eq:elin_bound_ageq0}
    \|\bm{\tilde{\xi}}^{\mathrm{L}}_t\| \le \frac{f_{\max}^{\bm{u}^\dagger}}{\alpha}\bigl(e^{\alpha t}-1\bigr), \qquad (\alpha>0).
\end{equation}
For $\alpha=0$, the integral evaluates to
\begin{equation}
    \int_{\tau_1}^{t}e^{\alpha(t-s)} \mathrm{d}s = \int_{\tau_1}^{t}e^{0(t-s)} \mathrm{d}s = \int_{\tau_1}^{t} 1 \mathrm{d}s = t.
\end{equation}
Hence, 
\begin{equation} \label{eq:elin_bound_aeq0}
    \|\bm{\tilde{\xi}}^{\mathrm{L}}_t\| \le f_{\max}^{\bm{u}^\dagger} t, \qquad (\alpha=0).
\end{equation}
which coincides with the limit $\displaystyle\lim_{\alpha\to0^{+}} \tfrac{f_{\max}^{\bm{u}^\dagger}}{\alpha}(e^{\alpha t}-1)=f_{\max}^{\bm{u}^\dagger}t$, so no discontinuity arises.
\end{proof}

% ===========================================================================
\subsubsection{Safe State-Space Radius For A Given Error Tolerance} \label{subsubsec:theoretical bounds on maximum missed thrust duration:safe state-space radius for a given error tolerance}

Building on Lemma~\ref{lem:lemma_1}, which bounds the magnitude of the \emph{linear} deviation induced by an MTE, we now characterize a \emph{safe state-space radius} $\delta(\varepsilon)$ which guarantees the second order Taylor remainder remains uniformly small relative to the linear deviation.

\begin{lemma}[\textbf{Safe State-Space Radius For A Given Error Tolerance}] \label{lem:lemma_2}
Under assumptions \ref{ass:function_is_C2}-\ref{ass:control_bounds}, the second order Taylor remainder obeys the relative-error bound
\begin{equation}
    \|r(t,\bm{\tilde{\xi}}_t)\| \leq \varepsilon\|A_t\bm{\tilde{\xi}}_t + B_t\bm{\tilde{u}}_t\| \qquad \forall t\in[\tau_1,\tau_1+\Delta\tau],
\end{equation}
if the error trajectory satisfies $\|\bm{\tilde{\xi}}_t\| \leq \delta = \delta(\varepsilon)$ for all $t\in[\tau_1,\tau_1+\Delta\tau]$ for any $\varepsilon\in(0,1)$, where $\delta$ is chosen as
\begin{equation}
    \delta  \equiv  
    \begin{cases}
        \min\bigl\{\hat{\delta}, {f_{\min}^{\bm{u}^\dagger}}/{\alpha}\bigr\}, & \alpha>0, \\ 
        \hat{\delta}, & \alpha=0,
    \end{cases}
    \quad \text{where} \quad
    \hat{\delta}  \equiv 
    \begin{cases}
        \displaystyle
            \frac{\sqrt{\varepsilon^{2}\alpha^{2} + 2\varepsilon H f_{\min}^{\bm{u}^\dagger}} - \varepsilon\alpha}{H}, & \alpha>0, \\
        \displaystyle
            \sqrt{\tfrac{2\varepsilon f_{\min}^{\bm{u}^\dagger}}{H}}, & \alpha=0.
    \end{cases}
\end{equation}
\end{lemma}

% ---------------------------------------------------------------------------
\begin{proof}
Let $\varepsilon\in(0,1)$ be fixed.
We seek conditions on $\delta$ which guarantees that if $\|\bm{\tilde{\xi}}_t\| \leq \delta = \delta(\varepsilon)$ for all $t\in[\tau_1,\tau_1+\Delta\tau]$,
the nonlinear remainder obeys the relative error bound
\begin{equation}
    \|r(t,\bm{\tilde{\xi}}_t)\|  \le  \varepsilon \|A_t \bm{\tilde{\xi}}_t+B_t \bm{\tilde{u}}_t\| \quad \forall t\in[\tau_1,\tau_1+\Delta\tau].
\end{equation}

\paragraph{Step 1: Lower bounding $\|A_t \bm{\tilde{\xi}}_t+B_t \bm{\tilde{u}}_t\|$}
\leavevmode\\
By Assumption \ref{ass:control_bounds}, $0<f_{\min}^{\bm{u}^\dagger}\le \|B_t \bm{\tilde{u}}_t\|$.  By Assumption \ref{ass:first_order_system_bounds}, $\|A_t \bm{\tilde{\xi}}_t\|\le \alpha \|\bm{\tilde{\xi}}_t\|\le \alpha \delta$.  Using the reverse triangle inequality (Eq.~\eqref{eq:reverse_triangle_inequality}) yields
\begin{equation}
    \|A_t \bm{\tilde{\xi}}_t+B_t \bm{\tilde{u}}_t\| \ge \bigl|\|B_t \bm{\tilde{u}}_t\|-\|A_t \bm{\tilde{\xi}}_t\|\bigr| \ge f_{\min}^{\bm{u}^\dagger} - \alpha \delta.
\end{equation}
In order for this lower bound to be non‑negative (so we can divide by it), we require $\delta \le \frac{f_{\min}^{\bm{u}^\dagger}}{\alpha}$ if $\alpha > 0$.
If $\alpha=0$, then $\|A_t \bm{\tilde{\xi}}_t\|\equiv 0$ and no additional constraint is needed.

\paragraph{Step 2: Upper bounding $\|r(t,\bm{\tilde{\xi}}_t)\|$}
\leavevmode\\
Assumption \ref{ass:second_order_system_bounds} and the Taylor remainder formula give
\begin{equation}
    \|r(t,\bm{\tilde{\xi}}_t)\|  \le  \frac{H}{2} \|\bm{\tilde{\xi}}_t\|^2 \le \frac{H}{2} \delta^2.
\end{equation}

\paragraph{Step 3: Imposing the relative error bound}
\leavevmode\\
Combining the above estimates, the relative error requirement
    $\|r(t,\bm{\tilde{\xi}}_t)\| \le \varepsilon \|A_t \bm{\tilde{\xi}}_t+B_t \bm{\tilde{u}}_t\|$ will hold if
\begin{equation} \label{eq:delta_inequality}
    \frac{\tfrac{H}{2} \delta^2}{f_{\min}^{\bm{u}^\dagger}-\alpha \delta} \le \varepsilon \qquad \text{and} \qquad f_{\min}^{\bm{u}^\dagger}-\alpha \delta  > 0.
\end{equation}
The positivity of the denominator is guaranteed by our assumption from earlier $\delta \le \frac{f_{\min}^{\bm{u}^\dagger}}{\alpha}$.
Multiplying Eq.~\eqref{eq:delta_inequality} by the positive denominator yields the quadratic inequality
\begin{equation}
    \frac{H}{2} \delta^2 + \varepsilon \alpha \delta - \varepsilon f_{\min}^{\bm{u}^\dagger}  \le  0.
\end{equation}
Solving this inequality gives
\begin{equation}
    \delta \le 
    \begin{cases}
        \displaystyle
            \frac{\sqrt{\varepsilon^2 \alpha^2 + 2 \varepsilon H f_{\min}^{\bm{u}^\dagger}} - \varepsilon \alpha}{H}, & \alpha>0, \\
        \displaystyle
            \sqrt{\tfrac{2 \varepsilon f_{\min}^{\bm{u}^\dagger}}{H}}, & \alpha=0.
    \end{cases}
\end{equation}
We denote the right-hand side by $\hat{\delta}$.  In summary, any $\delta\ge 0$ satisfying $\delta\le\hat{\delta}$ and (for $\alpha>0$) $\delta \le f_{\min}^{\bm{u}^\dagger}/\alpha$ will make Eq.~\eqref{eq:delta_inequality} true, and hence the relative error bound holds.

\paragraph{Step 4: Choosing $\delta$}
\leavevmode\\
Define
\begin{equation}
    \delta  = 
    \begin{cases}
        \min \bigl\{\hat{\delta}, {f_{\min}^{\bm{u}^\dagger}}/{\alpha}\bigr\}, & \alpha>0, \\
        \hat{\delta}, & \alpha=0.
    \end{cases}
\end{equation}
Whenever $\|\bm{\tilde{\xi}}_t\|\le\delta$ on $[\tau_1,\tau_1+\Delta\tau]$, the remainder obeys
$\|r(t,\bm{\tilde{\xi}}_t)\| \le \varepsilon \|A_t \bm{\tilde{\xi}}_t+B_t \bm{\tilde{u}}_t\|$ for all $t$ in that interval.
\end{proof}

% ===========================================================================
\subsubsection{Scalar Envelope For The Total Error Norm} \label{subsubsec:theoretical bounds on maximum missed thrust duration:scalar envelope for the total error norm}

Whereas Lemma~\ref{lem:lemma_2} provides a deviation threshold $\delta$ under which the linearization error is controlled, we next construct a scalar envelope that upper bounds the \emph{total} nonlinear deviation $\|\bm{\tilde{\xi}}_t\|$ over the missed thrust interval, enabling an explicit temporal bound.

% ---------------------------------------------------------------------------
\begin{lemma}[\textbf{Scalar envelope for the total error norm}] \label{lem:lemma_3}
Let $\bm{\tilde{\xi}}_t$ denote the error trajectory under the full nonlinear dynamics, and set $\rho_t  \equiv  \|\bm{\tilde{\xi}}_t\|$.
Under Assumptions \ref{ass:function_is_C2}--\ref{ass:control_bounds}, $\rho_t$ is bounded above by the solution $\bar\rho_t$ of the scalar Riccati equation i.e., 
\begin{equation}
    \rho_t  \le  \bar\rho_t, \qquad \forall t\in[\tau_1,\tau_1+\Delta\tau],
\end{equation}
where 
\begin{equation} \label{eq:riccati}
    \dot{\bar\rho}_t = \alpha\bar\rho_t + f_{\max}^{\bm{u}^\dagger} + \tfrac{H}{2}\bar\rho_t^2, \qquad \bar\rho_{\tau_1}=0.
\end{equation}
The closed-form expression for $\bar\rho_t$ depends on the sign of the discriminant $\Delta  \equiv  \alpha^2 - 2H f_{\max}^{\bm{u}^\dagger}$ associated with the quadratic quation $Hr^2 + \alpha r + f_{\max}^{\bm{u}^\dagger}=0$. 
Writing
\begin{equation}
    r_{1}  \equiv  \tfrac{-\alpha + \sqrt{\Delta}}{H} \qquad \text{and} \qquad r_{2}  \equiv  \tfrac{-\alpha - \sqrt{\Delta}}{H}.
\end{equation}
we can write analytical expressions for $\bar\rho_t$
\begin{enumerate}[label=(\roman*)]
    \item If $\Delta > 0$, then
    \begin{equation} \label{eq:lemma_3_case_1}
        \bar\rho_t = \frac{r_{1}\bigl(1 - e^{\tfrac{H}{2}(r_{1}-r_{2}) t}\bigr)} {1 - \tfrac{r_{1}}{r_{2}} e^{\tfrac{H}{2}(r_{1}-r_{2}) t}}.
    \end{equation}
    \item If $\Delta = 0$, then
    \begin{equation} \label{eq:lemma_3_case_2}
        \bar\rho_t = \frac{\alpha^2 t}{2H\left(1 - \tfrac{\alpha}{2} t\right)}.
    \end{equation}
    \item If $\Delta < 0$, then
    \begin{equation} \label{eq:lemma_3_case_3}
        \bar\rho_t = - \frac{\alpha}{H} + \gamma \tan\bigl(\gamma t + \phi\bigr), \quad \text{where} \quad 
        \gamma \equiv \tfrac{1}{H}\sqrt{2H f_{\max}^{\bm{u}^\dagger} - \alpha^2} \quad \text{and} \quad \phi \equiv \tan^{-1}\Bigl(\tfrac{\alpha}{\gamma H}\Bigr).
    \end{equation}
\end{enumerate}
\end{lemma}

% ---------------------------------------------------------------------------
\begin{proof}
Recall that the exact error dynamics Eq.~\eqref{eq:full_error} can be written
\begin{equation}
    \dot{\bm{\tilde{\xi}}}_t = A_t \bm{\tilde{\xi}}_t+B_t \bm{\tilde{u}}_t+r\bigl(\bm{\tilde{\xi}}_t,t\bigr).
\end{equation}

\paragraph{Step 1: Differentiating the norm $\rho_t$}
\leavevmode\\
Define $\rho_t\equiv\|\bm{\tilde{\xi}}_t\|$.  By the chain rule and the triangle inequality,
\begin{equation}
    \dot\rho_t = \frac{\mathrm d}{\mathrm d t} \|\bm{\tilde{\xi}}_t\| = \frac{\bm{\tilde{\xi}}_t^{\top}}{\|\bm{\tilde{\xi}}_t\|}\dot{\bm{\tilde{\xi}}}_t \le \|A_t\| \rho_t+\|B_t\| \|\bm{\tilde{u}}_t\|+\|r\bigl(\bm{\tilde{\xi}}_t,t\bigr)\|.
\end{equation}

\paragraph{Step 2: Using uniform bounds}
\leavevmode\\
From Assumptions \ref{ass:first_order_system_bounds}, \ref{ass:control_bounds}, and \ref{ass:second_order_system_bounds}, we have
\begin{equation} \label{eq:rho-riccati-ineq}
    \dot\rho_t \le \alpha \rho_t+f_{\max}^{\bm{u}^\dagger}+\frac{H}{2} \rho_t^2 =: g\bigl(\rho_t\bigr), \qquad \rho_{\tau_1}=0.
\end{equation}

\paragraph{Step 3: Constructing an auxiliary equation}
\leavevmode\\
Because $g(r)$ is locally Lipschitz on $r\ge 0$, the initial value problem
\begin{equation} \label{eq:rho-bar-ode}
    \dot{\bar\rho}_t=g\bigl(\bar\rho_t\bigr), \qquad \bar\rho_{\tau_1}=0,
\end{equation}
has a unique solution $\bar\rho:[\tau_1, \tau_1+\Delta \tau]\to\mathbb{R}_{\ge 0}$.
Furthermore, by standard comparison results for scalar ODEs,
any function $\hat\rho_t$ satisfying $\dot{\hat\rho}_t\le g(\hat\rho_t)$ with the same initial condition remains below $\bar\rho_t$ for all $t$.

\paragraph{Step 4: Solving the Riccati equation}
\leavevmode\\
Equation~\eqref{eq:rho-bar-ode} is the Riccati equation
$\dot{\bar\rho} = \alpha \bar\rho + f_{\max}^{\bm{u}^\dagger} + \tfrac{H}{2} \bar\rho^2$.
Its solution with $\bar\rho_0=0$ can be obtained in closed form by factoring the quadratic and integrating with partial fractions.
This yields the three cases summarised in the statement of the lemma:
\begin{enumerate}[label=(\roman*)]
    \item if $\Delta=\alpha^2 - 2 H f_{\max}^{\bm{u}^\dagger}>0$ then $\bar\rho$ is given by Eq.~\eqref{eq:lemma_3_case_1};
    \item if $\Delta=0$ then $\bar\rho$ is given by Eq.~\eqref{eq:lemma_3_case_2};
    \item if $\Delta<0$ then $\bar\rho$ is given by Eq.~\eqref{eq:lemma_3_case_3}.
\end{enumerate}

\paragraph{Step 5: Comparing with the true error}
\leavevmode\\
From Eq.~\eqref{eq:rho-riccati-ineq} and the definition of $\bar{\rho}$ in Eq.~\eqref{eq:rho-bar-ode}, it follows that
$\dot{\rho}_t \le \dot{\bar{\rho}}_t$.
By the scalar comparison principle for ordinary differential equations, we therefore obtain $\rho_t \le \bar{\rho}_t \ \forall \ t \in [\tau_1, \tau_1 + \Delta\tau]$ as claimed.
As illustrated in Fig.~\ref{fig:main_theorem}, once the safe radius $\delta$ is fixed via Lemma~\ref{lem:lemma_2}, we select a $\delta\tau_{\max} > 0$ such that $\bar{\rho}(\delta\tau_{\max}) = \delta$, thereby ensuring the validity of the linear feedback control law over the missed thrust interval.
\end{proof}
\asnotes{Stating the supporting lemmas}

Theorem~\ref{thm:maximum_missed_thrust_duration} provides a \emph{sufficient} condition under which the linear model remains within a prescribed relative error tolerance during the missed thrust interval $[\tau_1,\tau_1+\Delta\tau]$ (Phase I in Fig.~\ref{fig:main_theorem}).
Accordingly, if $\Delta\tau\le\delta\tau_{\max}$, then the deviation is guaranteed to remain inside the safe radius $\delta$, and the linear LQR structure embedded in the UL constraints is analytically justified.
However, the converse does not hold.
Observing $\Delta\tau>\delta\tau_{\max}$ in simulation does not imply that recovery is impossible, nor does it imply loss of controllability after the MTE.
Rather, it only indicates that the particular first-order model validity certificate has been exceeded, so the analysis no longer certifies that the LL linear approximation accurately represents the nonlinear deviation dynamics.
For this reason, in \S~\ref{subsec:experimental results relative spacecraft motion:recovery after missed thrust event}, we introduce a complementary, post-hoc recovery feasibility diagnostic (Phase II in Fig.~\ref{fig:main_theorem}) based on a finite-horizon controllability Gramian computed along the reference trajectory after the outage ends at $\tau_2$.
This controllability energy analysis evaluates whether the available recovery horizon and admissible control authority are, in principle, sufficient to drive the realized deviation back toward the reference, thereby helping interpret why successful recovery may occur even when $\Delta\tau$ exceeds the conservative bound $\delta\tau_{\max}$.

%% file: sections/experimental_results.tex
% ===========================================================================
\section{Experimental Results: Relative Spacecraft Motion} \label{sec:experimental results relative spacecraft motion}

In this section, we present the numerical results by applying the proposed framework to a nonlinear relative spacecraft motion problem in which a chaser spacecraft pursues a target spacecraft in a two-body gravitational field.
This setting provides a representative simple test case for assessing both the numerical performance of the single-level formulation and the tightness of the theoretical robustness certificates.

In \S~\ref{subsec:experimental results relative spacecraft motion:cwh equations with third-order nonlinearities}, we introduce the dynamical model to describe the relative motion of the chaser with respect to the target.
In this work, we study two representative scenarios: a circular target orbit and an eccentric target orbit, which together capture different levels of nonlinearity in the relative dynamics.
In \S~\ref{subsec:experimental results relative spacecraft motion:transcription into a nonlinear program}, we present the direct multiple-shooting transcription, specify the constraints defining the numerical setup, and formulate the resulting NLP.
In \S~\ref{subsec:experimental results relative spacecraft motion:constraint jacobian}, we examine the structure of the constraint Jacobian, verifying its correctness through comparison with a finite difference approximation.
In \S~\ref{subsec:experimental results relative spacecraft motion:hessian bounds}, we numerically evaluate Hessian norms of the nonlinear dynamics in a neighborhood of the reference trajectory, thereby numerically validating the curvature assumptions underlying the theoretical error analysis.
In \S~\ref{subsec:experimental results relative spacecraft motion:example solutions}, we present an illustrative example solution from the dataset which highlights key features of the framework.
In \S~\ref{subsec:experimental results relative spacecraft motion:safe state-space radius delta and maximum missed thrust duration delta tau_max}, we verify the regularity and boundedness assumptions underlying the robustness certificate and compare the safe radius and maximum allowable missed thrust duration with numerical experiments.
Finally, in \S~\ref{subsec:experimental results relative spacecraft motion:recovery after missed thrust event}, we examine recovery following an MTE using a finite-horizon controllability energy analysis to interpret the discrepancies between the theoretical bounds and the actual closed loop performance.

% ===========================================================================
\subsection{Clohessy-Wiltshire-Hill Equations With Third Order Nonlinearities} \label{subsec:experimental results relative spacecraft motion:cwh equations with third-order nonlinearities}

We model the relative motion between the target and the chaser spacecraft using the \emph{Clohessy-Wiltshire-Hill} (CWH) framework augmented with second and third order nonlinear terms, under two different fictitious resident orbits for the target spacecraft with varying levels of nonlinearity.

\ifthenelse{\boolean{includefigures}}
{
\begin{figure}[!htb]
    \centering
    \includegraphics[keepaspectratio, width=0.45\linewidth]{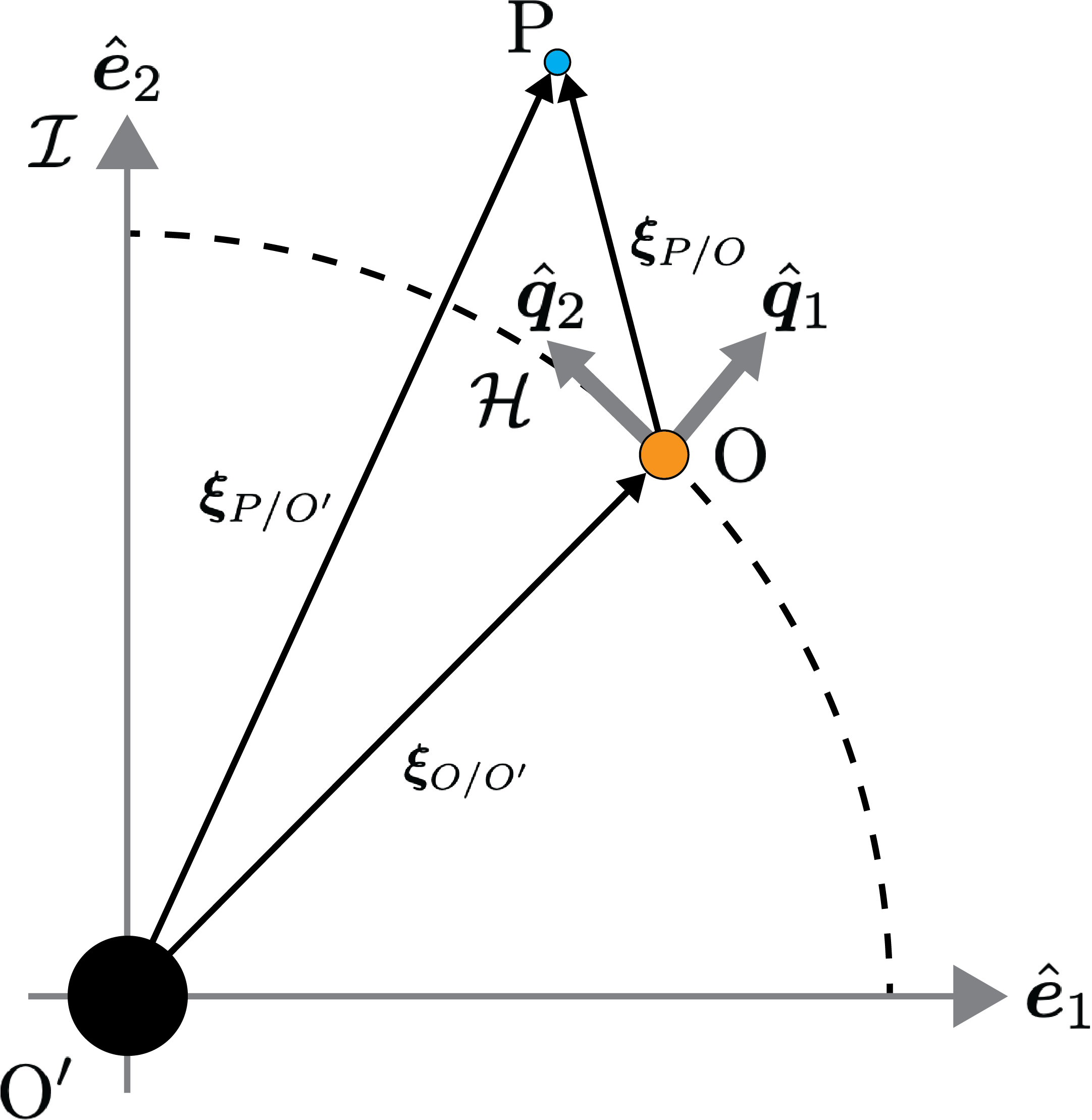}
    \caption{The motion of the chaser spacecraft $\textcolor{DodgerBlue}{P}$ relative to the target spacecraft $\textcolor{Orange}{O}$ can be described in the target centered rotating Hill frame $\mathcal{H}\equiv{\textcolor{Orange}{O}, \hat{\bm{q}}_1, \hat{\bm{q}}_2, \hat{\bm{q}}_3}$, which rotates with angular rate $n$ relative to the inertial frame $\mathcal{I}\equiv{\textcolor{Orange}{O'}, \hat{\bm{q}}_1, \hat{\bm{q}}_2, \hat{\bm{q}}_3}$}.
    \label{fig:cwh_frame}
\end{figure}
}
{
}

In the first case, the target spacecraft is assumed to follow a fictitious circular orbit of radius $|\bm{\xi}_{O'/O}| = R_{O'}$ about the central body $O'$ and possesses no active control authority.
The motion of the chaser spacecraft is described in a rotating reference frame $\mathcal{H}$ (Hill Frame) centered on the target. 
Let $\bm{\xi}_{P/\mathcal{O}}^{\mathcal{H}} \equiv [q_1, q_2, q_3, \dot{q}_1, \dot{q}_2, \dot{q}_3]^\top_{\mathcal{H}} \in \mathbb{R}^6$ denote the chaser's relative state expressed in $\mathcal{H}$, which is centered at the target position $\mathcal{O}$ and rotates with the target, as shown in Fig. \ref{fig:cwh_frame}. 
The orbital dynamics are parameterized by the mean motion $n = \sqrt{\mu/R_{O'}^3}$, where $\mu$ is the gravitational parameter of the central body. 
The numerical values of the orbital parameters used in the subsequent experiments are summarized in Table~\ref{tab:circular_orbit_parameters}. 
Here, $\bm{\hat{\xi}}_{P/\mathcal{O}}^{\mathcal{H}}$ denotes the initial relative state of the chaser spacecraft with respect to the target spacecraft expressed in $\mathcal{H}$, chosen such that the initial relative separation is exactly $1$ km.
\begin{table}[!htb]
\centering
\begin{tabular}{ll}
    \hline
    \textbf{Parameter} & \textbf{Value} \\ \hline
    $R_{O'}$ [km] & $6.871 \times 10^{3}$ \\
    $n$ [s$^{-1}$] & $1.109 \times 10^{-3}$ \\
    $\bm{\hat{\xi}}_{P/\mathcal{O}}^{\mathcal{H}}$ [km, km/s] & $[\frac{\sqrt{2}}{2}, 0, \frac{\sqrt{2}}{2}, 0, 0, 0]^\top$ \\
    \hline
\end{tabular}
\caption{Circular Orbit Parameters}
\label{tab:circular_orbit_parameters}
\end{table}

The translational dynamics of the chaser in the Hill frame are given by:
\begin{equation} \label{eq:cwh_block_form_third}
    \prescript{\mathcal{H}}{}{\frac{d}{dt}}
    \begin{bmatrix}
        \bm{\xi}_{P/\mathcal{O}} \\
        \dot{\bm{\xi}}_{P/\mathcal{O}}
    \end{bmatrix}
    =
    \begin{bmatrix}
        \bm{0}_{3\times3} & \bm{I}_{3\times3} \\
        \bm{\Xi}          & \bm{\Omega}
    \end{bmatrix}
    \begin{bmatrix}
        \bm{\xi}_{P/\mathcal{O}} \\
        \dot{\bm{\xi}}_{P/\mathcal{O}}
    \end{bmatrix}
    +
    \begin{bmatrix}
        \bm{0}_{3\times1}\\
        \bm{u}
    \end{bmatrix}
    +
    \frac{3n^{2}}{2R_{O'}}
    \begin{bmatrix}
        \bm{0}_{3\times1}\\
        \bm{g}^{(2)}(\bm{\xi}_{P/\mathcal{O}})
    \end{bmatrix}
    +
    \frac{n^{2}}{R_{O'}^{2}}
    \begin{bmatrix}
        \bm{0}_{3\times1}\\
        \bm{g}^{(3)}(\bm{\xi}_{P/\mathcal{O}})
    \end{bmatrix},
    \qquad \bm{\xi}_{P/\mathcal{O}}^{\mathcal{H}}(t_0)=\bm{\hat{\xi}}_{P/\mathcal{O}}^{\mathcal{H}},
\end{equation}
where $\bm{u} \equiv [u_{q_1}, u_{q_2}, u_{q_3}]^\top$ represents the external thrust acceleration in $\mathcal{H}$.
The first-order terms defining the linear motion in the Hill frame is defined as:
\begin{equation} \label{eq:cwh_g1}
    \bm{\Xi} \equiv
    \begin{bmatrix}
        3n^2 & 0 & 0 \\
        0 & 0 & 0 \\
        0 & 0 & -n^2
    \end{bmatrix},
    \qquad
    \bm{\Omega} \equiv
    \begin{bmatrix}
        0 & 2n & 0 \\
        -2n & 0 & 0 \\
        0 & 0 & 0
    \end{bmatrix}.
\end{equation}
where $n$ describes the mean motion. 
The second and third order corrections $\bm{g}^{(2)}$ and $\bm{g}^{(3)}$ are:
\begin{equation} \label{eq:cwh_g2_g3}
    \bm{g}^{(2)}(\bm{\xi}_{P/\mathcal{O}})=
    \begin{bmatrix}
        -2{q_1}^{2}+{q_2}^{2}+{q_3}^{2} \\
         2{q_1}{q_2} \\
         2{q_1}{q_3}
    \end{bmatrix},
    \qquad
    \bm{g}^{(3)}(\bm{\xi}_{P/\mathcal{O}})=
    \begin{bmatrix}
        4{q_1}^3-6{q_1}{q_2}^2-6{q_1}{q_3}^2\\
        -6{q_1}^2{q_2}+\frac{3}{2}\left({q_2}^3+{q_2} {q_3}^2\right)\\
        -6{q_1}^2{q_3}+\frac{3}{2}\left({q_2}^2 {q_3}+{q_3}^3\right)
    \end{bmatrix}.
\end{equation}
Equivalently, the equations split by components are:
\begin{equation}
    \begin{aligned}
        \ddot{q}_1 - 2n\dot{q}_2 - 3n^2 {q_1}  &= 
        u_{q_1}
        + \frac{n^2}{R_{O'}}\left[-3{q_1^2}+\frac{3}{2}\left({q_2^2}+{q_3^2}\right)\right]
        + \frac{n^2}{R_{O'}^2}\left[4{q_1^3}-6{q_1}{q_2^2}-6{q_1}{q_3^2}\right], \\
        \ddot{q}_2 + 2n\dot{q}_1  &= 
        u_{q_2}
        + \frac{n^2}{R_{O'}}\left(3{q_1}{q_2}\right)
        + \frac{n^2}{R_{O'}^2}\left[-6{q_1^2}{q_2}+\frac{3}{2}\left({q_2^3}+{q_2} {q_3^2}\right)\right], \\
        \ddot{q}_3 + n^2 {q_3} &= 
        u_{q_3}
        + \frac{n^2}{R_{O'}}\left(3{q_1}{q_3}\right)
        + \frac{n^2}{R_{O'}^2}\left[-6{q_1^2}{q_3}+\frac{3}{2}\left({q_2^2} {q_3}+{q_3^3}\right)\right].
    \end{aligned}
\label{eq:cwh_third_order_dimensional}
\end{equation}

\subsubsection{Extension to Eccentric Orbits}
\label{subsec:experimental results relative spacecraft motion:cwh equations with-third-order-nonlinearities:extension-to-eccentric-orbits}

For the second case where the target spacecraft resides on an eccentric orbit, the classical CWH equations are no longer valid because the local orbital rate and gravity gradient terms vary along the orbit thereby increasing the degree of nonlinearity in the problem.
To obtain a tractable dynamical model, we reformulate the dynamics using the target's true anomaly $\nu$ as the independent variable, following the \emph{Tschauner-Hempel} (TH) framework~\cite{tschauner_1965_rendezvous}.

Let the target follow a Keplerian ellipse with semi-major axis $a$, eccentricity $e$, and true anomaly $\nu$.
The instantaneous orbital radius is given by:
\begin{equation}
    r(\nu) = \frac{a(1-e^2)}{1+e\cos\nu}.
\end{equation}
The rate of change in the true anomaly is given by:
\begin{equation} \label{eq:r_dnu}
    \dot{\nu}(\nu)
    = \sqrt{\frac{\mu}{a^3}} \frac{(1+e\cos\nu)^2}{(1-e^2)^{3/2}},
\end{equation}
and differentiating this expression yields:
\begin{equation} \label{eq:r_ddnu}
    \ddot{\nu}(\nu)
    = -2\sqrt{\frac{\mu}{a^3}}
      \frac{e\sin\nu(1+e\cos\nu)}{(1-e^2)^{3/2}}.
\end{equation}
As before, let $\bm{\xi}_{P/\mathcal{O}}^{\mathcal{H}}(\nu) \equiv [q_1, q_2, q_3, q_1', q_2', q_3']^\top_{\mathcal{H}} \in \mathbb{R}^6$ denote the chaser's relative state expressed in $\mathcal{H}$ but this time as a function of true anomaly where $(\cdot)' \equiv d(\cdot)/d\nu$.
The chain rule relates the derivatives with respect to time to the derivatives with respect to true anomaly via:
\begin{equation}
    \dot{\bm{q}} = \dot{\nu}\bm{q}', 
    \qquad
    \ddot{\bm{q}} = \ddot{\nu}\bm{q}' + \dot{\nu}^2 \bm{q}'',
    \qquad
    \bm{q} \equiv [q_1,q_2,q_3]^\top.
\end{equation}
Using the TH formulation, the translational dynamics of the chaser in $\mathcal{H}$, including second and third order gravitational corrections, can be written in block form as:
\begin{equation} \label{eq:cwh_block_form_third_eccentric_nu}
    \prescript{\mathcal{H}}{}{\frac{d}{d\nu}}
    \begin{bmatrix}
        \bm{\xi}_{P/\mathcal{O}} \\
        \bm{\xi}_{P/\mathcal{O}}'
    \end{bmatrix}
    =
    \begin{bmatrix}
        \bm{0}_{3\times3} & \bm{I}_{3\times3} \\
        \bm{\Xi}_e(\nu)   & \bm{\Omega}_e(\nu)
    \end{bmatrix}
    \begin{bmatrix}
        \bm{\xi}_{P/\mathcal{O}} \\
        \bm{\xi}_{P/\mathcal{O}}'
    \end{bmatrix}
    +
    \begin{bmatrix}
        \bm{0}_{3\times1}\\
        \frac{\bm{u}}{\dot{\nu}(\nu)^{-2}}
    \end{bmatrix}
    +
    \frac{\mu}{\dot{\nu}(\nu)^2 r(\nu)^{4}}
    \begin{bmatrix}
        \bm{0}_{3\times1}\\
        \bm{g}^{(2)}(\bm{\xi}_{P/\mathcal{O}})
    \end{bmatrix}
    +
    \frac{\mu}{\dot{\nu}(\nu)^2 r(\nu)^{5}}
    \begin{bmatrix}
        \bm{0}_{3\times1}\\
        \bm{g}^{(3)}(\bm{\xi}_{P/\mathcal{O}})
    \end{bmatrix},
\end{equation}
where $\bm{u} \equiv [u_{q_1}, u_{q_2}, u_{q_3}]^\top$ is the thrust acceleration in $\mathcal{H}$ as before, and $\bm{g}^{(2)}$, $\bm{g}^{(3)}$ are the same second and third order correction terms defined in Eq.~\eqref{eq:cwh_g2_g3}.
The anomaly dependent linear operators are
\begin{equation} \label{eq:cwh_Xi_Omega_eccentric_nu}
    \bm{\Xi}_e(\nu) \equiv
    \begin{bmatrix}
        -\dfrac{3+e\cos\nu}{1+e\cos\nu} & 2\dfrac{e\sin\nu}{1+e\cos\nu} & 0 \\
        -2\dfrac{e\sin\nu}{1+e\cos\nu} & -\dfrac{e\cos\nu}{1+e\cos\nu} & 0 \\
        0 & 0 & \dfrac{1}{1+e\cos\nu}
    \end{bmatrix},
    \qquad
    \bm{\Omega}_e(\nu) \equiv
    \begin{bmatrix}
        -2\dfrac{e\sin\nu}{1+e\cos\nu} & -2 & 0 \\
        2 & -2\dfrac{e\sin\nu}{1+e\cos\nu} & 0 \\
        0 & 0 & -2\dfrac{e\sin\nu}{1+e\cos\nu}
    \end{bmatrix}.
\end{equation}
Equivalently, the equations split by components are
\begin{equation}
\label{eq:cwh_third_order_anomaly}
    \begin{aligned}
        q_1'' 
        &=
        -2\frac{e\sin\nu}{1+e\cos\nu} q_1'
        - 2 q_2'
        - \frac{3+e\cos\nu}{1+e\cos\nu} q_1
        + 2\frac{e\sin\nu}{1+e\cos\nu} q_2
        \\
        &\quad
        + \frac{1}{\dot{\nu}(\nu)^2}
          \Bigg[
            u_{q_1}
            + \mu r(\nu)^{-4}\left(-3q_1^2 + \tfrac{3}{2}(q_2^2+q_3^2)\right)
            + \mu r(\nu)^{-5}\left(4q_1^3 - 6q_1q_2^2 - 6q_1q_3^2\right)
          \Bigg],
        \\
        q_2'' 
        &=
        2 q_1'
        - 2\frac{e\sin\nu}{1+e\cos\nu} q_2'
        - 2\frac{e\sin\nu}{1+e\cos\nu} q_1
        - \frac{e\cos\nu}{1+e\cos\nu} q_2
        \\
        &\quad
        + \frac{1}{\dot{\nu}(\nu)^2}
          \Bigg[
            u_{q_2}
            + \mu r(\nu)^{-4}\left(3q_1 q_2\right)
            + \mu r(\nu)^{-5}\left(-6q_1^2 q_2 + \tfrac{3}{2}(q_2^3+q_2q_3^2)\right)
          \Bigg],
        \\
        q_3'' 
        &=
        -2\frac{e\sin\nu}{1+e\cos\nu} q_3'
        + \frac{1}{1+e\cos\nu} q_3
        \\
        &\quad
        + \frac{1}{\dot{\nu}(\nu)^2}
          \Bigg[
            u_{q_3}
            + \mu r(\nu)^{-4}\left(3q_1 q_3\right)
            + \mu r(\nu)^{-5}\left(-6q_1^2 q_3 + \tfrac{3}{2}(q_2^2q_3+q_3^3)\right)
          \Bigg].
    \end{aligned}
\end{equation}
In the limit $e \to 0$, one has $r(\nu)\to R_{O'}$ and $\dot{\nu}(\nu)\to n$, and the coefficients in Eq.~\eqref{eq:cwh_third_order_anomaly} reduce to those of the circular case in Eq.~\eqref{eq:cwh_third_order_dimensional} expressed with respect to true anomaly.

For numerical integration and transcription of the optimal control problem, it is convenient to formulate the dynamics in terms of time derivatives of the state variables.
Using the relations
\begin{equation}
    \bm{q}' = \frac{\dot{\bm{q}}}{\dot{\nu}},
    \qquad
    \bm{q}'' = \frac{1}{\dot{\nu}^2}
    \left( \ddot{\bm{q}} - \ddot{\nu}\bm{q}' \right),
\end{equation}
and substituting them in Eq.~\eqref{eq:cwh_third_order_anomaly}, we recover the equations of motion in the time domain as
\begin{equation}
\label{eq:th_time_domain_componentwise_final}
\begin{aligned}
    \ddot{q}_1 
    - 2\dot{\nu}\dot{q}_2 
    - \dot{\nu}^2 q_1 
    - \ddot{\nu}q_2 
    - 2\frac{\mu}{r(\nu)^3} q_1
    &=
    u_{q_1}
    + \mu r(\nu)^{-4}\left[-3q_1^2 + \tfrac{3}{2}(q_2^2+q_3^2)\right]
    + \mu r(\nu)^{-5}\left[4q_1^3 - 6q_1q_2^2 - 6q_1q_3^2\right],
    \\
    \ddot{q}_2 
    + 2\dot{\nu}\dot{q}_1
    + \ddot{\nu}q_1 
    - \dot{\nu}^2 q_2
    + \frac{\mu}{r(\nu)^3} q_2
    &=
    u_{q_2}
    + \mu r(\nu)^{-4}\left[3q_1 q_2\right]
    + \mu r(\nu)^{-5}\left[-6q_1^2 q_2 + \tfrac{3}{2}(q_2^3+q_2q_3^2)\right],
    \\
    \ddot{q}_3 
    + \frac{\mu}{r(\nu)^3} q_3
    &=
    u_{q_3}
    + \mu r(\nu)^{-4}\left[3q_1 q_3\right]
    + \mu r(\nu)^{-5}\left[-6q_1^2 q_3 + \tfrac{3}{2}(q_2^2q_3+q_3^3)\right],
\end{aligned}
\end{equation}
where $\dot{\nu}=\dot{\nu}(\nu)$ and $\ddot{\nu}=\ddot{\nu}(\nu)$ are given by Eqs.~\eqref{eq:r_dnu}–\eqref{eq:r_ddnu}, evaluated along the target spacecraft trajectory.
Equation~\eqref{eq:th_time_domain_componentwise_final} is the eccentric orbit counterpart of the circular orbit model Eq.~\eqref{eq:cwh_third_order_dimensional} and is the form used in the numerical experiments below.

Table~\ref{tab:eccentric_orbit_parameters} summarizes the orbital parameters for the two cases considered in this study for the eccentric orbit. 
Both cases share the same semi-major axis and eccentricity, yielding an identical periapse radius equal to the circular-orbit radius, but differ in the true anomaly interval over which the analysis is performed. 
Specifically, the initial time and the maximum allowable flight time are selected such that the target spacecraft remains within the prescribed true anomaly range.
The reason for choosing these intervals is that the true anomaly range $\nu \in [60^\circ,120^\circ]$ corresponds to motion closer to periapse, where higher orbital velocity and stronger gravitational gradients give rise to stronger nonlinear effects in the relative dynamics.
The interval $\nu \in [120^\circ,180^\circ]$ corresponds to motion closer to apoapse, where orbital velocities are lower and the relative dynamics exhibit weaker nonlinear effects.
\begin{table}[!htb]
\centering
\begin{tabular}{lll}
    \hline
    \textbf{Parameter} & \textbf{Case 1} & \textbf{Case 2} \\ \hline
    $a$ [km] & $22{,}903.33$ & $22{,}903.33$ \\
    $e$ & $0.7$ & $0.7$ \\
    $r_p = a(1-e)$ [km] & $6{,}871.0$ & $6{,}871.0$ \\
    $\nu$ [deg] & $[60^\circ, 120^\circ]$ & $[120^\circ, 180^\circ]$ \\
    $\bm{\hat{\xi}}_{P/\mathcal{O}}^{\mathcal{H}}$ [km, km/s] & $[\frac{\sqrt{2}}{2}, 0, \frac{\sqrt{2}}{2}, 0, 0, 0]^\top$ & $[\frac{\sqrt{2}}{2}, 0, \frac{\sqrt{2}}{2}, 0, 0, 0]^\top$ \\
    \hline
\end{tabular}
\caption{Eccentric Orbit Parameters}
\label{tab:eccentric_orbit_parameters}
\end{table}
Table~\ref{tab:control_parameters} summarizes the weighting matrices defining the LQR problem in the LL, the relative error tolerance $\varepsilon$ (Lemma~\ref{lem:lemma_2}), and the maximum allowable thrust acceleration $T_{\mathrm{max}}^{\mathrm{acc}}$ for the reference and the realization solutions used throughout the numerical study.
\begin{table}[!htb]
\centering
\begin{tabular}{ll}
    \hline
    \textbf{Parameter} & \textbf{Value} \\ \hline
    $Q$ & $I_6$ \\
    $R_u$ & $0.1\, I_6$ \\
    $Q_f$ & $10\, I_6$ \\
    $\varepsilon$ & $0.05$ \\
    $T_{\mathrm{max}}^{\mathrm{acc}}$ [m/s$^2$] & $10^{-3}$ \\
    \hline
\end{tabular}
\caption{Control Parameters}
\label{tab:control_parameters}
\end{table}

% ===========================================================================
\subsection{Transcription Into A Nonlinear Program} \label{subsec:experimental results relative spacecraft motion:transcription into a nonlinear program}

We transcribe the continuous-time bi-level optimal control problem in Eq.~\eqref{eq:mte mpcc} into a finite-dimensional NLP by a direct forward-backward multiple shooting transcription method. 
The reference trajectory is discretized into $N^\dagger\in\mathbb{Z}_+$ segments, each consisting of ${n_s^\dagger}$ integration steps.
The state and control vectors are of dimension $n_x\in\mathbb{Z}_+$ and $n_u\in\mathbb{Z}_+$ respectively. 
The decision variables corresponding to the reference trajectory are the final times $T^\dagger\in\mathbb{R}$, segment boundary states $\bm{X}_{k^\dagger}^\dagger\in\mathbb{R}^{n_x}$ and piecewise-constant controls $\bm{U}_{k^\dagger}^\dagger\in\mathbb{R}^{n_u}$ for each reference segment $k^\dagger \in \mathcal{K}^\dagger$, where $\mathcal{K}^\dagger$ denotes the segment index set for the reference solution.
Similarly, the realization trajectory is discretized into $N^\omega\in\mathbb{Z}_+$ segments, with the decision variables corresponding to the realization trajectory being the final times $T^\omega\in\mathbb{R}$, segment boundary states $\bm{X}_{k^\omega}^\omega\in\mathbb{R}^{n_x}$, piecewise-constant controls $\bm{U}_{k^\omega}^\omega\in\mathbb{R}^{n_u}$, and the discrete costates $\bm{\Lambda}_{k^\omega}^\omega\in\mathbb{R}^{n_x}$ for each realization segment $k^\omega \in \mathcal{K}^\omega$, where $\mathcal{K}^\omega$ denotes the segment index set for the realization solution. 
The indices corresponding to the missed thrust segments are collected in the set $\mathcal{M}^\omega$ (see Eq.~\ref{eq:nlp_m_omega} for the mathematical definition).
Consequently, the total number of realization segments $N^\omega$ depends on the cardinality of the missed thrust set $|\mathcal{M}^\omega|$, and its exact value is determined by an adaptive segmentation strategy (see Sinha and Beeson~\cite{amlans_initial_2025j} for details).
The resulting decision variables, summarized in Table~\ref{tab:transcription parameters}, define a single-level NLP which preserves the problem's sparse structure.

\begin{table}[h]
\centering
\begin{tabular}{ll}
    \hline
    \textbf{Parameter}                                            & \textbf{Value}                \\ \hline
    $N^\dagger$                                                   & 50                            \\
    $n^\dagger_s$                                                 & 100                           \\
    $N^\omega$                                                    & 40                            \\
    $\mathcal{M}^\omega$                                          & \{5, \ldots, 14\}            \\ 
    $n^\omega_s$                                                  & 100                           \\ \hline
\end{tabular}
\caption{Transcription Parameters}
\label{tab:transcription parameters}
\end{table}

The bi-level objective function $J^{\dagger}$, which is the same as the leader's nominal objective, may be expressed under an appropriate numerical discretization scheme as:
\begin{equation} \label{eq:leader_obj}
    J^{\dagger}=w_t^{\dagger}T^{\dagger}+w_u^{\dagger}\sum_{k=0}^{N^\dagger-1}\bigl|\bm{U}_{k^\dagger}^\dagger\bigr|,
\end{equation}
with fixed weights $w_t^{\dagger},w_f^{\dagger}\ge 0$ denoting the relative weights between the flight time and the control effort respectively.
In our numerical experiments, we only consider a minimum fuel case, and therefore set $w_t^{\dagger} = 0$ and $w_f^{\dagger} = 1$.
The follower objective function $J^{\omega}$, which does not directly affect the bi-level objective function but used to derive the first-order optimality conditions for the LL, can be written as:
\begin{equation} \label{eq:follower_obj}
    J^{\omega}=\sum_{k=0}^{N^\omega-1}\Bigl(\tilde{\bm{X}}_k{}^{\top}Q\tilde{\bm{X}}_k+\tilde{\bm{U}}_k{}^{\top} R\tilde{\bm{U}}_k\Bigr)+\tilde{\bm{X}}_{N^\omega}{}^{\top}Q_f\tilde{\bm{X}}_{N^\omega},
\end{equation}
where $\tilde{\bm{X}}_k\equiv \bm{X}^{\omega}_k-\pi_i(\bm{X}^{\dagger}_k)$ and $\tilde{\bm{U}}_k\equiv \bm{U}^{\omega}_k-\pi_i(\bm{U}^{\dagger}_k)$ with $\pi:\mathcal{K}^{\omega}\to\mathcal{K}^\dagger$ being a map for the state and the control from the LL segment $k^{\omega} \in \mathcal{K}^{\omega}$ to the nearest UL segment $k^\dagger \in \mathcal{K}^\dagger$.
As described in Table \ref{tab:control_parameters}, the matrices $Q, Q_f \succeq 0$ are required to be positive semi-definite, while $R \succ 0$ is required to be positive definite.

The following constraints define the feasible region of the transcribed single-level NLP for both the leader and the follower. 
Specifically, we enforce initial state conditions (Eq.~\eqref{eq:nlp_initial_conditions}), continuity constraints (Eq.~\eqref{eq:nlp_continuity}), terminal goal neighborhoods (Eq.~\eqref{eq:nlp_terminal_conditions}), missed thrust constraints on missed thrust intervals $\mathcal{O}^\omega$ with free control on active thrust intervals $\mathcal{A}^\omega$. (Eq.~\eqref{eq:nlp_mte_constraints}), obstacle avoidance via state-space inequalities (Eq.~\eqref{eq:nlp_obstacles_constraints}), simple bounds on states and controls (Eqs.~\eqref{eq:nlp_domain_constraints}-\eqref{eq:nlp_control_bounds}) and the first-order optimality conditions on active thrust intervals in $\mathcal{A}^\omega$ (Eqs.~\eqref{eq:nlp_ll_stationarity}-\eqref{eq:nlp_ll_transversality}).

% ---------------------------------------------------------------------------
\paragraph{Initial State Conditions:}
Given prescribed initial reference state $\bm{x}_0^\dagger \in \mathbb{R}^{n_x}$ and the realization dependent missed thrust initiation time $\tau_1(\omega)$, we impose:
\begin{equation} \label{eq:nlp_initial_conditions}
\begin{aligned}
    \bm{c}_{\mathrm{initial}}^\dagger(\bm{X}_0^\dagger) &\equiv \bm{X}_0^\dagger - \bm{x}_0^\dagger = \bm{0}, \\
    \bm{c}_{\mathrm{initial}}^\omega(\bm{X}_0^\omega) &\equiv \bm{X}_0^\omega - \Phi\bigl(\tau_1(\omega), 0; \bm{x}_0^\dagger, \bm{U}^\dagger\bigr) = \bm{0}.
\end{aligned}
\end{equation}
where $\Phi(t_{k+1}, t_k; \bm{x}_k, \bm{U})$ denotes the state obtained by integrating the state $\bm{x}_k$ forward from time $t_k$ to $t_{k+1}$ under control input $\bm{U}$.

% ---------------------------------------------------------------------------
\paragraph{Continuity Constraints:}
For each reference segment $k^\dagger \in \mathcal{K}^\dagger$ and realization segment $k^\omega \in \mathcal{K}^\omega$, we impose:
\begin{equation} \label{eq:nlp_continuity}
\begin{aligned} 
    \bm{c}_{\mathrm{continuity},k^\dagger}^\dagger(\bm{X}^\dagger_{k^\dagger+1}, \bm{X}^\dagger_{k^\dagger}; \bm{U}^\dagger_{k^\dagger}) &\equiv \bm{X}^\dagger_{k^\dagger+1} - \Phi_{k^\dagger+1, k^\dagger}^{\bm{U}^\dagger_{k^\dagger}} \bm{X}^\dagger_{k^\dagger} = \bm{0}, \\
    \bm{c}_{\mathrm{continuity},k^\omega}^\omega(\bm{X}^\omega_{k^\omega+1}, \bm{X}^\omega_{k^\omega}; \bm{U}^\omega_{k^\omega}) &\equiv \bm{X}^\omega_{k^\omega+1} - \Phi_{k^\omega+1, k^\omega}^{\bm{U}^\omega_{k^\omega}} \bm{X}^\omega_{k^\omega} = \bm{0}.
\end{aligned}
\end{equation}
where the operator $\Phi(t_{k+1}, t_k; \bm{x}_k, \bm{U})$ described in Eq.~\eqref{eq:nlp_initial_conditions} reduces to the numerical flow map $\Phi_{k+1,k}^{\bm{U}}$ when $t_{k+1}$ and $t_k$ align with discretized segment boundaries acting on the state $\bm{x}_k$ under the control input $\bm{U}$.

% ---------------------------------------------------------------------------
\paragraph{Terminal State Conditions:}
Given prescribed terminal states $\bm{x}_{1}\in\mathbb{R}^{n_x}$, we impose:
\begin{equation} \label{eq:nlp_terminal_conditions}
\begin{aligned}
    \bm{c}_{\mathrm{terminal}}^\dagger\bigl(\bm{X}^\dagger_1\bigr) &\equiv \bm{X}^\dagger_{1}-\bm{x}_{1}=\bm{0}, \\
    \bm{c}_{\mathrm{terminal}}^\omega\bigl(\bm{X}^\omega_1\bigr) &\equiv \bm{X}^\omega_{1}-\bm{x}_{1}=\bm{0}.
\end{aligned}
\end{equation}

% ---------------------------------------------------------------------------
\paragraph{Missed Thrust Constraints In $\mathcal{O}^\omega$:}
Recall that $\mathcal{O}^\omega \subset [0,1]$ denotes the time interval during which the MTE occurs for realization $\omega$.  
We define $\mathcal{M}^\omega \subseteq \mathcal{K}^\omega$ as the index set of all discrete segments $k^\omega \in \mathcal{K}^\omega$ whose time domain intersects with $\mathcal{O}^\omega$, i.e.,
\begin{equation} \label{eq:nlp_m_omega}
    \mathcal{M}^\omega \equiv \left\{ k^\omega \in \mathcal{K}^\omega \middle| [t_{k^\omega}, t_{k^\omega+1}] \cap \mathcal{O}^\omega \neq \emptyset \right\},
\end{equation}
On the set $\mathcal{M}^\omega$, the spacecraft has no control authority i.e.,
\begin{equation} \label{eq:nlp_mte_constraints}
    \bm{c}_{\mathrm{mte}}^\omega \bigl(\bm{U}^\omega_{k^\omega}\bigr) \equiv \bm{U}^\omega_{k^\omega} = \bm{0}, \qquad \forall k^\omega \in \mathcal{M}^\omega.
\end{equation}
On the complement set $\bar{\mathcal{M}}^\omega \equiv \mathcal{K}^\omega \setminus \mathcal{M}^\omega$, thrust controls $\bm{U}^\omega_k$ remain decision variables, subject to the control bounds described in Eq.~\eqref{eq:nlp_control_bounds}.

% ---------------------------------------------------------------------------
\paragraph{Obstacle Avoidance:}
For each spherical obstacle $o_j\in\mathcal{O}$ with $(\bm{r}_{o_j/O},R_j)\in\mathbb{R}^{n_p}\times\mathbb{R}_+$ describing the object position $\bm{r}_{o_j/O}$ and the object radius $R_j$, we impose:
\begin{equation} \label{eq:nlp_obstacles_constraints}
\begin{aligned}
    \bm{c}_{\mathrm{obstacle},k^\dagger,j}^\dagger\bigl(\bm{X}^\dagger_{k^\dagger}\bigr) &\equiv R_j - \left\|\mathrm{P}(\bm{X}^\dagger_{k^\dagger})-\bm{r}_{j/O}^o\right\|_2 \le 0 \qquad \forall k^\dagger\in\mathcal{K}^\dagger, \\
    \bm{c}_{\mathrm{obstacle},k^\omega,j}^\omega\bigl(\bm{X}^\omega_{k^\omega}\bigr) &\equiv R_j - \left\|\mathrm{P}(\bm{X}^\omega_k)-\bm{r}_{j/O}^o\right\|_2 \le 0 \qquad \forall k^\omega\in\mathcal{K}^\omega.
\end{aligned}
\end{equation}
where $\mathrm{P}(\bm{X})$ simply extracts the position components from the state $\bm{X}$.

% ---------------------------------------------------------------------------
\paragraph{Domain Bounds:}
Let $\underline{\bm{x}}_k,\overline{\bm{x}}_k\in\mathbb{R}^{n_x}$, where $\underline{\bm{x}}_k$ and $\overline{\bm{x}}_k$ denotes the lower and the upper bounds on the state components in $\bm{x}_k$ respectively.
We impose:
\begin{equation} \label{eq:nlp_domain_constraints}
    \begin{aligned}
        \bm{c}_{\mathrm{domain}+}^\dagger\bigl(\bm{X}^\dagger_{k^\dagger}\bigr) &\equiv \underline{\bm{x}}_{k^\dagger} - \bm{X}^\dagger_{k^\dagger} \le \bm{0} \quad \forall k^\dagger\in\mathcal{K}^\dagger, \\
        \bm{c}_{\mathrm{domain}-}^\dagger\bigl(\bm{X}^\dagger_{k^\dagger}\bigr) &\equiv \bm{X}^\dagger_{k^\dagger}-\overline{\bm{x}}_{k^\dagger}^\dagger \le \bm{0} \quad \forall k^\dagger\in\mathcal{K}^\dagger, \\
        \bm{c}_{\mathrm{domain}+}^\omega\bigl(\bm{X}^\omega_{k^\omega}\bigr) &\equiv \underline{\bm{x}}_{k^\omega} - \bm{X}^\omega_{k^\omega} \le \bm{0} \quad \forall k^\omega\in\mathcal{K}^\omega, \\
        \bm{c}_{\mathrm{domain}-}^\omega\bigl(\bm{X}^\omega_{k^\omega}\bigr) &\equiv \bm{X}^\omega_{k^\omega}-\overline{\bm{x}}_{k^\omega}^\omega \le \bm{0} \quad \forall k^\omega\in\mathcal{K}^\omega.
    \end{aligned}
\end{equation}

% ---------------------------------------------------------------------------
\paragraph{Control Bounds:}
Let $\underline{\bm{u}}_k,\overline{\bm{u}}_k\in\mathbb{R}^{n_u}$, where $\underline{\bm{u}}_k$ and $\overline{\bm{u}}_k$ denotes the lower and the upper bounds on the control components in $\bm{u}_k$ respectively.
We impose:
\begin{equation} \label{eq:nlp_control_bounds}
    \begin{aligned}
        \bm{c}_{\mathrm{control}+}^\dagger\bigl(\bm{u}^\dagger_{k^\dagger}\bigr) &\equiv \underline{\bm{u}}_{k^\dagger} - \bm{U}^\dagger_{k^\dagger} \le \bm{0} \quad \forall k^\dagger\in\mathcal{K}^\dagger, \\
        \bm{c}_{\mathrm{control}-}^\dagger\bigl(\bm{U}^\dagger_{k^\dagger}\bigr) &\equiv \bm{U}^\dagger_{k^\dagger}-\overline{\bm{u}}_{k^\dagger}^\dagger \le \bm{0} \quad \forall k^\dagger\in\mathcal{K}^\dagger, \\
        \bm{c}_{\mathrm{control}+}^\omega\bigl(\bm{U}^\omega_{k^\omega}\bigr) &\equiv \underline{\bm{u}}_{k^\omega} - \bm{U}^\omega_{k^\omega} \le \bm{0} \quad \forall k^\omega\in\mathcal{K}^\omega, \\
        \bm{c}_{\mathrm{control}-}^\omega\bigl(\bm{U}^\omega_{k^\omega}\bigr) &\equiv \bm{U}^\omega_{k^\omega}-\overline{\bm{u}}_{k^\omega}^\omega \le \bm{0} \quad \forall k^\omega\in\mathcal{K}^\omega.
    \end{aligned}
\end{equation}

% ---------------------------------------------------------------------------
\paragraph{Optimality Conditions In $\mathcal{A}^\omega$:}
The first-order optimality conditions associated with the LL sub-problem apply only on the active thrust segments, i.e., the index set $\bar{\mathcal{M}}^\omega \equiv \mathcal{K}^\omega \setminus \mathcal{M}^\omega$.  
For each $k^\omega \in \bar{\mathcal{M}}^\omega$, the stationarity, costate, and terminal conditions are given by:
\begin{subequations} \label{eq:disc_LQR_KKT}
    \begin{align}
        R \bm{U}^\omega_k + (B_k^\omega)^\top \bm{\Lambda}^\omega_k &= \bm{0}, 
        &&\text{(stationarity)}, \label{eq:nlp_ll_stationarity} \\
        -\bm{\Lambda}^\omega_{k+1} + \Psi_k \bm{\Lambda}^\omega_k + \Gamma_k \tilde{\bm{X}}^\omega_k &= \bm{0}, 
        &&\text{(costate recursion)}, \label{eq:nlp_ll_recursion} \\
        \bm{\Lambda}^\omega_{N^\omega} - 2 Q_f \tilde{\bm{X}}^\omega_{N^\omega} &= \bm{0},
        &&\text{(transversality)}. \label{eq:nlp_ll_transversality}
    \end{align}
\end{subequations}
where $\Psi_k$ and $\Gamma_k$ denote the discrete adjoint transition and forcing matrices obtained by integrating the continuous-time costate dynamics over segment $k$ using the same numerical scheme as for the state dynamics.
For example, if one is using a forward Euler scheme, Eq.~\eqref{eq:nlp_ll_recursion} would simply reduce to:
\begin{equation} \label{eq:nlp_ll_recursion_reduced}
    -\bm{\Lambda}^\omega_{k+1} + \bigl(I - \Delta t A_k^\top\bigr)\bm{\Lambda}^\omega_k - \Delta t Q = \bm{0}.
\end{equation}
where $\Delta t$ represents the time discretization for the numerical integration scheme.
On the missed thrust segments, i.e., for all $k^\omega \in \mathcal{M}^\omega$, the constraint given by Eq.~\eqref{eq:nlp_mte_constraints} enforces $\bm{U}^\omega_{k^\omega} = \bm{0} \ \forall \ k^\omega \in \mathcal{M}^\omega$, and the stationarity condition~\eqref{eq:nlp_ll_stationarity} is therefore omitted.

Collecting all individual constraints, the single-level NLP can be simply written as: 
\begin{equation} \label{eq:single_level_NLP} 
    \begin{aligned} 
        \min_{t^i_f, \bm{X}^i_k,\bm{U}^i_k, \bm{\Lambda}^i_k \forall i>0} \quad & J^{\dagger} \\ 
        \text{s.t.} \quad 
        & \text{Eq.~\eqref{eq:nlp_initial_conditions}}, \\ 
        & \text{Eq.~\eqref{eq:nlp_continuity}}, \\ 
        & \text{Eq.~\eqref{eq:nlp_mte_constraints}}, \\ 
        & \text{Eq.~\eqref{eq:nlp_terminal_conditions}}, \\ 
        & \text{Eq.~\eqref{eq:nlp_obstacles_constraints}}, \\ 
        & \text{Eq.~\eqref{eq:nlp_domain_constraints}}, \\ 
        & \text{Eq.~\eqref{eq:nlp_control_bounds}}, \\ 
        & \text{Eq.~\eqref{eq:nlp_ll_stationarity},Eq.~\eqref{eq:nlp_ll_recursion}, Eq.~\eqref{eq:nlp_ll_transversality}}. 
    \end{aligned} 
\end{equation}
All derivatives are computed using automatic differentiation (AD) with \texttt{JAX}~\cite{jax_2018_github}. 
The breadth of the imposed constraints illustrates the generality of the proposed numerical framework and its flexibility in incorporating additional constraints when required.

% ===========================================================================
\subsection{Constraint Jacobian} \label{subsec:experimental results relative spacecraft motion:constraint jacobian}

The nonlinear program in Eq.~\eqref{eq:single_level_NLP} consists of a collection of equality and inequality constraints $c_j(\bm{z}) = 0$ and $c_j(\bm{z}) \le 0$ defined over the global decision vector:
\begin{equation}
    \bm{z} \equiv \bigl(T^\dagger, \bm{X}^\dagger_k, \bm{U}^\dagger_k, T^\omega, \bm{X}^\omega_k, \bm{U}^\omega_k, \bm{\Lambda}^\omega_k\bigr),
\end{equation}
where the index $k$ ranges over the reference and realization segment sets $\mathcal{K}^\dagger$ and $\mathcal{K}^\omega$.  
The Jacobian of the constraints $\bm{J}$ with respect to the decision variables $\bm{z}$ is the matrix:
\begin{equation}
    \bm{J}(\bm{z}) \equiv \nabla_{\bm{z}} \bm{c}(\bm{z}) =
    \begin{bmatrix}
        \nabla_{\bm{z}} c_1(\bm{z})^\top \\
        \vdots \\
        \nabla_{\bm{z}} c_m(\bm{z})^\top
    \end{bmatrix},
\end{equation}
which encodes the first-order sensitivities of every constraint with respect to every decision variable.
This matrix plays a central role in large constrained optimization e.g., Newton-type SQP and interior-point methods require repeated evaluation of $\bm{J}(\bm{z})$ and its associated linear systems, making its structure a key determinant of computational efficiency.
Because the transcription is constructed using forward-backward multiple shooting, each constraint depends only on variables from a small neighboring set of segments.  
Consequently, $\bm{J}(\bm{z})$ is block sparse and exhibits a banded pattern whose width is dictated by the integrator and by the coupling between reference and realization trajectories.
Each constraint class contributes a distinct structured block, and the resulting sparsity pattern is shown in Fig.~\ref{fig:jacobian_sparsity}.
This structure is exploited directly by sparse linear solvers inside the NLP algorithm, enabling memory efficient and scalable solution of the single-level formulation.
\ifthenelse{\boolean{includefigures}}
{
    \begin{figure}[!htb]
        \centering
        \begin{tikzonimage}[keepaspectratio, width=0.5\linewidth]{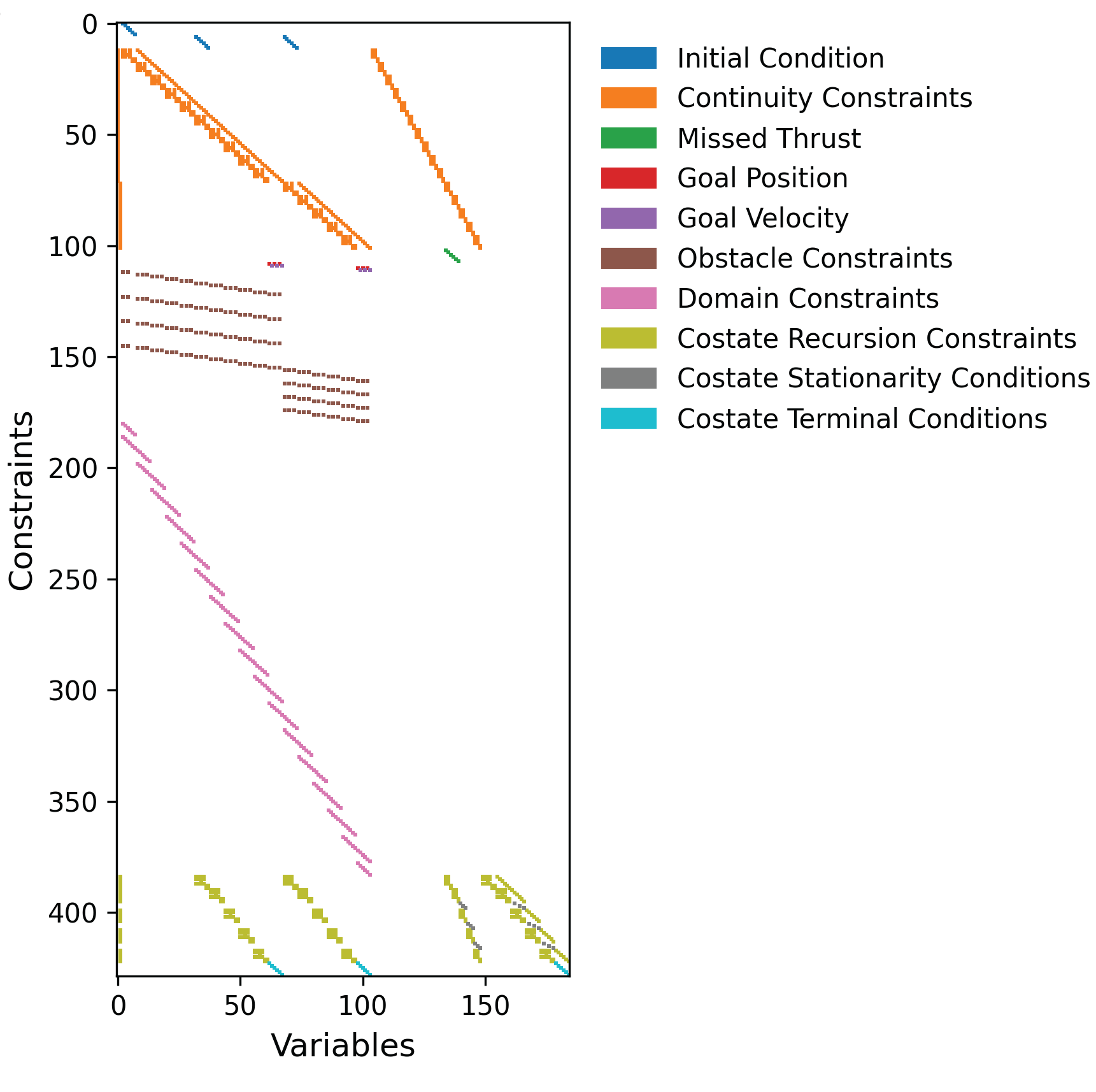} %[tsx/show help lines]
            \small
            \node[fill=white, opacity=1.0, text opacity=1, anchor=south, rotate=90] at (0.035,0.55) {Constraints};
            \node[fill=white, opacity=1.0, text opacity=1, anchor=south] at (0.30,0.0125) {Variables};
        \end{tikzonimage}
        \caption{Sparsity pattern of the Jacobian $\bm{J}(\bm{z})$, showing the block structure induced by multiple shooting and the decomposition of the constraints.}
        \label{fig:jacobian_sparsity}
    \end{figure}
}
{
}

Although all derivatives are computed using exact reverse-mode AD, we can easily validate the correctness of the Jacobian $\bm{J}(\bm{z})$ by comparing it against an independent finite-difference (FD) approximation. 
Such a comparison verifies that the AD pipeline correctly propagates derivatives through the dynamics and also serves as a safeguard against implementation errors that AD would otherwise differentiate exactly. 
Figure~\ref{fig:jacobian_sparsity_ad_vs_fd} shows the value of the derivatives obtained via AD as well as the pointwise relative difference $e_{\mathrm{rel}}$ between the AD Jacobian and a central-difference FD approximation with step size $\epsilon = 10^{-6}$ where,
\begin{equation}
    e_{\mathrm{rel}} \equiv \frac{|\bm{J}_{\mathrm{AD}} - \bm{J}_{\mathrm{FD}}|}{1 + |\bm{J}_{\mathrm{FD}}|}
\end{equation}
The observed discrepancies lie within the expected truncation error envelope for second order central differences on the order of $\mathcal{O}(\epsilon^2) \approx 10^{-12}$, with rounding effects up to $\approx 10^{-10}$, demonstrating numerical agreement and confirming the correctness of the Jacobian evaluation via AD.

\ifthenelse{\boolean{includefigures}}
{
    \begin{figure}[!htb]
        \centering
        \begin{subfigure}[t]{0.45\textwidth}
            \centering
            \includegraphics[keepaspectratio, width=\textwidth]{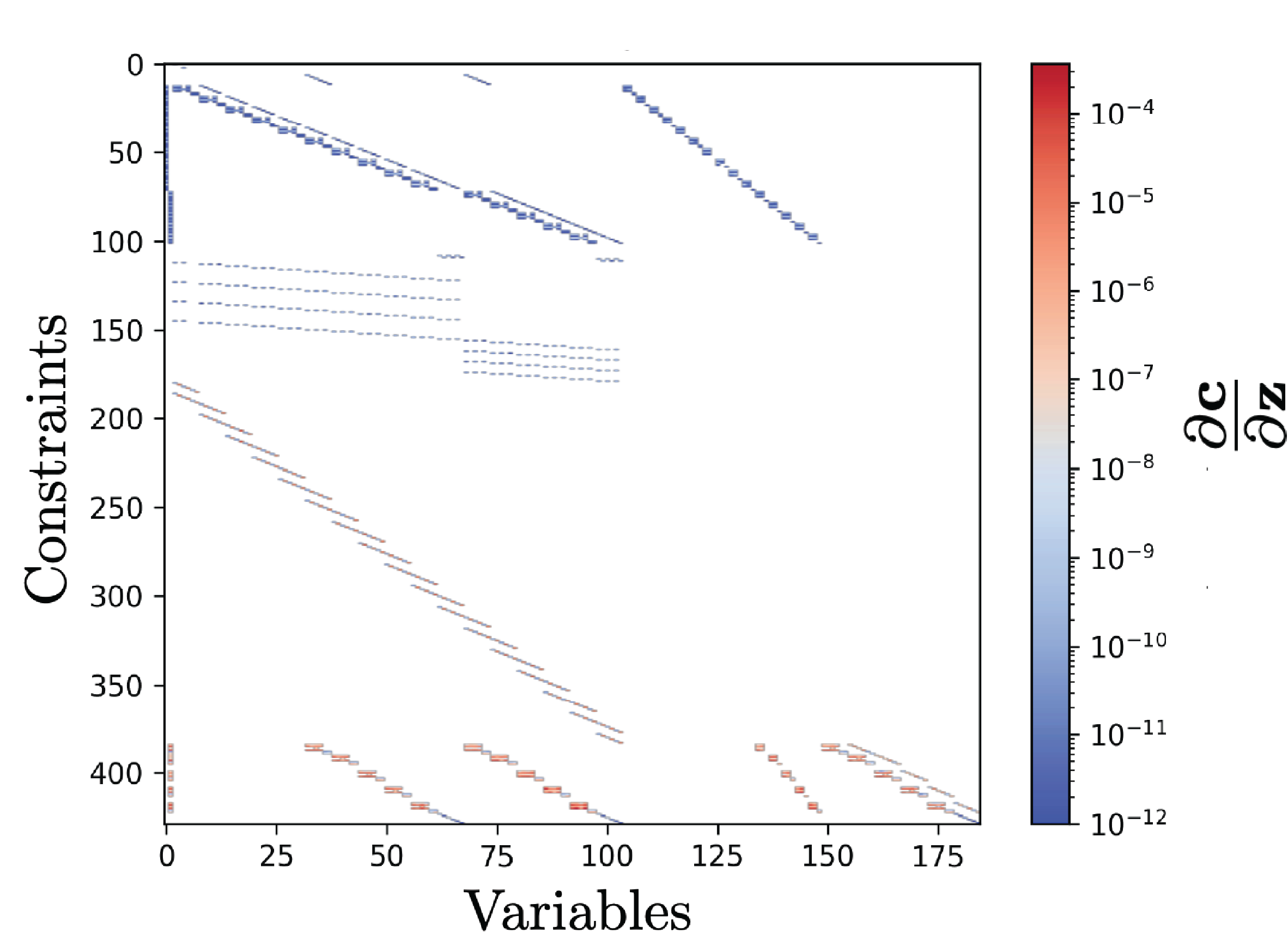}
            \caption{
            Absolute values $|[\bm{J}_{\mathrm{AD}}]_{i,j}| \equiv |\frac{\partial{c_j}}{\partial{z}_i}|$.
            }
            \label{fig:jacobian_sparsity_ad}
        \end{subfigure}
        \hfill
        \begin{subfigure}[t]{0.42\textwidth}
            \centering
            \includegraphics[keepaspectratio, width=\textwidth]{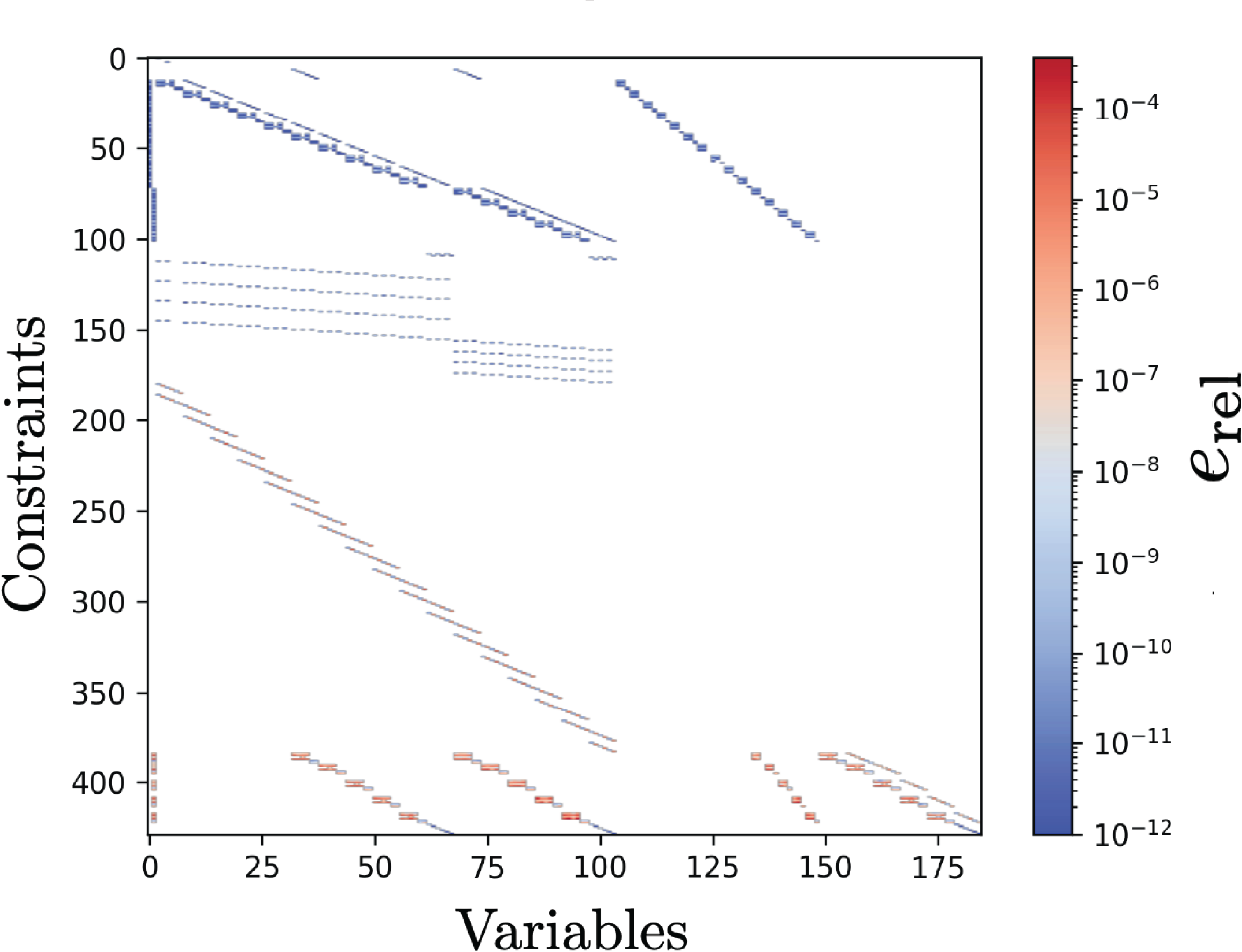}
            \caption{
            Pointwise relative difference $e_{\mathrm{rel}} \equiv \frac{|\bm{J}_{\mathrm{AD}} - \bm{J}_{\mathrm{FD}}|}{1 + |\bm{J}_{\mathrm{FD}}|}$
            }
            \label{fig:jacobian_sparsity_difference}
        \end{subfigure}
        \caption{
        Validation of the analytic Jacobian $\bm{J}_{\mathrm{AD}}$ against the finite-difference approximation $\bm{J}_{\mathrm{FD}}(\epsilon = 10^{-6})$ evaluated at a representative decision variable vector $\bm{z}$ in the circular case.
        }
        \label{fig:jacobian_sparsity_ad_vs_fd}
    \end{figure}
}
{
    % Do nothing
}

% ===========================================================================
\subsection{Hessian Bounds} \label{subsec:experimental results relative spacecraft motion:hessian bounds}

To bound the second order remainder term appearing in the Taylor expansion of the dynamics, we require a uniform upper bound on the Hessian of the nonlinear vector field over a tubular neighborhood $\mathcal T_{\rho}$ around the reference solution.  
Recall that the controlled dynamics take the form:
\begin{equation}
    \dot{\xi}_t = \tilde{f}(\xi_t, u_t),
\end{equation}
with $\xi_t \in \mathbb{R}^n$.  
For any fixed control input $u_t$, the Hessian of the dynamics with respect to the state is the third order tensor
\begin{equation}
    \nabla_{\xi}^2 \tilde{f}(\xi_t,u_t) = \left[\frac{\partial^2 \tilde{f}_i}{\partial \xi_j \partial \xi_k}(\xi_t,u_t)\right]_{ijk},
\end{equation}
which acts on vectors through the bilinear form
\begin{equation}
    \nabla_{\xi}^2 \tilde{f}(\xi_t,u_t)[v,w] = \sum_{j,k=1}^n\frac{\partial^2 \tilde{f}}{\partial \xi_j  \partial \xi_k}(\xi_t,u_t)v_j w_k.
\end{equation}
By Taylor's theorem in mean-value form (see Appendix ~\ref{app:taylor_expansion}) the deviation $\tilde{\xi}_t = \xi_t^\omega - \xi_t^\dagger$ satisfies
\begin{equation}
    r(\tilde{\xi}_t,t) = \frac{1}{2} \nabla_{\xi}^2 \tilde{f}(\xi_t^\ddagger, u_t^\dagger)[\tilde{\xi}_t, \tilde{\xi}_t],
\qquad
\xi_t^\ddagger = \xi_t^\dagger + \theta \tilde{\xi}_t \ \text{for an unknown } \theta\in(0,1).
\end{equation}
Thus, a sufficient condition for the remainder to be uniformly bounded on a tube
\begin{equation}
    \mathcal{T}_\rho \equiv \{\xi : \|\xi - \xi_t^\dagger\|\le \rho\},
\end{equation}
is the existence of a constant $H(\rho)$ such that
\begin{equation}
    \bigl\|\nabla_{\xi}^2 \tilde{f}(\xi_t^\ddagger,u_t^\dagger)\bigr\| \le H(\rho), \qquad \forall \xi_t^\ddagger \in \mathcal{T}_\rho.
\end{equation}
To quantify $\|\nabla^2_{\xi} \tilde{f}\|$, we evaluate the Hessian tensor componentwise and reshape it into the corresponding $n\times n$ linear map acting as $v\mapsto \nabla_{\xi}^2 \tilde{f}(\xi^\dagger_t,u_t^\dagger)[v,v]$.  
We report three standard matrix norms for a matrix $\bm{H}$:
\begin{subequations} \label{eq:hessian_norms}
    \begin{align}
        \|\bm{H}\|_{\mathrm{oper}} &\equiv \sup_{\|v\|_2=1} \|\bm{H}v\|_2 &&\text{(operator norm)}, \\
        \|\bm{H}\|_{\mathrm{frob}} &\equiv \sqrt{\sum_{i,j} H_{ij}^2} &&\text{(Frobenius norm)}, \\
        \|\bm{H}\|_{\mathrm{spec}} &\equiv \max\{|\lambda|:\lambda\in\mathrm{eig}(\bm{H})\} &&\text{(spectral radius)}.
    \end{align}
\end{subequations}
All matrix norms on $\mathbb{R}^{n\times n}$ are equivalent but their numerical values can differ substantially due to problem dependent scaling of the underlying coordinates. 
In poorly scaled systems, such as those with heterogeneous units, anisotropic dynamics, or large disparities in sensitivity along different state directions, a particular norm may either exaggerate or mask curvature effects. 
Examining several norms is therefore informative: the operator norm captures the worst-case amplification direction and is most relevant for assessing the local stability of the linearization; the Frobenius norm reflects an average measure of curvature and is less sensitive to isolated large entries; and the spectral radius emphasizes the dominant eigenmode of the Hessian and therefore highlights the principal direction of nonlinear distortion. 
For example, in spacecraft relative motion, position variables may be expressed in kilometers while velocity variables appear in kilometers per second. 
The Hessian entries associated with the velocity components can therefore be several orders of magnitude smaller than those associated with position, even if the underlying nonlinear coupling is dynamically significant. 
In such a case, the Frobenius norm may substantially underrepresent curvature because the many small velocity-related entries dilute the effect of a few large position-related entries, whereas the operator norm or spectral radius more accurately reflects the dominant direction of nonlinear growth. 
Conversely, if the curvature is distributed broadly across many small terms, the Frobenius norm may provide a more faithful measure than the operator norm. 
Considering all three norms together provides a more complete and scale-aware characterization of the local curvature of the dynamics, even though the theoretical bound requires only the existence of a uniform constant $H$ rather than a specific choice of norm.

To assess how tightly the Hessian at the reference state approximates the maximal curvature within a small tube, we evaluate the three norms above over a $[-5,5]$ km $\times$ $[-5,5]$ km domain in the $(x,y)$-plane centered at $\xi^\dagger_t$.  
Figure~\ref{fig:hessian_comparison_third_order} shows the resulting fields where the position of the target spacecraft $O$ is indicated by the \textcolor{Red}{red} star and the initial condition of the chaser spacecraft $\bm{\hat{\xi}}_{P/\mathcal{O}}^{\mathcal{H}}$ is denoted by the \textcolor{DodgerBlue}{blue} circle.  
All three norms exhibit smooth but minimal variations across the domain.
There also appears to be negligible differences between the three norms considered, which further indicates minimal nonlinear effects for this particular problem.
\ifthenelse{\boolean{includefigures}}
{
    \begin{figure}[!htb]
        \centering
        \begin{tikzonimage}[keepaspectratio, width=\linewidth]{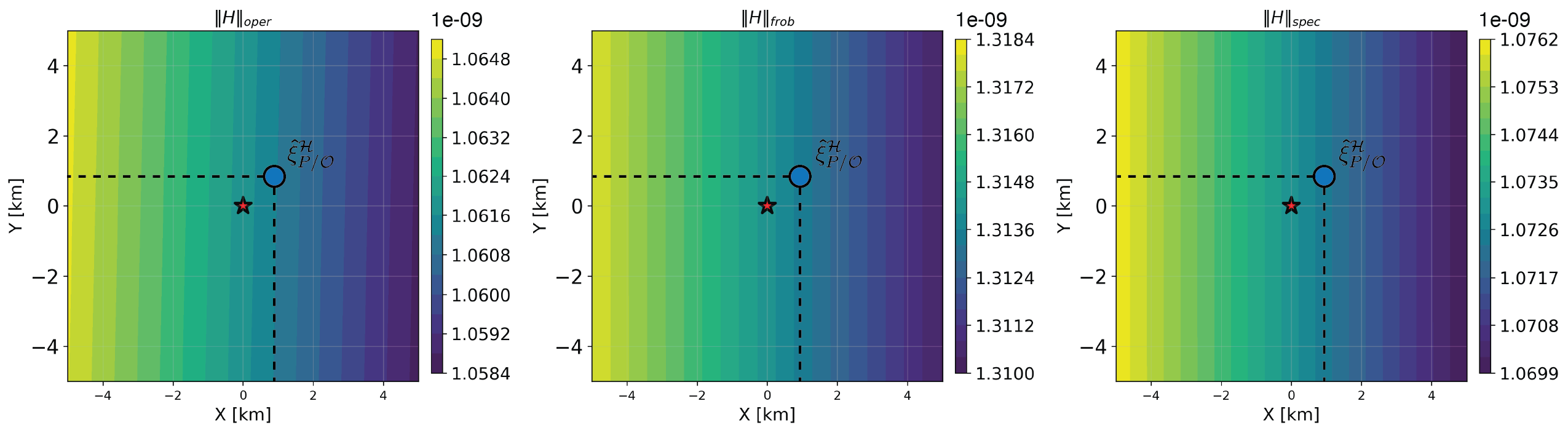}%[tsx/show help lines]
            \small
            \node[fill=white, opacity=1.0, text opacity=1, anchor=south, rotate=90] at (0.025,0.5) {Y [km]};
            \node[fill=white, opacity=1.0, text opacity=1, anchor=south, rotate=90] at (0.36,0.5) {Y [km]};
            \node[fill=white, opacity=1.0, text opacity=1, anchor=south, rotate=90] at (0.70,0.5) {Y [km]};
            \node[fill=white, opacity=1.0, text opacity=1, anchor=south] at (0.16,-0.025) {X [km]};
            \node[fill=white, opacity=1.0, text opacity=1, anchor=south] at (0.49,-0.025) {X [km]};
            \node[fill=white, opacity=1.0, text opacity=1, anchor=south] at (0.83,-0.025) {X [km]};
            \node[fill=white, opacity=1.0, text opacity=1, anchor=south] at (0.16,0.93) {$\|\bm{H}\|_{\mathrm{oper}}$};
            \node[fill=white, opacity=1.0, text opacity=1, anchor=south] at (0.49,0.93) {$\|\bm{H}\|_{\mathrm{frob}}$};
            \node[fill=white, opacity=1.0, text opacity=1, anchor=south] at (0.83,0.93) {$\|\bm{H}\|_{\mathrm{spec}}$};
        \end{tikzonimage}
        \caption{Spatial distribution of Hessian norms (operator $\|\bm{H}\|_{\mathrm{oper}}$, frobenius $\|\bm{H}\|_{\mathrm{frob}}$, and spectral $\|\bm{H}\|_{\mathrm{spec}}$) for the nonlinear dynamics in the circular case, evaluated on a $[-5,5]$ km domain centered at the reference state (indicated by the \textcolor{Red}{red} star).}
        \label{fig:hessian_comparison_third_order}
    \end{figure}
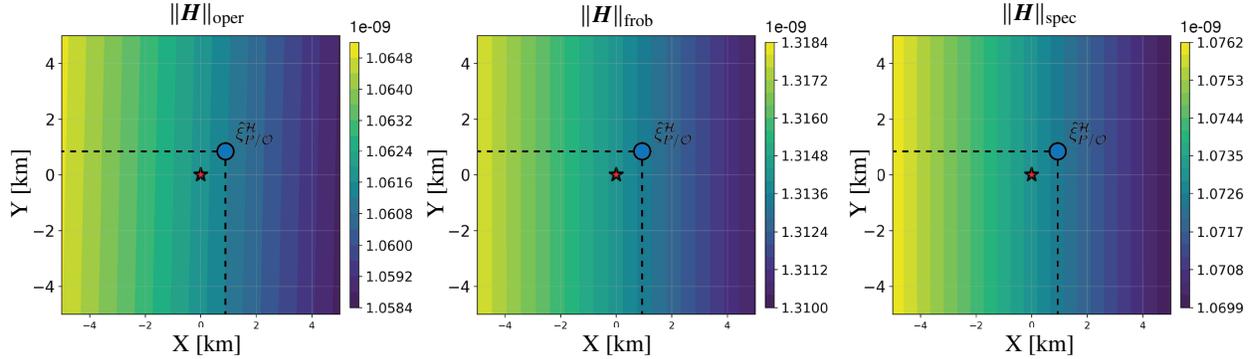
}
{
}
To further illustrate this behavior, Fig.~\ref{fig:hessian_norm_crosssection} presents cross-sections of the same fields along the $x$-axis (Fig.~\ref{fig:hessian_norm_crosssection_x}) and $y$-axis (Fig.~\ref{fig:hessian_norm_crosssection_y}).  
Across both directions, the three norms vary only minimally, and the value at the reference state again closely matches the local maxima and minima.  
These observations support the use of the Hessian bound evaluated at the reference state $\xi_t^\dagger$ as a sufficient approximation of the conservative bound representative of nearby states in the tube $\mathcal{T}_\rho$.
\ifthenelse{\boolean{includefigures}}
{
    \begin{figure}[!htb]
        \centering
        \begin{subfigure}[t]{0.45\textwidth}
            \centering
            \includegraphics[keepaspectratio, width=\textwidth]{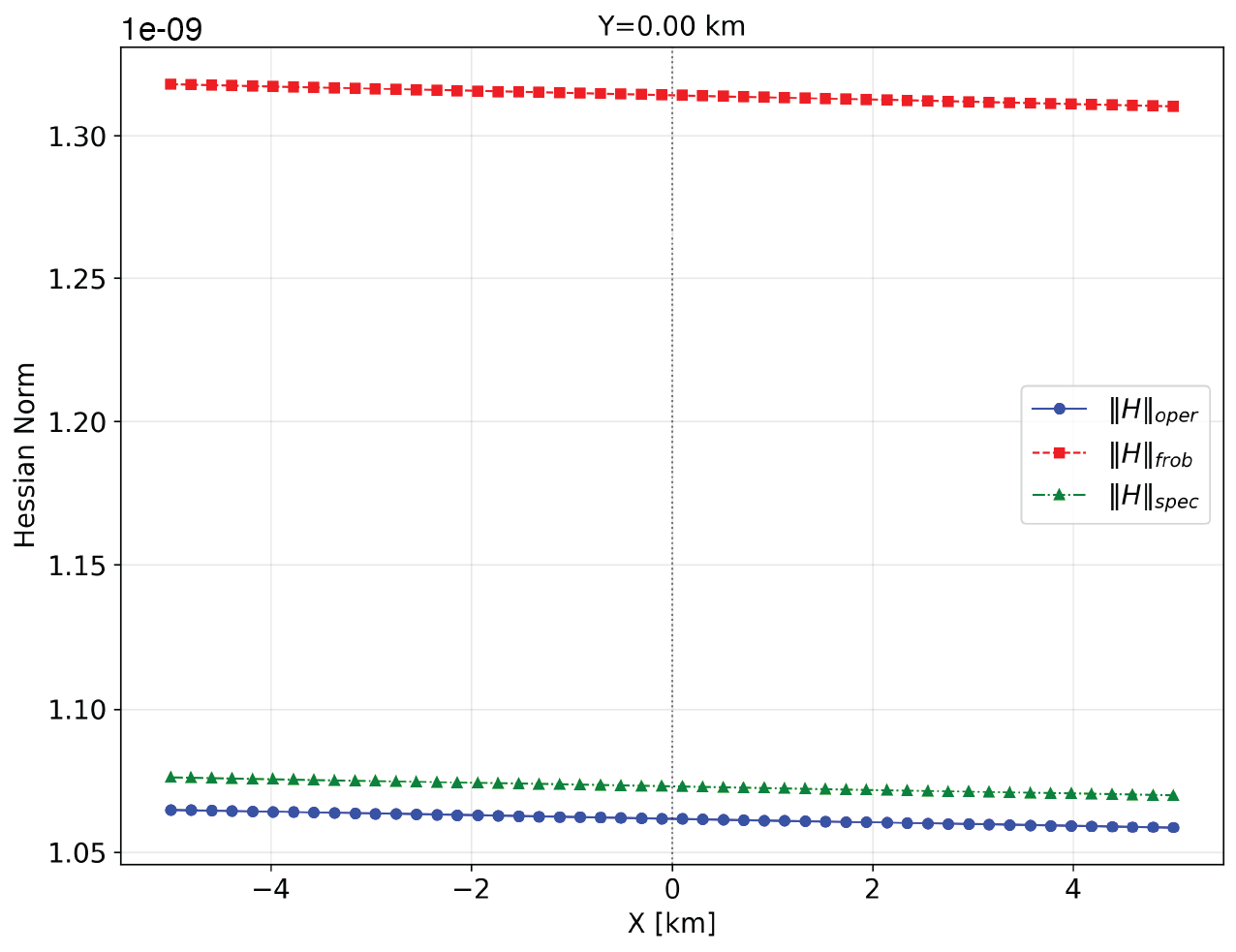}
            \caption{
                Cross-section at $Y=0.0$ km
            }
            \label{fig:hessian_norm_crosssection_y}
        \end{subfigure}
        \hfill
        \begin{subfigure}[t]{0.45\textwidth}
            \centering
            \includegraphics[keepaspectratio, width=\textwidth]{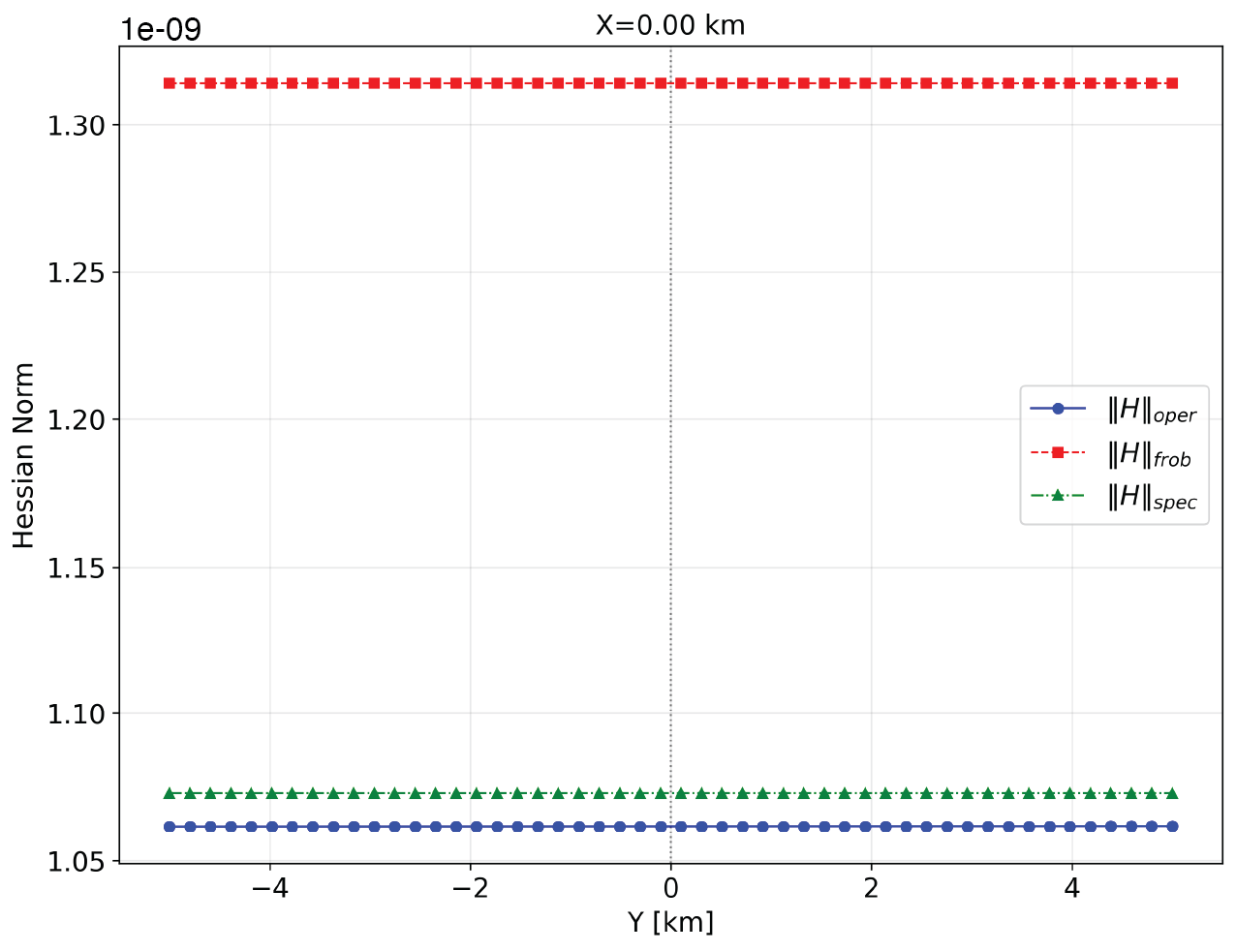}
            \caption{
                Cross-section at $X=0.0$ km
            }
            \label{fig:hessian_norm_crosssection_x}
        \end{subfigure}
        \caption{
        Cross-sections of the Hessian norm shown in Fig.~\ref{fig:hessian_comparison_third_order}, taken along the $x$ and $y$ axes through the reference state.
        }
        \label{fig:hessian_norm_crosssection}
    \end{figure}
}
{
    % Do nothing
}

In addition to characterizing curvature within a fixed tube around the reference state, it is useful to examine how nonlinear effects strengthen as the target orbit departs from circularity. 
To this end, we compute, for each eccentricity $e \in [0,0.9]$, the maximum operator norm $\|\bm{H}\|_{\mathrm{oper}}$ over the full $[-5,5]~\text{km} \times [-5,5]~\text{km}$ planar domain around the target spacecraft and over all true anomalies $\nu \in [0,2\pi]$. 
For each eccentricity, the semimajor axis is adjusted so that the periapse altitude remains identical to the orbital altitude in the circular case.
The resulting trend is shown in Fig.~\ref{fig:max-operator-norm-eccentricity}.
The two specific cases for the eccentric orbit are highlighted in Fig.~\ref{fig:hessian_norm_crosssection_y}.
We observe that the maximum curvature occurs near periapse, while the minimum curvature occurs near apoapse (Fig.~\ref{fig:max-operator-norm-fixed_sma}).
Furthermore, as $e$ increases, the maximum curvature remains constant since the periapse altitude remains constant and the minimal curvature decays exponentially as the apoapse altitude increases.

\ifthenelse{\boolean{includefigures}}
{
    \begin{figure}[!htb]
        \centering
        \begin{subfigure}[b]{0.625\textwidth}
            \centering
            \includegraphics[keepaspectratio, width=\textwidth]{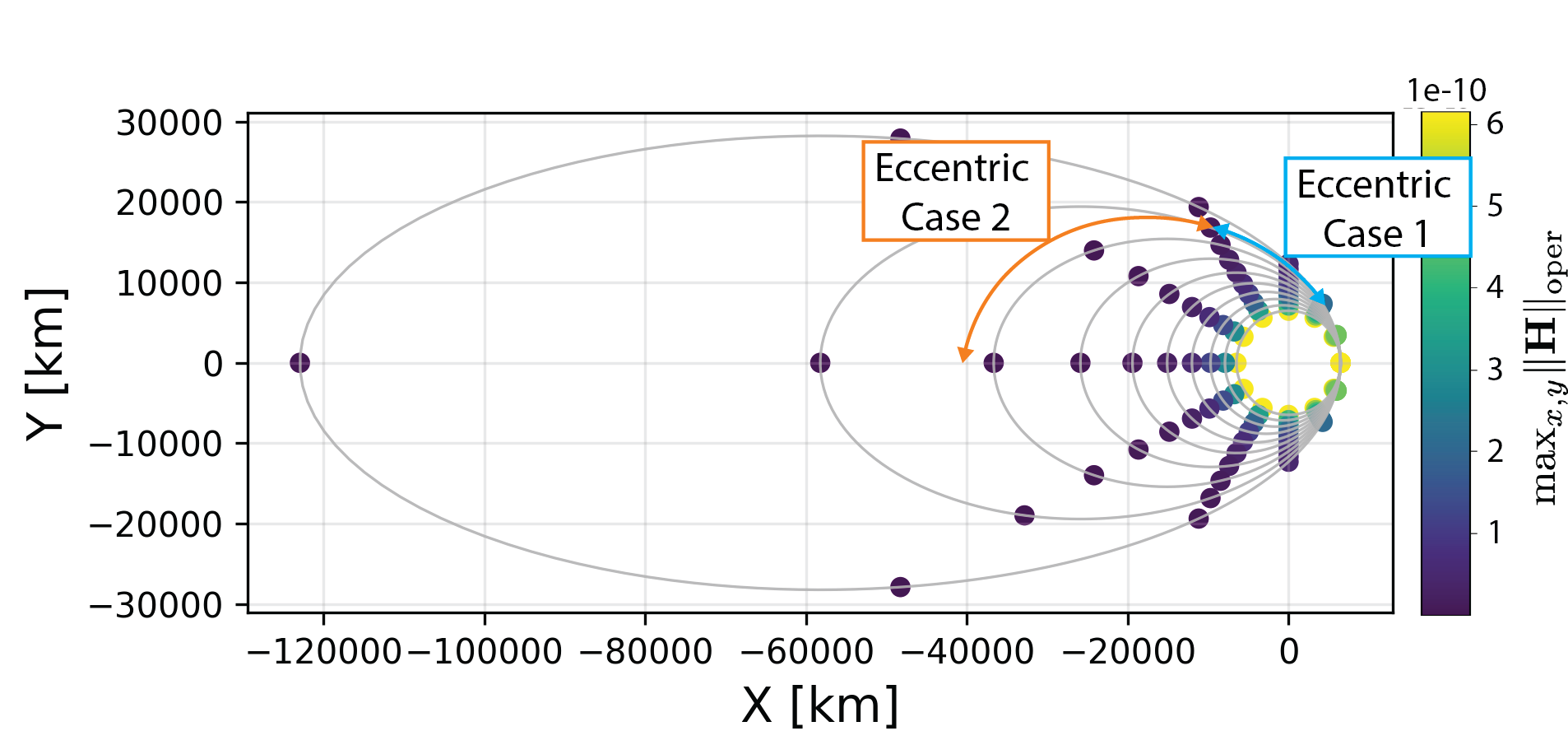}
            \caption{
                Family of eccentric target orbits sharing a common periapse altitude, with discrete $\nu$ samples colored by $\max_{x, y} \|\bm{H}\|_{\mathrm{oper}}$.
            }
            \label{fig:max-operator-norm-varying_sma}
        \end{subfigure}
        \hfill
        \begin{subfigure}[b]{0.35\textwidth}
            \centering
            \includegraphics[keepaspectratio, width=\textwidth]{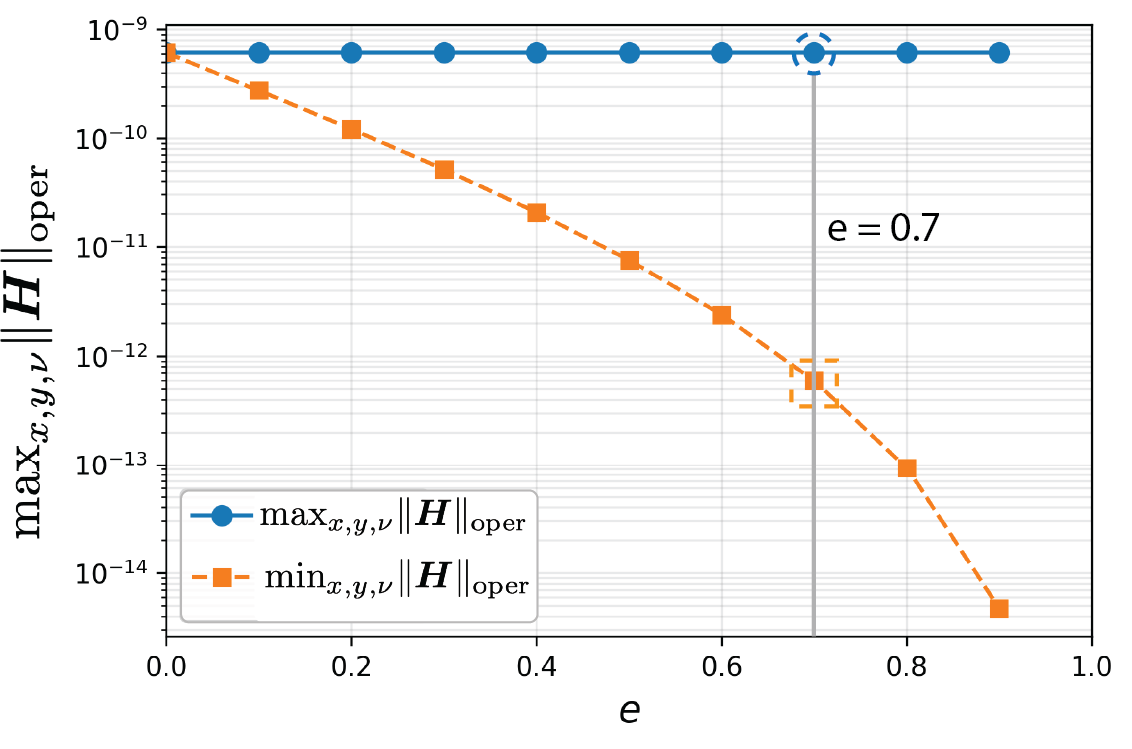}
            \caption{
                Minimum and maximum values of $\|\bm{H}\|_{\mathrm{oper}}$ attained along each orbit in the same family
            }
            \label{fig:max-operator-norm-fixed_sma}
        \end{subfigure}
        \caption{
        Maximum operator norm $\max_{x,y,\nu} \| \bm{H}\|_{\mathrm{oper}}$ for eccentric target orbits. The norm is evaluated over the planar domain $x,y\in[-5,5]$ km about the reference origin and over all true anomalies $\nu\in[0,2\pi]$, for eccentricities $e\in\{0.0,\cdots,0.9\}$ with fixed periapse altitude $r_p = 500$ km.
        }
        \label{fig:max-operator-norm-eccentricity}
    \end{figure}
}
{
    % Do nothing
}

% ===========================================================================
\subsection{Example Solutions} \label{subsec:experimental results relative spacecraft motion:example solutions}

In this section, we present a representative solution generated by the proposed bi-level framework to illustrate the qualitative behavior of both the reference and realization trajectories in the circular case. 
The example highlights how the follower trajectory diverges from the leader trajectory during an MTE, how the recovery maneuver steers the system back toward the goal set while satisfying obstacle avoidance and control constraints.

Figure~\ref{fig:example-trajectory-3d} shows the temporal evolution of the leader and follower trajectories. 
The leader trajectory (\textcolor{DodgerBlue}{blue}) represents the nominal solution generated under full control authority, shaped to accommodate the MTE encountered by the follower solution. 
The follower trajectory (\textcolor{Orange}{orange}) coincides with the leader until the MTE begins at the location marked in \textcolor{Red}{red}, after which it experiences a natural drift due to loss in control authority. 
The resulting deviation grows along the locally unstable directions of the nonlinear dynamics before active control resumes. 
Once control authority is restored, the follower executes continuous correction maneuvers to track the leader while avoiding the obstacles. 
The zoomed two-dimensional view on the right highlights the fine-scale differences between the trajectories near the terminal region. 
A recurring feature across the majority of solutions is a deliberate, slight overshoot by the leader immediately preceding the MTE, which serves to mitigate the impact of the subsequent actuator loss in the follower trajectory.
\ifthenelse{\boolean{includefigures}}
{
\begin{figure}[!htb]
  \centering
  \includegraphics[width=1.0\linewidth]{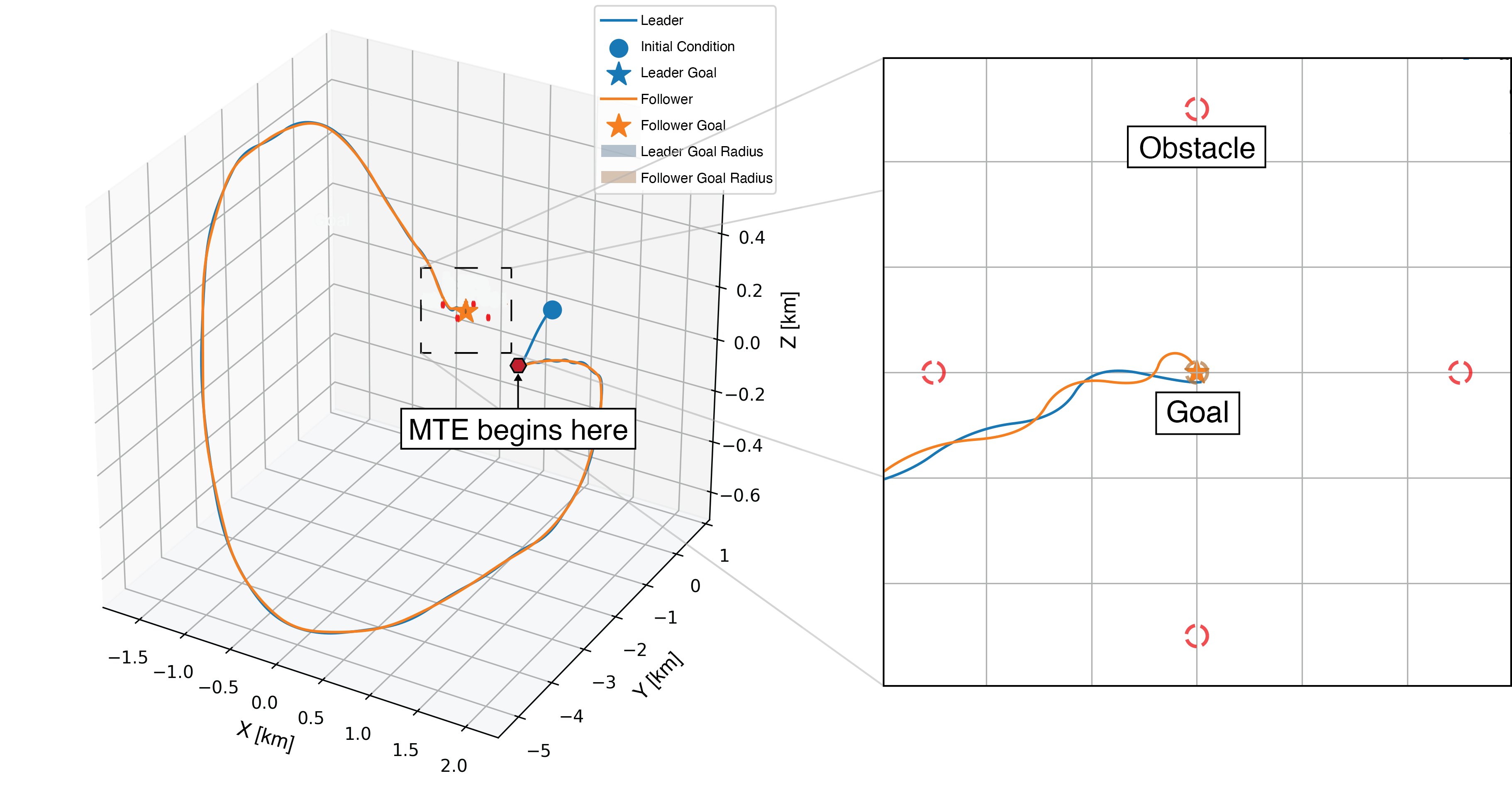}
  \caption{Leader (\textcolor{DodgerBlue}{blue}) and follower (\textcolor{Orange}{orange}) trajectories for an example MTE. The follower drifts during the MTE and subsequently performs an optimal recovery maneuver to reach the terminal goal while avoiding obstacles. The right panel shows a magnified view of the terminal region.}
  \label{fig:example-trajectory-3d}
\end{figure}
}
{
}
Figure~\ref{fig:example-throttle-profile} shows the corresponding throttle profiles for the leader (top) and follower (bottom). 
The leader throttle exhibits a piecewise-constant structure induced by the finite-burn discretization and displays an anticipatory, more energetic response immediately prior to the MTE. 
This behavior allows the leader to partially offset the impending loss of control authority in the follower. 
During the missed thrust interval of duration $\Delta\tau$, the leader correspondingly reduces its control effort, thereby mitigating the downstream impact of the outage on the coupled follower dynamics.
For the follower, the control input is identically zero throughout the missed thrust interval, after which the throttle exhibits a more energetic transient response as it compensates for the accumulated tracking error. 
Over the missed thrust interval, the minimum forcing amplitude corresponding to $\underline{u}^\dagger$ and the maximum forcing amplitude corresponding to $\overline{u}^\dagger$ are highlighted, which enter the theoretical bounds described in Theorem~\ref{thm:maximum_missed_thrust_duration}.
The follower control additionally exhibits higher-frequency variations compared to the leader, consistent with the costate-driven correction terms arising from the embedded KKT conditions in the bi-level formulation.
\ifthenelse{\boolean{includefigures}}
{
\begin{figure}[!htb]
  \centering
  \includegraphics[width=0.7\linewidth]{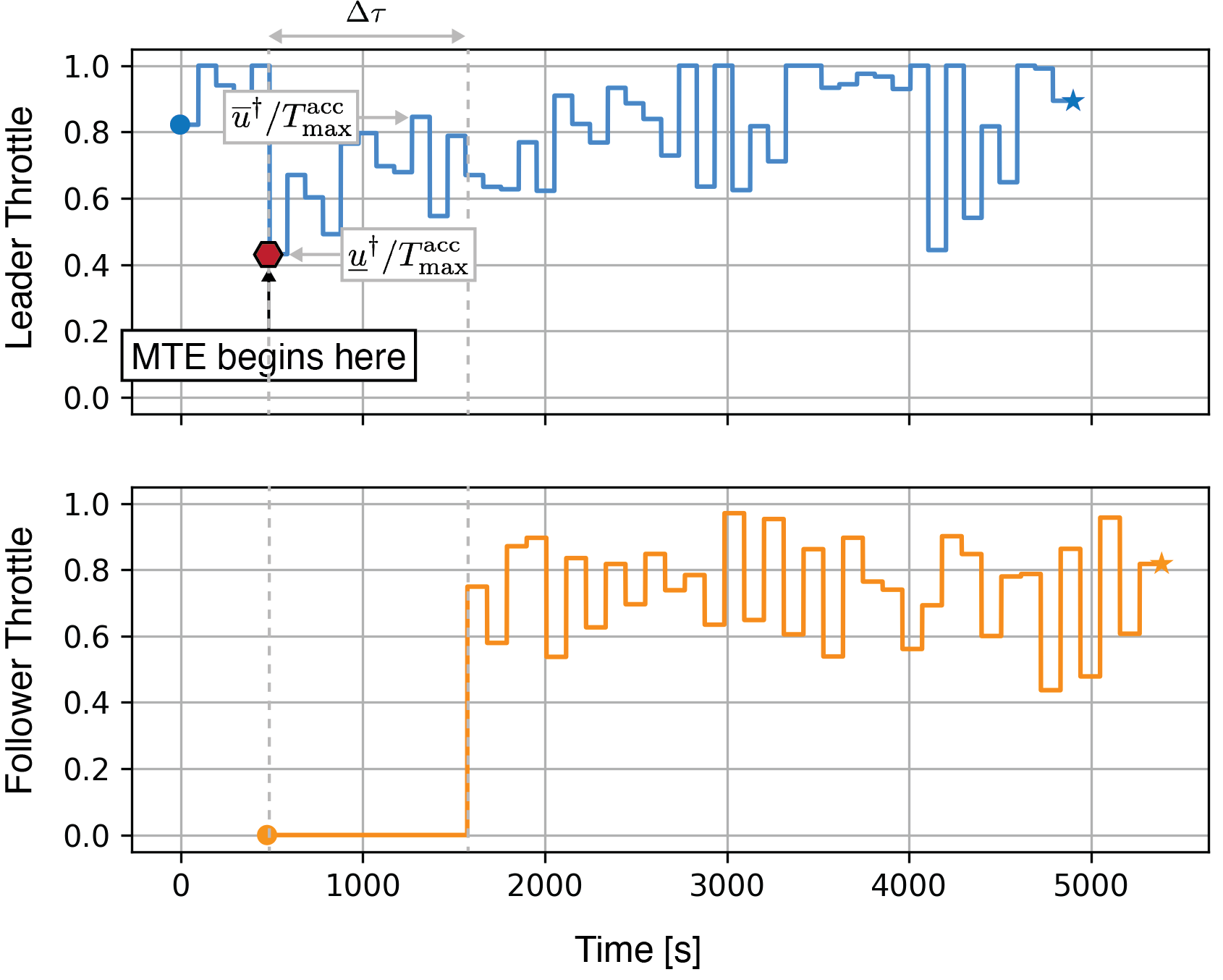}
  \caption{Leader (top) and follower (bottom) throttle profiles for the example solution in Fig.~\ref{fig:example-trajectory-3d}, showing the missed thrust interval followed by the follower's optimal recovery maneuver.}
  \label{fig:example-throttle-profile}
\end{figure}
}
{
}

% ===========================================================================
\subsection{Safe State-Space Radius $\delta$ and Maximum Missed Thrust Duration $\delta\tau_{\max}$} \label{subsec:experimental results relative spacecraft motion:safe state-space radius delta and maximum missed thrust duration delta tau_max}

In this section, we assess the numerical performance of the proposed framework by applying it to three representative cases, as described in \S~\ref{subsec:experimental results relative spacecraft motion:cwh equations with third-order nonlinearities}. Table~\ref{tab:numerical_results} reports the feasibility ratios obtained for each case.
Among the three cases, the circular orbit scenario exhibits the lowest feasibility ratio, reflecting the persistently strong nonlinear effects present throughout the entire target orbit. 
For the eccentric orbit scenarios, Eccentric (Case~1), which corresponds to the target spacecraft motion closer to periapse, yields a lower feasibility ratio than Eccentric (Case~2), which corresponds to motion nearer to apoapse and is characterized by comparatively weaker nonlinearities.
\begin{table}[!htbp]
\centering
\begin{tabular}{cccc}
\hline
\textbf{Scenarios}          & \textbf{\begin{tabular}[c]{@{}c@{}}Number of \\ Initializations\end{tabular}} & \textbf{\begin{tabular}[c]{@{}c@{}}Number of \\ Solutions\end{tabular}} & \textbf{\begin{tabular}[c]{@{}c@{}}Feasibility Ratio\\ {[}\%{]}\end{tabular}} \\ \hline
\textbf{Circular}           & 500                                                                           & 240                                                                     & 48.0                                                                          \\
\textbf{Eccentric (Case 1)} & 500                                                                           & 256                                                                     & 51.2                                                                          \\
\textbf{Eccentric (Case 2)} & 500                                                                           & 272                                                                     & 54.4                                                                          \\ \hline
\end{tabular}
\caption{}
\label{tab:numerical_results}
\end{table}

Figure~\ref{fig:parameters-circular} shows the empirical distributions of the effective forcing bounds $f_u^{\min}$ and $f_u^{\max}$, together with the second-order curvature measure $H$, for the circular orbit case. 
The first two subfigures indicate that $f_u^{\min}$ show some spread, while $f_u^{\max}$ is more tightly concentrated, suggesting that most reference solutions operate close to the available control authority over the missed trust interval. 
The distribution of $H$ is narrowly concentrated at small values, indicating weak curvature of the dynamics over the admissible region. 
\ifthenelse{\boolean{includefigures}}
{
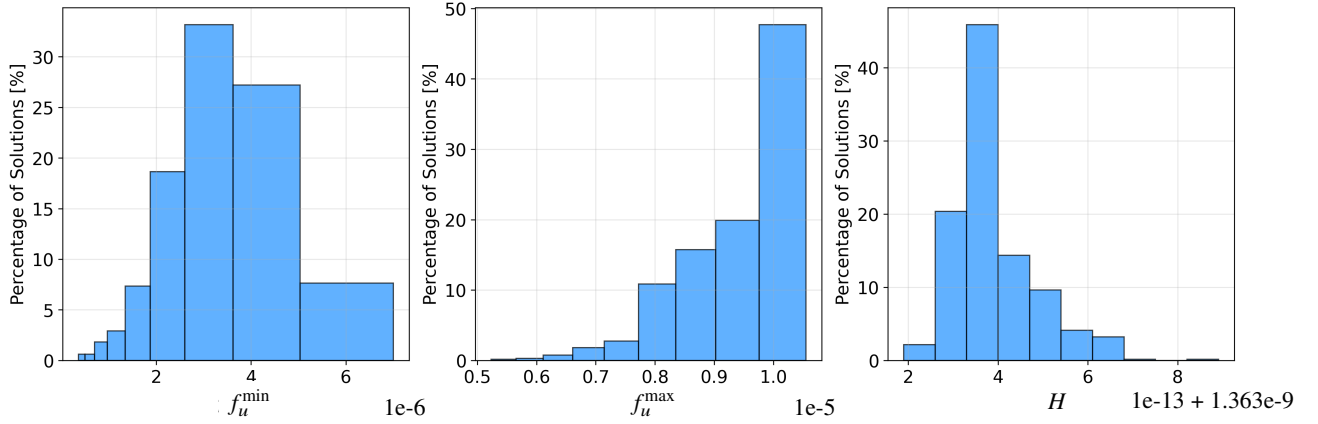
\begin{figure}[!htb]
  \centering
  \begin{tikzonimage}[keepaspectratio, width=1.0\linewidth]{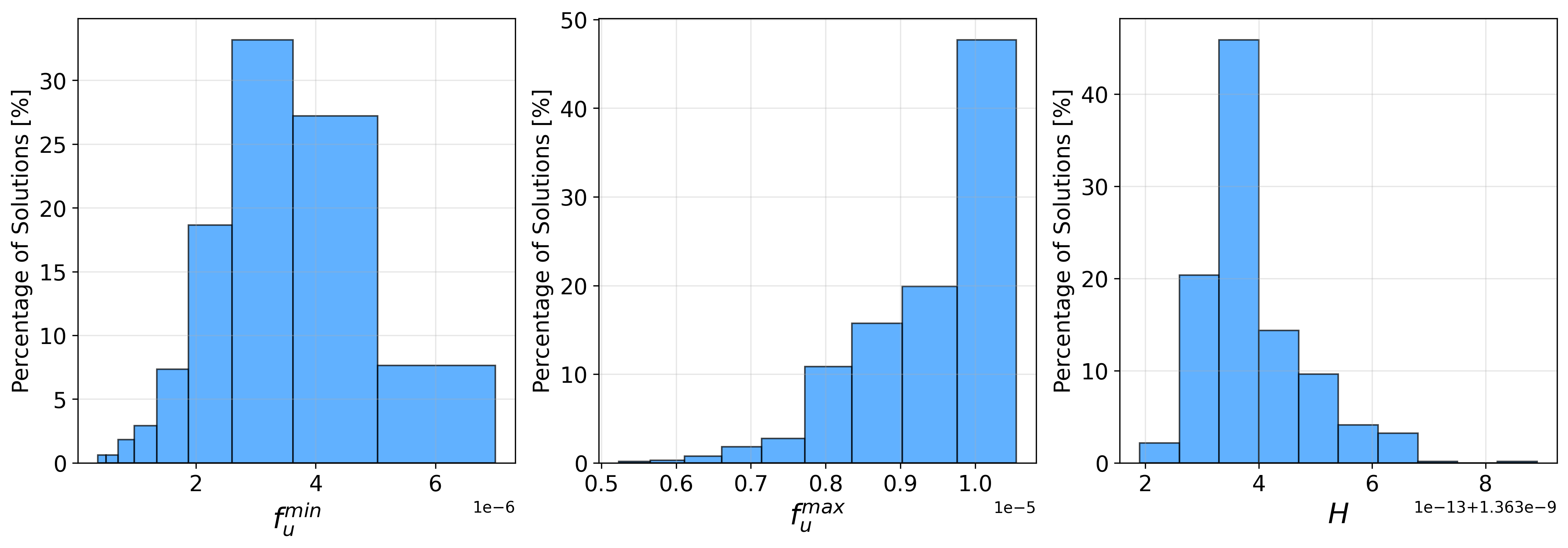}%[tsx/show help lines]
    \small
    \node[fill=white, opacity=1.0, text opacity=1, anchor=south] at (0.2,0.0075) {$f_u^{\min}$};
    \node[fill=white, opacity=1.0, text opacity=1, anchor=south] at (0.525,0.0075) {$f_u^{\max}$};
    \node[fill=white, opacity=1.0, text opacity=1, anchor=south] at (0.85,0.025) {$H$};
    \node[fill=white, opacity=1.0, text opacity=1, anchor=south] at (0.325,0.01) {1e-6};
    \node[fill=white, opacity=1.0, text opacity=1, anchor=south] at (0.655,0.01) {1e-5};
    \node[fill=white, opacity=1.0, text opacity=1, anchor=south] at (0.975,0.025) {1e-13 + 1.363e-9};
  \end{tikzonimage}
  \caption{Empirical distributions of $f_u^{\min}$, $f_u^{\max}$ and $H$ across the solution ensemble for the circular orbit case.}
  \label{fig:parameters-circular}
\end{figure}
}
{
}

Figure~\ref{fig:parameters-eccentric} reports the corresponding empirical distributions for the eccentric orbit with $e=0.7$, separating the two true anomaly intervals considered. 
For Case~1 ($\nu\in[60^\circ,120^\circ]$), which corresponds to motion closer to periapse, the distribution of $H$ exhibits a noticeably larger spread, reflecting stronger and more variable nonlinear effects. 
In this regime, both $f_u^{\min}$ and $f_u^{\max}$ attain smaller values, indicating a decrease in effective forcing amplitudes. 
By contrast, for Case~3 ($\nu\in[120^\circ,180^\circ]$), corresponding to motion nearer to apoapse, the distribution of $H$ is tightly concentrated at smaller values, and both $f_u^{\min}$ and $f_u^{\max}$ shift toward larger magnitudes. 
This combination of weaker curvature and higher effective forcing leads to substantially less restrictive missed thrust duration bounds near apoapse compared to regions closer to periapse, contributing to the higher feasibility ratio observed.
\ifthenelse{\boolean{includefigures}}
{
\begin{figure}[!htb]
  \centering
  \begin{tikzonimage}[keepaspectratio, width=1.0\linewidth]{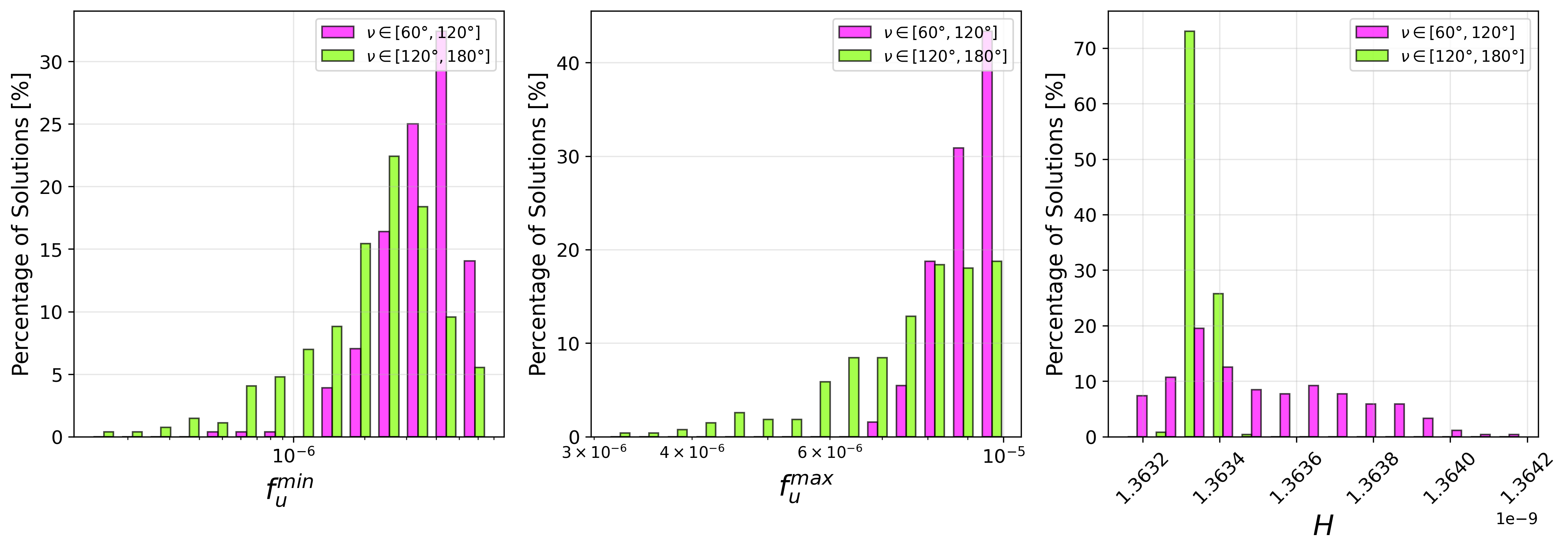}%[tsx/show help lines]
    \small
    \node[fill=white, opacity=1.0, text opacity=1, anchor=south] at (0.19,0.06) {$f_u^{\min}$};
    \node[fill=white, opacity=1.0, text opacity=1, anchor=south] at (0.515,0.06) {$f_u^{\max}$};
    \node[fill=white, opacity=1.0, text opacity=1, anchor=south] at (0.85,0.0075) {$H$};
    \node[fill=white, opacity=1.0, text opacity=1, anchor=south] at (0.97,0.0075) {1e-9};
  \end{tikzonimage}
  \caption{Empirical distributions of $f_u^{\min}$, $f_u^{\max}$ and $H$ across the solution ensemble for the eccentric orbit case with $e=0.7$.}
  \label{fig:parameters-eccentric}
\end{figure}
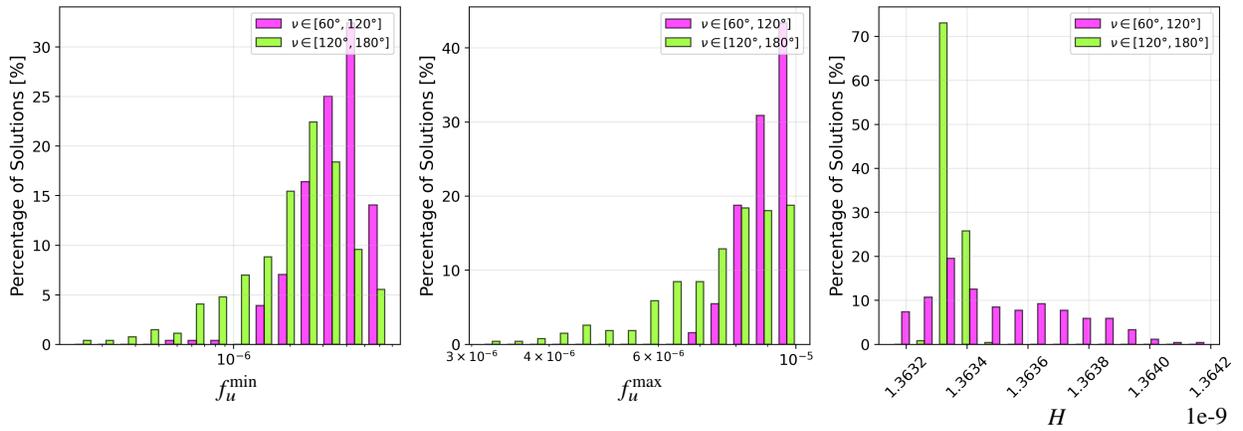
}
{
}

Having characterized how $f_u^\mathrm{min}$, $f_u^\mathrm{max}$ and $H$ vary across the circular and eccentric cases, we now examine how these parameter distributions translate into the certified safe radius $\delta_{\max}$ and the maximum allowable missed thrust duration $\delta\tau_{\max}$. 
In particular, we assess the conservatism of the theoretical certificates by comparing them against quantities extracted directly from the nonlinear simulations.
For each reference solution, we construct three versions of the $\delta\tau_{\mathrm{max}}$, reflecting increasing levels of trajectory specific information:
\begin{enumerate}[label=(\roman*)]
    \item a \emph{theoretical} pair $(\delta_{\max}^{\mathrm{theoretical}}, \delta\tau_{\max}^{\mathrm{theoretical}})$ obtained by applying Theorem~\ref{thm:maximum_missed_thrust_duration} to the reference trajectory using the uniform bounds $(\alpha,\beta,H,f_u^{\min},f_u^{\max})$ obtained from the dataset;
    \item a \emph{computed} pair $(\delta_{\max}^{\mathrm{computed}}, \delta\tau_{\max}^{\mathrm{computed}})$ obtained by first measuring the maximum deviation of the realization trajectory from its reference during the missed thrust interval (yielding $\delta_{\max}^{\mathrm{computed}}$) and then substituting this value into the analytical expression for $\delta\tau_{\max}$;
    \item an \emph{actual} missed thrust duration $\delta\tau_{\max}^{\mathrm{actual}}$, obtained directly from the nonlinear simulation of the realization trajectory, without invoking any bounds or analytical approximations.
\end{enumerate}

Figure~\ref{fig:delta-tau-circular} summarizes the resulting distributions for the circular-orbit case. The top panel shows that the theoretical safe radii $\delta_{\max}^{\mathrm{theoretical}}$ are tightly clustered at small values. By contrast, the computed radii $\delta_{\max}^{\mathrm{computed}}$ are systematically larger and shifted toward higher values, as they are based on the actual maximum deviations attained by the trajectories rather than on worst-case supremum bounds.
A similar trend is observed in the bottom panel for the maximum missed thrust duration. The theoretical durations $\delta\tau_{\max}^{\mathrm{theoretical}}$ are narrowly concentrated near small values, whereas the computed durations $\delta\tau_{\max}^{\mathrm{computed}}$ lie significantly closer to the numerically observed missed thrust durations $\delta\tau_{\max}^{\mathrm{actual}}$. In particular, we observe
\begin{equation}
\delta\tau_{\max}^{\mathrm{theoretical}} \le \delta\tau_{\max}^{\mathrm{computed}} \ll \delta\tau_{\max}^{\mathrm{actual}},
\end{equation}
demonstrating that incorporating trajectory specific information through measured deviations substantially narrows the gap to achievable performance, while the purely theoretical bound remains conservative by construction.
However, the relative ordering across the solution ensemble is preserved such that trajectories with larger certified values of $\delta\tau_{\max}^{\mathrm{theoretical}}$ also tend to exhibit larger $\delta\tau_{\max}^{\mathrm{computed}}$ and $\delta\tau_{\max}^{\mathrm{actual}}$. 
This indicates that the certificate is order preserving, even though it underestimates the absolute tolerable missed thrust duration.
\ifthenelse{\boolean{includefigures}}
{
\begin{figure}[H]
  \centering
  \includegraphics[width=1.0\linewidth]{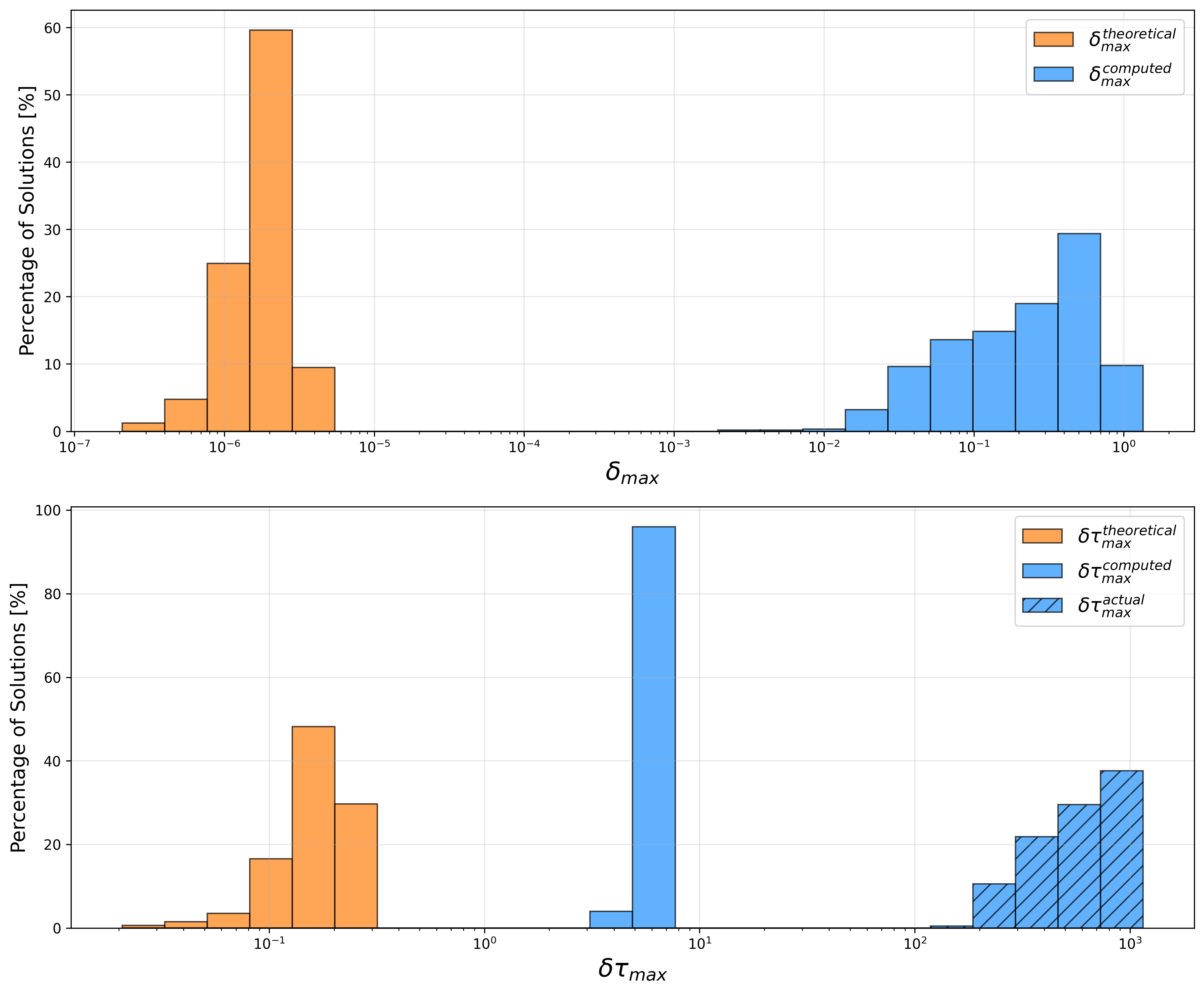}
  \caption{Distributions of the safe radius $\delta_{\max}$ (top) and maximum missed thrust duration $\delta\tau_{\max}$ (bottom) across the solution ensemble for the circular orbit case. 
  The theoretical bounds are conservative, while the computed values lie closer to the actual tolerances.}
  \label{fig:delta-tau-circular}
\end{figure}
}
{
}

While the above comparisons establish that the theoretical certificate is conservative, they do not, by themselves, explain the dependence of $\delta_{\max}$ and $\delta\tau_{\max}$ on the confidence parameter $\varepsilon$ observed in Fig.~\ref{fig:delta-tau-circular}. 
In particular, one might expect smaller values of $\varepsilon$ to produce markedly tighter safe radii $\delta_{\max}$ and correspondingly shorter admissible missed thrust durations $\delta\tau_{\max}$, reflecting stricter control of the nonlinear remainder. 
To clarify this behavior, we next revisit the structure of the $\delta_{\max}$ construction in Lemma~\ref{lem:lemma_2} and explicitly examine the role of the saturation constraint in shaping the resulting certificates.

Lemma~\ref{lem:lemma_2} shows that the safe radius $\delta(\varepsilon)$ is defined by the minimum of two mechanisms,
\begin{equation}
\label{eq:delta_min_decomposition}
\delta(\varepsilon)=
\begin{cases}
\min\!\left\{\hat{\delta}(\varepsilon), \dfrac{f_{\min}^{\bm{u}^\dagger}}{\alpha}\right\}, & \alpha>0,\\[6pt]
\hat{\delta}(\varepsilon), & \alpha=0,
\end{cases}
\end{equation}
where $\hat{\delta}(\varepsilon)$ is the \emph{nonlinearity-limited} radius induced by the Taylor remainder bound, and $f_{\min}^{\bm{u}^\dagger}/\alpha$ is a \emph{control-limited} radius required to keep the lower bound $f_{\min}^{\bm{u}^\dagger}-\alpha\delta$ in the denominator of~\eqref{eq:delta_inequality} positive.
Because the missed thrust duration certificate $\delta\tau_{\max}$ in Theorem~\ref{thm:maximum_missed_thrust_duration} is obtained by inverting the scalar envelope at the threshold $\bar{\rho}(\delta\tau_{\max})=\delta(\varepsilon)$, any regime in which $\delta(\varepsilon)$ is dominated by the saturation limit will also weaken (or eliminate) the sensitivity of $\delta\tau_{\max}$ to the confidence parameter $\varepsilon$.
Therefore, to interpret the observed dependence (or lack thereof) of $\delta_{\max}$ and $\delta\tau_{\max}$ on $\varepsilon$, we plot the empirical distribution of the saturation ratio
\begin{equation}
\label{eq:rsat_def}
    r_{\mathrm{sat}}(\varepsilon)\triangleq\frac{\delta_{\max}(\varepsilon)}{f_{\min}^{\bm{u}^\dagger}/\alpha},
\end{equation}
where $r_{\mathrm{sat}}(\varepsilon)\approx 1$ indicates that $\delta_{\max}(\varepsilon)$ is saturating at
$f_{\min}^{\bm{u}^\dagger}/\alpha$ (control-limited), whereas $r_{\mathrm{sat}}(\varepsilon)\ll 1$ indicates that
$\delta_{\max}(\varepsilon)\approx \hat{\delta}(\varepsilon)$ (nonlinearity-limited).

Figure~\ref{fig:saturation_ratio} motivates the weak dependence of the theoretical certificates on $\varepsilon$ in the circular case.
As shown in Figs.~\ref{fig:delta_max_distribution_varying_eps}--\ref{fig:delta_tau_max_distribution_varying_eps}, the distributions of $\delta_{\max}(\varepsilon)$ and $\delta\tau_{\max}(\varepsilon)$ overlap substantially across several orders of magnitude of $\varepsilon$.
Figure~\ref{fig:saturation_ratio} explains this behavior: the saturation ratio $r_{\mathrm{sat}}(\varepsilon)$ is tightly concentrated near $1$, indicating that $\delta_{\max}(\varepsilon)$ is frequently determined by the saturation limit $f_{\min}^{\bm{u}^\dagger}/\alpha$ in~\eqref{eq:delta_min_decomposition}, rather than by the $\varepsilon$-dependent nonlinearity bound $\hat{\delta}(\varepsilon)$.
Since Theorem~\ref{thm:maximum_missed_thrust_duration} computes $\delta\tau_{\max}$ by inverting the envelope at $\delta(\varepsilon)$, i.e., $\bar{\rho}(\delta\tau_{\max})=\delta(\varepsilon)$, saturation of $\delta(\varepsilon)$ directly induces saturation of $\delta\tau_{\max}$, and consequently reduces the apparent sensitivity of the certificate to $\varepsilon$.
\begin{figure}[!htb]
    \centering
    \begin{subfigure}[b]{0.48\linewidth}
        \centering
        \includegraphics[width=\linewidth]{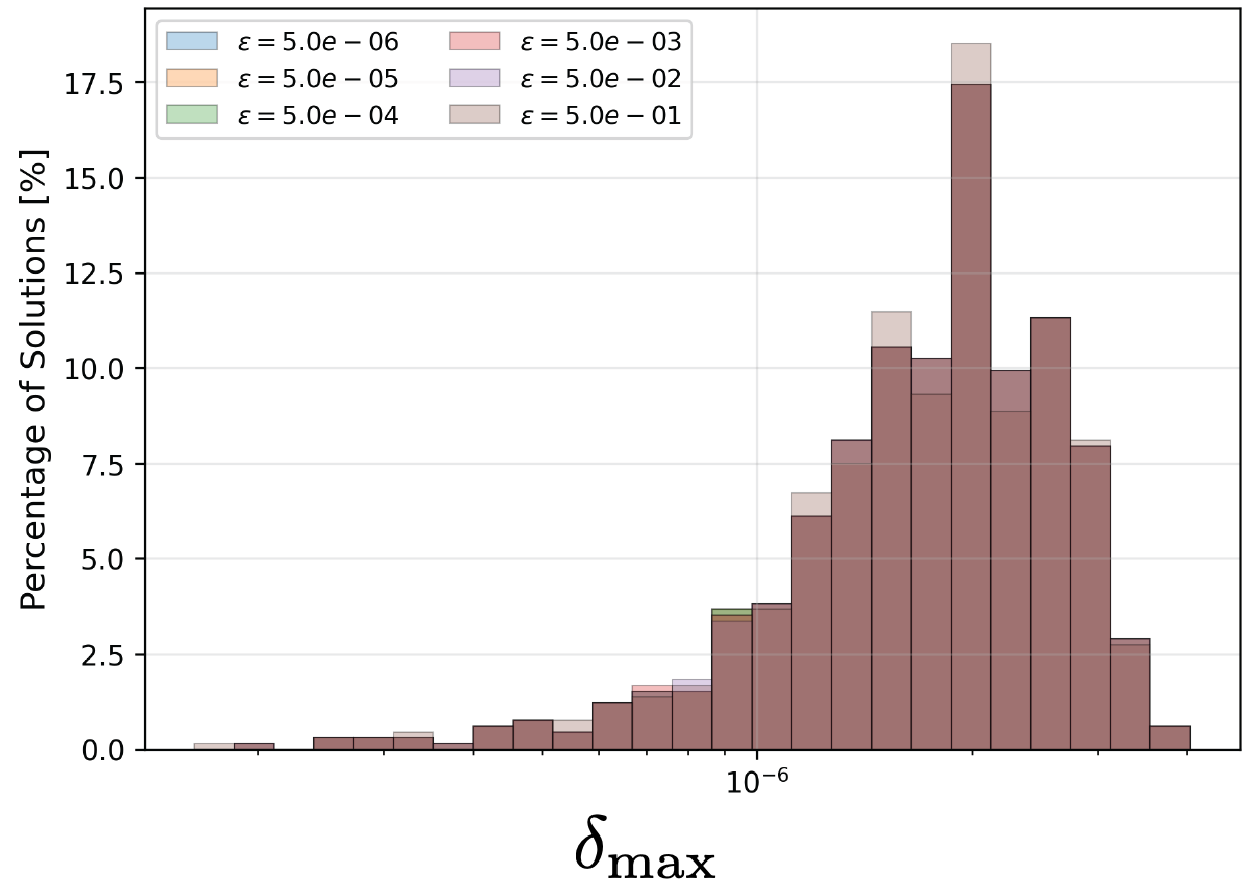}
        \caption{Theoretical $\delta_{\max}$ distributions show minimal dependence on $\epsilon$.}
        \label{fig:delta_max_distribution_varying_eps}
    \end{subfigure}
    \hfill
    \begin{subfigure}[b]{0.48\linewidth}
        \centering
        \includegraphics[width=\linewidth]{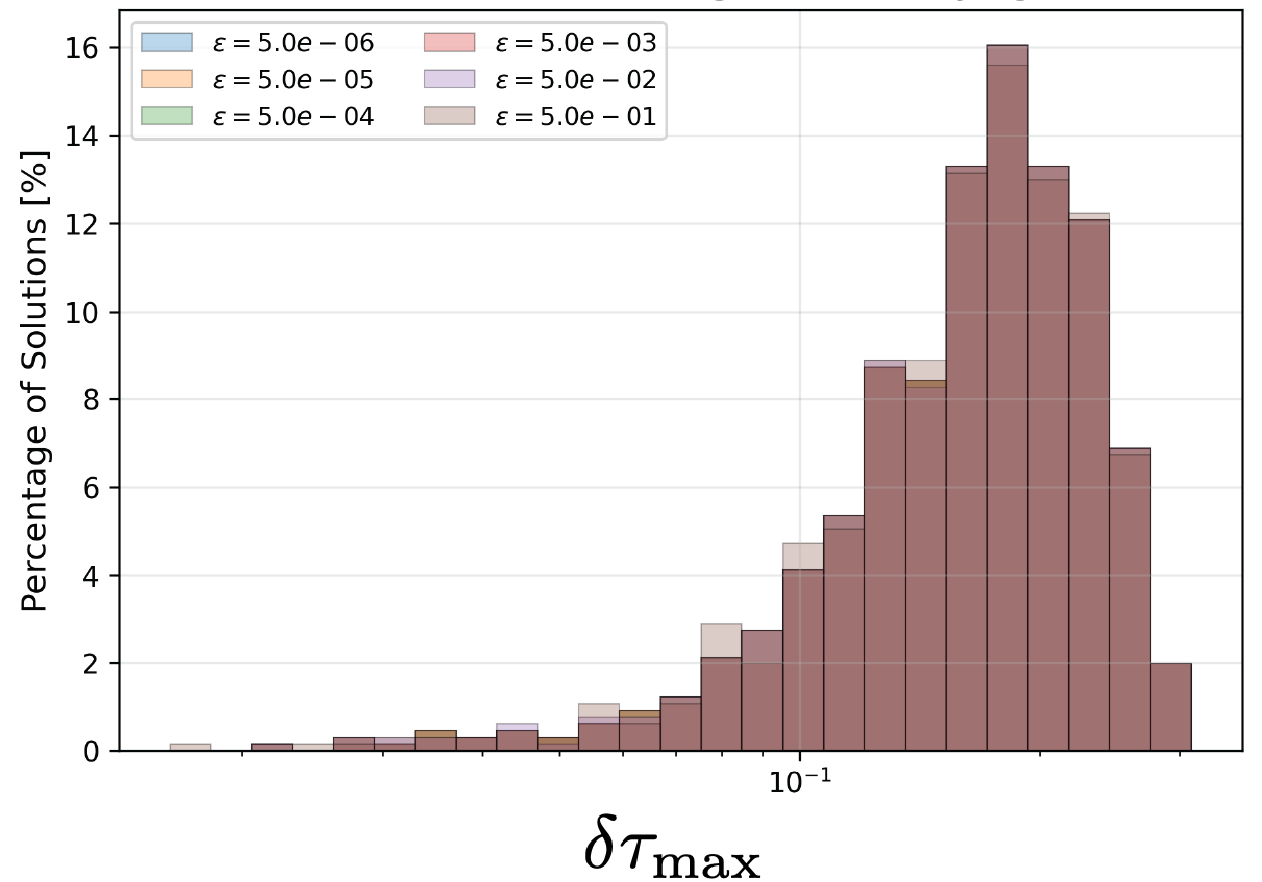}
        \caption{Theoretical $\delta\tau_{\max}$ distributions exhibit weak sensitivity to $\epsilon$.}
        \label{fig:delta_tau_max_distribution_varying_eps}
    \end{subfigure}

    \vspace{0.5em}

    \begin{subfigure}[b]{0.6\linewidth}
        \centering
        \includegraphics[width=\linewidth]{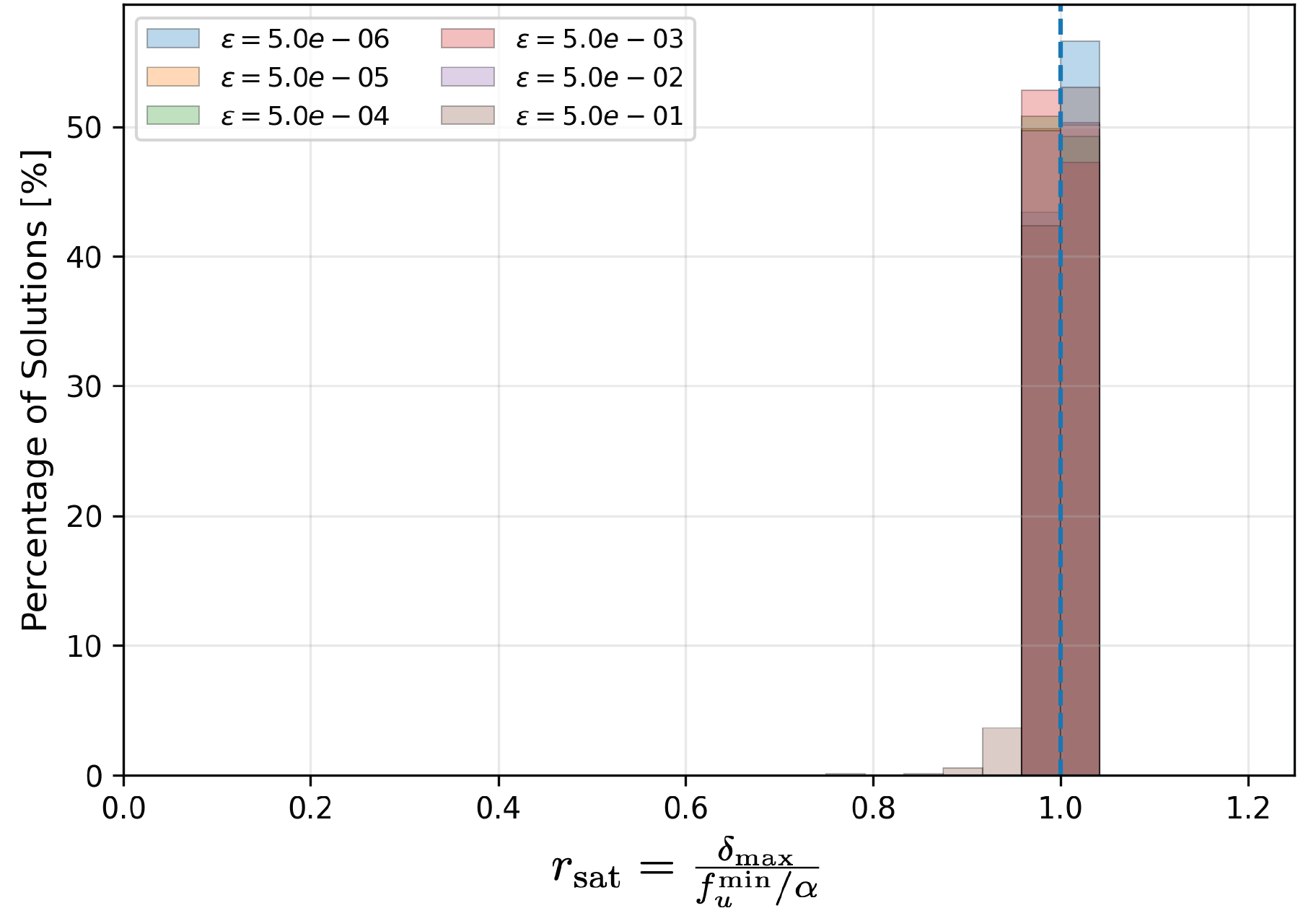}
        \caption{Saturation ratio $r_{\mathrm{sat}}$ concentrates near unity, indicating control-limited bounds.}
        \label{fig:control_saturation_limit_varying_eps}
    \end{subfigure}

    \caption{Saturation limit behavior for the circular reference orbit.}
    \label{fig:saturation_ratio}
\end{figure}

A similar comparison is performed for the eccentric reference orbit, with results summarized in Fig.~\ref{fig:delta-tau-eccentric}.
The distributions of $(\delta_{\max}^{\mathrm{theoretical}},\delta_{\max}^{\mathrm{computed}})$ and $(\delta\tau_{\max}^{\mathrm{theoretical}},\delta\tau_{\max}^{\mathrm{computed}},\delta\tau_{\max}^{\mathrm{actual}})$ for the eccentric case closely mirror those obtained in the circular case: the theoretical bounds remain the most conservative, the computed values lie substantially closer to the realized tolerances, and the ordering is preserved across the ensemble.  
As the target spacecraft moves towards the apoapsis, the level of nonlinearity decreases causing the $\delta_{\mathrm{max}}$ to shift to the right.
A similar shift is also seen with $\delta\tau_{\mathrm{max}}$.
\ifthenelse{\boolean{includefigures}}
{
\begin{figure}[!htb]
  \centering
  \includegraphics[width=1.0\linewidth]{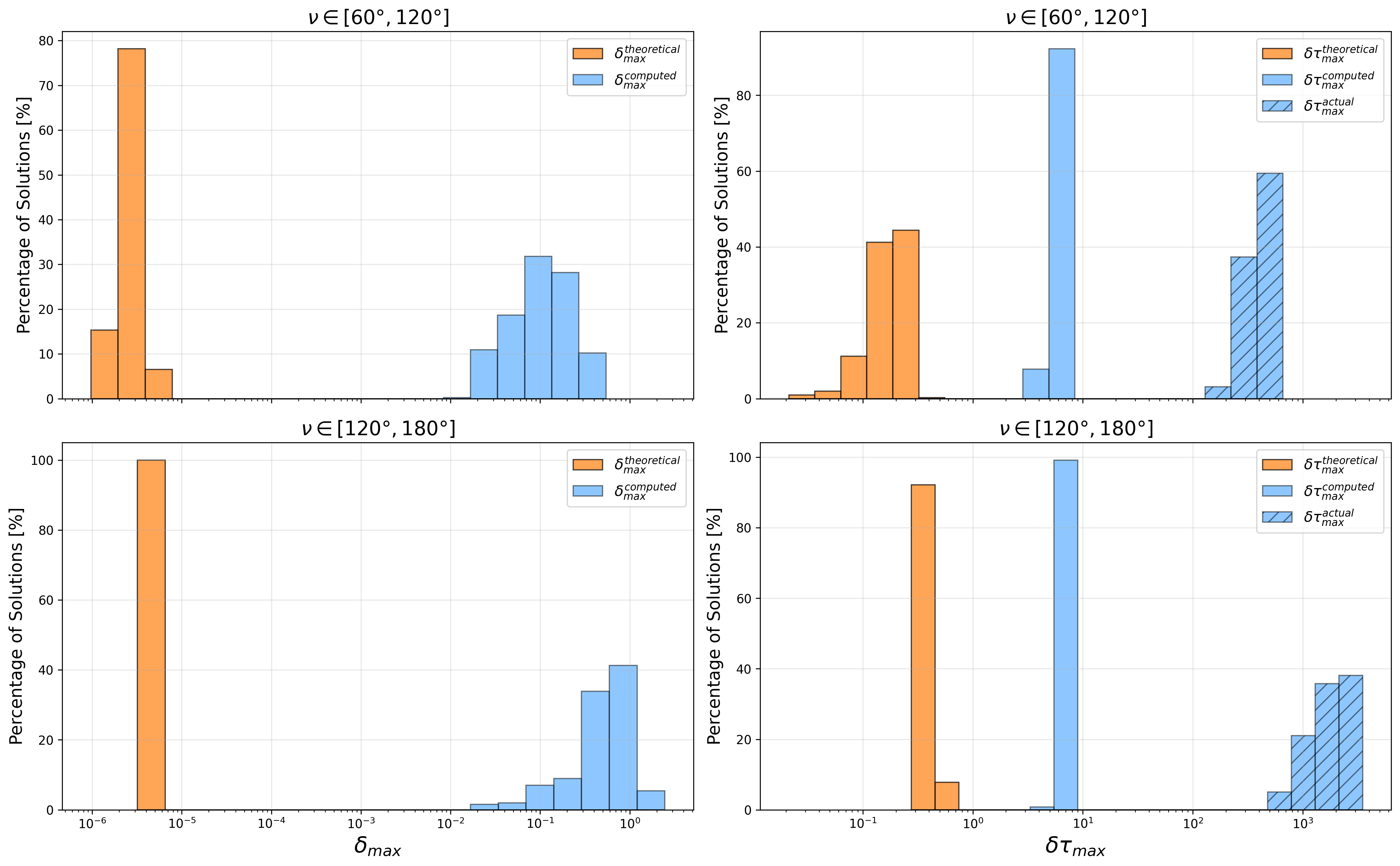}
  \caption{Distributions of the safe radius $\delta_{\max}$ (top) and maximum missed thrust duration $\delta\tau_{\max}$ (bottom) across the solution ensemble for the eccentric orbit case.}
  \label{fig:delta-tau-eccentric}
\end{figure}
}
{
}

The closed-form expressions in Theorem~\ref{thm:maximum_missed_thrust_duration} also permit an analytical and empirical assessment of the sensitivity of the certified quantities $\delta_{\max}$ and $\delta\tau_{\max}$ to the bounding parameters $(\alpha,\beta,H,f_u^{\min},f_u^{\max})$. 
We performed a correlation study across all solutions from the circular case, summarized in Fig.~\ref{fig:correlation}. 
\ifthenelse{\boolean{includefigures}}
{
\begin{figure}[!htb]
  \centering
  \includegraphics[width=0.45\textwidth]{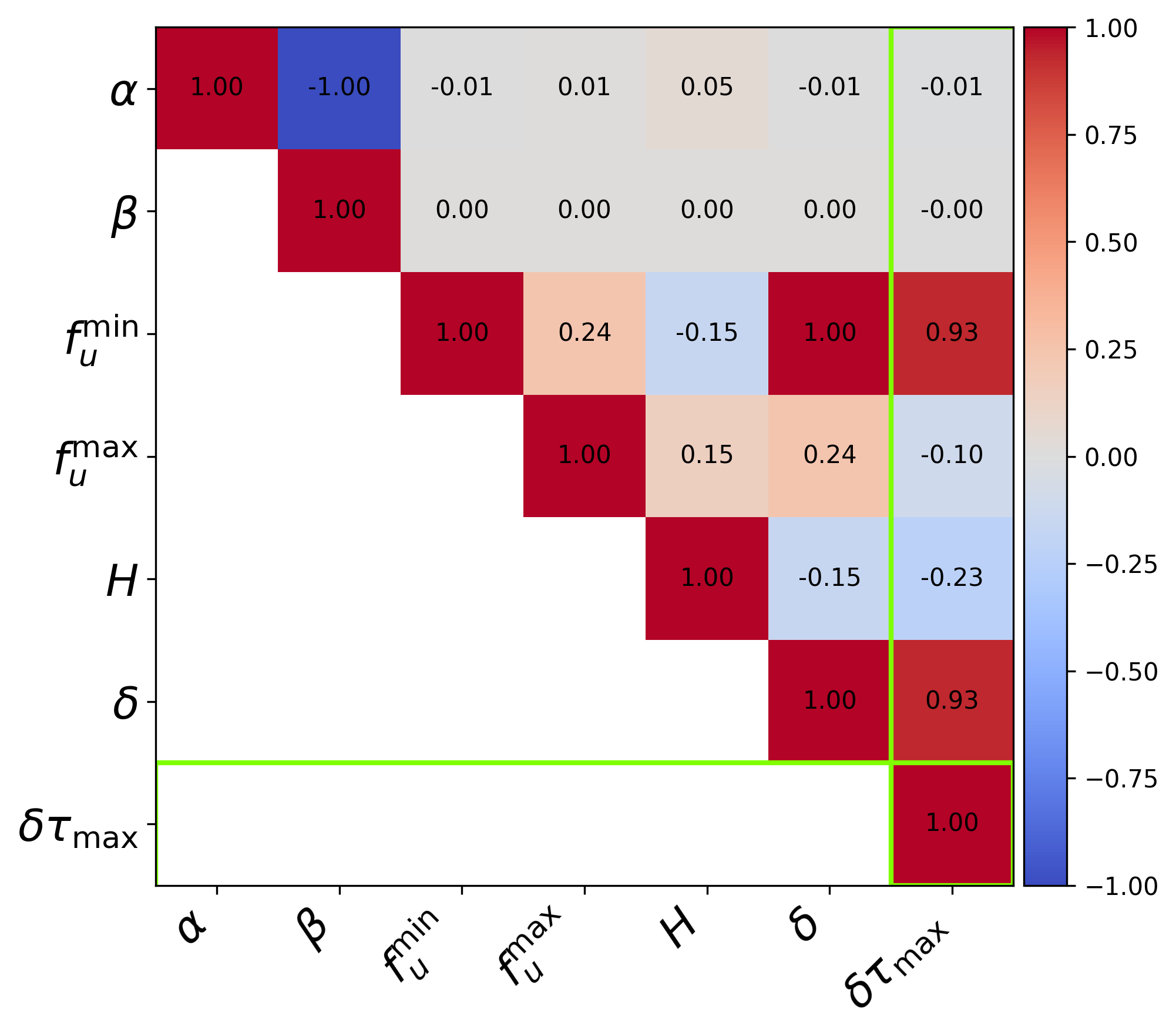}
  \caption{Correlation matrix across the ensemble. The certified quantities $\delta_{\max}$ and $\delta\tau_{\max}$ depend almost exclusively on $f_u^{\min}$, $f_u^{\max}$ and $H$ while the remaining parameters show minimal influence.}
  \label{fig:correlation}
\end{figure}
}
{
}
This joint analysis leads to several conclusions.
First, consistent with the results shown in Fig.~\ref{fig:saturation_ratio}, $\delta_{\max}$ exhibits a perfect positive correlation with $f_u^{\min}$, while showing negligible correlation with $(\alpha,\beta,H,f_u^{\max})$. 
The correlation study confirms that, for this particular case study, this single parameter dominates the variability of $\delta_{\max}$ across the ensemble.
Second, $\delta\tau_{\max}$ is almost perfectly correlated with $\delta_{\max}$ itself and, through this dependence, inherits its dominant sensitivity to $f_u^{\min}$. 
In contrast, the dependence on $f_u^{\max}$ is extremely weak, and the influence of $\alpha$ and $H$ is similarly negligible. 
Mathematically, this agrees with the structure of Theorem~\ref{thm:maximum_missed_thrust_duration}: for the parameter regime of interest, the Riccati discriminant $\Delta = \alpha^2 - 2H f_u^{\max}$  remains strictly positive with very small variation, so changes in $f_u^{\max}$ do not alter the qualitative solution branch or the dominant exponential rate. 
Once $\Delta>0$ is fixed, the leading-order dependence of the solution becomes linear in $\delta_{\max}$, and therefore in $f_u^{\min}$.
These results collectively identify the primary source of conservatism in the theoretical certificates.  
To improve the tightness of \emph{both} bounds $\delta_{\max}$ and $\delta\tau_{\max}$, the parameter that should be relaxed first is the lower control effectiveness bound $f_u^{\min}$ in Assumption~\ref{ass:control_bounds}.  
Replacing the global worst-case constant with a trajectory-dependent or interval-restricted lower bound would increase $\delta_{\max}$ directly and, through the near-linear dependence, increase $\delta\tau_{\max}$ by the same order of magnitude.  
In contrast, refining $f_u^{\max}$, $\alpha$, or $H$ would yield only marginal improvements, as confirmed by their weak empirical correlations with both certified quantities.
Finally, the structure of Fig.~\ref{fig:correlation} corroborates the monotonicity properties of the theoretical formulas: larger radius bounds reliably predict larger missed thrust duration bounds, even though the absolute values remain conservative.  
This aligns with the ordering observed in Fig.~\ref{fig:delta-tau-circular}, where the theoretical quantities underestimate the achievable tolerances but preserve the relative ranking across trajectories.

\subsection{Recovery After Missed Thrust Event}
\label{subsec:experimental results relative spacecraft motion:recovery after missed thrust event}

The bounds developed in Theorem~\ref{thm:maximum_missed_thrust_duration} provide a \emph{sufficient} condition under which the LL recovery subproblem may be accurately approximated by a first-order model over an MTE. 
The numerical results in Section~\ref{subsec:experimental results relative spacecraft motion:safe state-space radius delta and maximum missed thrust duration delta tau_max} show that the realized missed thrust durations in simulation can exceed this bound while the closed-loop controller still successfully recovers and tracks the reference. 
Importantly, exceeding the bound $\delta\tau_{\max}$ does not imply a loss of controllability or inevitable failure of the feedback law; rather, it indicates that the sufficient condition used to certify linear model validity is no longer guaranteed to hold. 
To interpret this discrepancy, we evaluate a post-hoc controllability energy metric computed from the linearization along the reference trajectory that quantifies whether the available actuator authority is, in principle, sufficient to recover from the deviation induced by missed thrust.

Recall the control-affine nonlinear dynamics $\dot{\bm{\xi}}_t = \tilde{f}(\bm{\xi}_t,\bm{u}_t)$ from Eq.~\eqref{eq:dynamics}, the reference pair $(\bm{\xi}^\dagger,\bm{u}^\dagger)$ satisfying Eq.~\eqref{eq:reference_dynamics}, and the realization pair $(\bm{\xi}^\omega,\bm{u}^\omega)$ satisfying Eq.~\eqref{eq:realization_dynamics}. 
Over the missed thrust interval $[\tau_1,\tau_2]$, the realization control is identically zero, $\bm{u}^\omega_t=0$, and the error state is defined as in Eq.~\eqref{eq:error}:
\begin{equation}
\bm{\tilde{\xi}}_t := \bm{\xi}^{\omega}_t - \bm{\xi}^{\dagger}_t .
\end{equation}
In particular, the MTE induces a nonzero deviation, which we denote by
\begin{equation}
    \bm{\tilde{\xi}}^{+} \equiv \bm{\tilde{\xi}}_{\tau_2} = \bm{\xi}^{\omega}_{\tau_2}-\bm{\xi}^{\dagger}_{\tau_2} .
    \label{eq:error_plus}
\end{equation}
The recovery phase begins at time $\tau_2$, and the central question addressed by the controllability analysis is whether, for a given available recovery horizon $T_{\mathrm{rec}}>0$, the deviation $\bm{\tilde{\xi}}^{+}$ can be driven back to zero by admissible control inputs.

To this end, we use the first-order sensitivity information defined in Eq.\eqref{eq:jacobians},
\begin{equation}
    A_t := \nabla_{\xi}\tilde f\left(\bm{\xi}^{\dagger}_t,\bm{u}^{\dagger}_t\right),
    \qquad
    B_t := \nabla_{u}\tilde f\left(\bm{\xi}^{\dagger}_t,\bm{u}^{\dagger}_t\right),
\end{equation}
and consider the time-varying linear approximation of the error dynamics about the reference over the recovery interval $T_{\mathrm{rec}}$:
\begin{equation}
    \dot{\bm{\tilde{\xi}}}_t = A_t \bm{\tilde{\xi}}_t + B_t \bm{\tilde{u}}_t,
    \qquad
    t \in [\tau_2,\tau_2+T_{\mathrm{rec}}],
    \label{eq:ltv_error_recovery}
\end{equation}
where the control deviation is $\bm{\tilde u}_t := \bm{u}^{\omega}_t-\bm{u}^{\dagger}_t$. 
During recovery, unlike the missed thrust interval, $\bm{u}^\omega_t$ is generally nonzero and is selected by the feedback controller or the LL recovery maneuver.
Let $\Phi(t,s)$ denote the state-transition matrix associated with $A_t$, i.e.,
\begin{equation}
    \frac{d}{dt}\Phi(t,s) = A_t \Phi(t,s),
    \qquad
    \Phi(s,s)=I,
    \qquad \tau_2\le s\le t\le \tau_2+T_{\mathrm{rec}}.
    \label{eq:stm_def}
\end{equation}
For a fixed recovery horizon $T_{\mathrm{rec}}>0$, the finite-horizon controllability Gramian of Eq.\eqref{eq:ltv_error_recovery} over the recovery interval is defined by
\begin{equation}
    \mathrm{W}\left(\tau_2,T_{\mathrm{rec}}\right)
    :=
    \int_{\tau_2}^{\tau_2+T_{\mathrm{rec}}}
    \Phi\left(t,\tau_2\right)\,
    B_t B_t^\top\,
    \Phi\left(t,\tau_2\right)^\top
    \,dt.
    \label{eq:gramian_continuous_recovery}
\end{equation}
The matrix $\mathrm{W}(\tau_2,T_{\mathrm{rec}})$ is generally symmetric positive semidefinite, and it is positive definite if and only if the linearized system is controllable on the finite horizon $[\tau_2,\tau_2+T_{\mathrm{rec}}]$.

Given the realized post-event deviation $\bm{\tilde{\xi}}^{+}$, we consider the finite-horizon minimum-energy recovery problem for the linear model Eq.~\eqref{eq:ltv_error_recovery}:
\begin{equation}
    E_{\min}\left(\bm{\tilde{\xi}}^{+};\tau_2,T_{\mathrm{rec}}\right)
    :=
    \min_{\bm{\tilde u}(\cdot)}
    \int_{\tau_2}^{\tau_2+T_{\mathrm{rec}}}
    \|\bm{\tilde u}_t\|_2^2\,dt
\end{equation}
subject to Eq.~\eqref{eq:ltv_error_recovery} and the boundary conditions
\begin{equation}
    \bm{\tilde{\xi}}_{\tau_2} = \bm{\tilde{\xi}}^{+},
    \qquad
    \bm{\tilde{\xi}}_{\tau_2+T_{\mathrm{rec}}} = 0.
    \label{eq:recovery_bc}
\end{equation}
Whenever $\mathrm{W}(\tau_2,T_{\mathrm{rec}})$ is invertible, this minimum energy admits the closed-form expression
\begin{equation}
    E_{\min}\left(\bm{\tilde{\xi}}^{+};\tau_2,T_{\mathrm{rec}}\right)
    =
    \bm{\tilde{\xi}}^{+ ^\top}
    W\left(\tau_2,T_{\mathrm{rec}}\right)^{-1}
    \bm{\tilde{\xi}}^{+}.
    \label{eq:emin_recovery}
\end{equation}

To compare this minimum required energy with available control authority, we impose an elementwise bound on the recovery deviation control,
\begin{equation}
    |\tilde u_{t,i}| \le \bar{u}_i \qquad \forall t \in [\tau_2,\tau_2+T_{\mathrm{rec}}],  i=1,\dots,m,
    \label{eq:u_bound}
\end{equation}
for some $\bar{u}\in\reals^m$. Then, by Cauchy-Schwarz,
\begin{equation}
    \int_{\tau_2}^{\tau_2+T_{\mathrm{rec}}}
    \|\bm{\tilde u}_t\|_2^2\,dt
    \le
    \int_{\tau_2}^{\tau_2+T_{\mathrm{rec}}}
    \|\bar{u}\|_2^2\,dt
    =
    T_{\mathrm{rec}}\,\|\bar{u}\|_2^2
    =:
    E_{\mathrm{ava}}\left(T_{\mathrm{rec}}\right).
    \label{eq:eavail_recovery}
\end{equation}
This motivates the controllability energy feasibility ratio
\begin{equation}
    r_\mathrm{e}
    :=
    \frac{E_{\mathrm{ava}}\left(T_{\mathrm{rec}}\right)}
    {E_{\min}\left(\bm{\tilde{\xi}}^{+};\tau_2,T_{\mathrm{rec}}\right)}.
    \label{eq:rhoE_recovery}
\end{equation}
A value $r_\mathrm{e}\ge 1$ indicates that, within the linearized recovery model Eq.~\eqref{eq:ltv_error_recovery} and given the bound Eq.~\eqref{eq:u_bound}, there exists a recovery input with sufficient energy to drive the realized deviation $\bm{\tilde{\xi}}^{+}$ back to the reference within the available recovery horizon. 
Conversely, $r_\mathrm{e}<1$ indicates energetic infeasibility under the chosen bounds and horizon.

\ifthenelse{\boolean{includefigures}}
{
    \begin{figure}[!htb]
        \centering
        \begin{subfigure}[t]{0.45\textwidth}
            \centering
            \includegraphics[keepaspectratio, width=\textwidth]{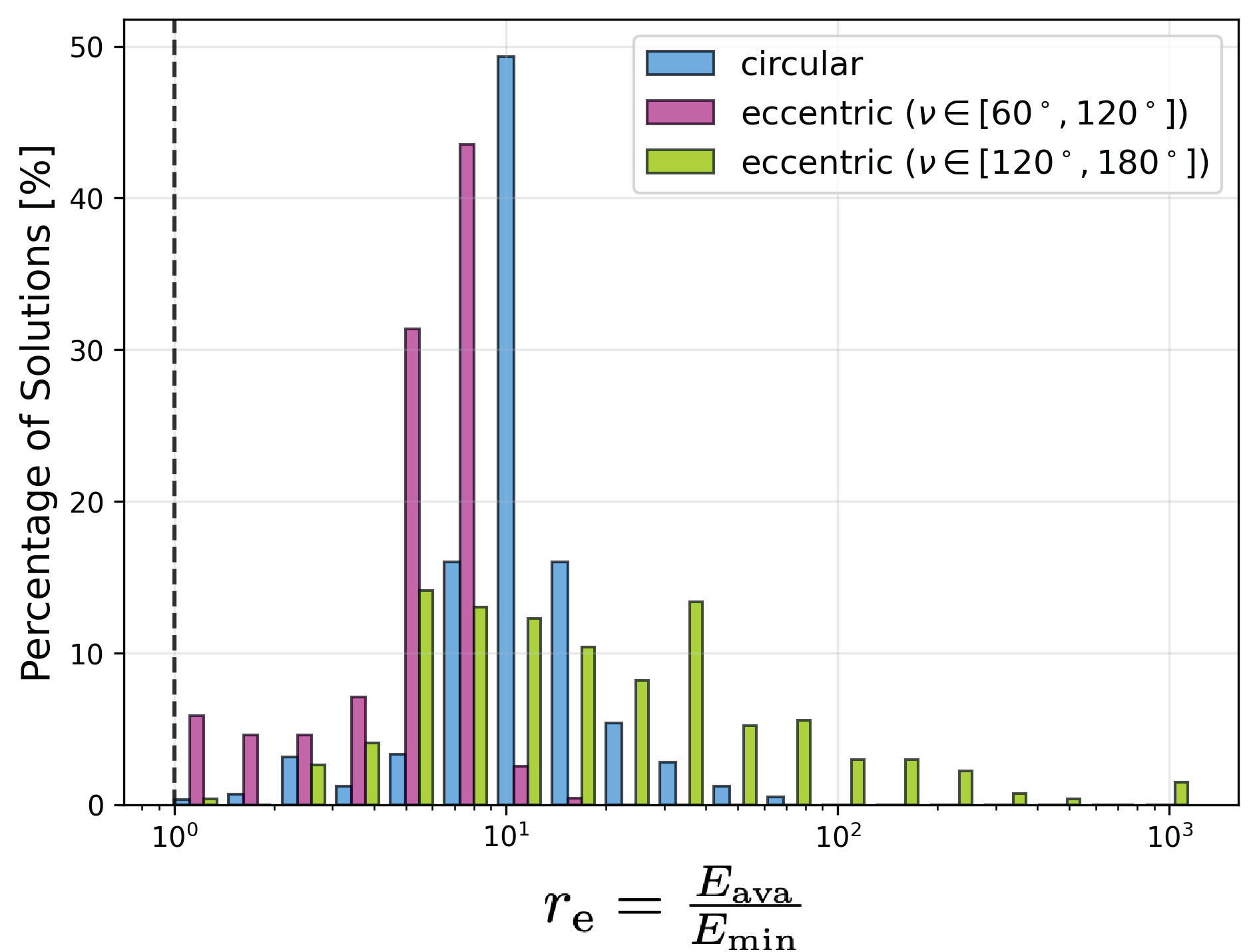}
            \caption{
            Distribution of the controllability energy ratio $r_\mathrm{e}$ with the dashed line indicating the feasibility threshold $r_\mathrm{e} = 1$
            }
            \label{fig:ratio_energy}
        \end{subfigure}
        \hfill
        \begin{subfigure}[t]{0.45\textwidth}
            \centering
            \includegraphics[keepaspectratio, width=\textwidth]{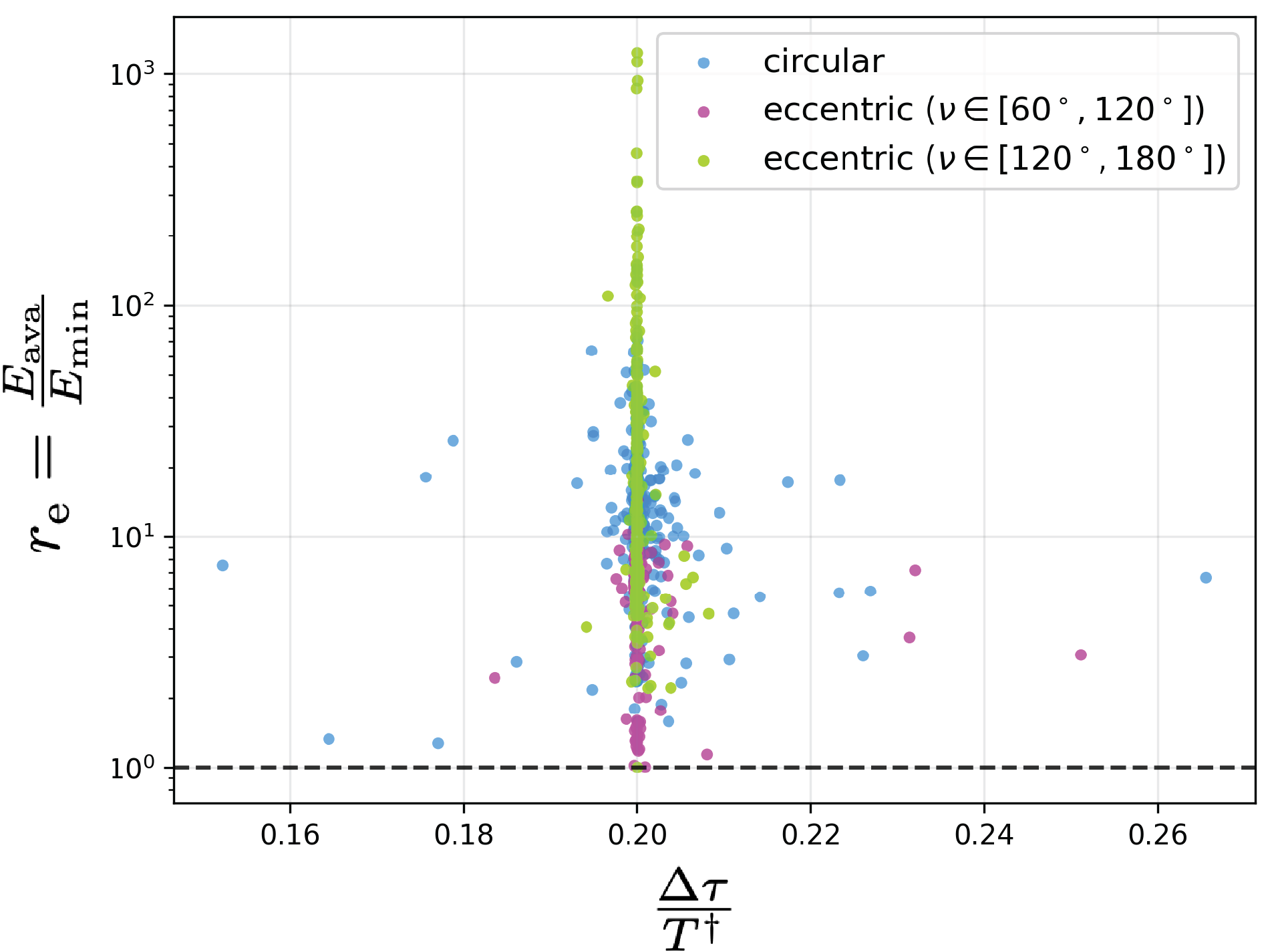}
            \caption{
               $r_\mathrm{e}$ versus normalized missed thrust duration
            }
            \label{fig:ratio_energy_vs_missed_thrust_duration}
        \end{subfigure}
        \caption{
        controllability energy diagnostics for circular and eccentric cases
        }
        \label{fig:controllability_analysis}
    \end{figure}
}
{
    % Do nothing
}
Figure~\ref{fig:controllability_analysis} summarizes the controllability energy analysis for the eccentric-orbit experiments: Fig.~\ref{fig:ratio_energy} reports the distribution of $r_\mathrm{e}$ across feasible solutions, and Fig.~\ref{fig:ratio_energy_vs_missed_thrust_duration} shows $r_\mathrm{e}$ as a function of the missed thrust duration. 
To compare across eccentric-orbit cases with differing overall horizons, missed thrust durations are normalized by the total mission time $T^\omega$ computed from the simulated time history.
The results indicate that $r_\mathrm{e}$ is frequently significantly larger than unity, implying substantial energetic margin for recovery even when $\Delta\tau$ exceeds the certificate $\delta\tau_{\max}$ from Theorem~\ref{thm:maximum_missed_thrust_duration}. This supports the interpretation that $\delta\tau_{\max}$ is a sufficient condition for first-order model validity, but not a necessary condition for closed-loop recoverability: recovery may remain feasible whenever the linearized dynamics over the available horizon are practically controllable and sufficient actuation authority exists, as captured by $r_\mathrm{e}\ge 1$.
Finally, we note that $r_\mathrm{e}$ is a post-hoc diagnostic computed from the linearization Eq.~\eqref{eq:jacobians} along the reference and evaluated over a finite recovery horizon. 
Its value depends on the realized deviation $\bm{\tilde{\xi}}^{+}$, the horizon $T_{\mathrm{rec}}$, and the actuation bounds Eq.~\eqref{eq:u_bound}, and it does not constitute a global guarantee for the original nonlinear dynamics Eq.~\eqref{eq:dynamics}. 
Rather, it provides a quantitative explanation for why realized missed thrust durations may exceed the first-order certificate while closed-loop recovery remains successful in simulation.

%% file: sections/conclusion.tex
\section{Conclusion}
\label{sec:conclusion}

In this work, a bi-level optimal control framework was introduced for missed thrust design in low-thrust spacecraft missions.
By posing the nominal trajectory design as an upper-level (UL) problem and embedding each potential recovery maneuver as a lower-level (LL) linear–quadratic regulator (LQR) subproblem, the framework departs fundamentally from the traditional operational paradigm in which a reference trajectory is designed first and recovery controllers are applied only post hoc. 
Instead, the proposed bi-level formulation explicitly accounts for recovery behavior during the trajectory design process itself, ensuring that the reference solution is intentionally placed in a regime where recovery can be effective in the presence of missed thrust events. 
The resulting single-level reformulation enforces the LL optimality conditions exactly, yielding a continuously differentiable, sparsity-preserving nonlinear program that can be solved efficiently using standard numerical optimization techniques while retaining full fidelity to the recovery dynamics.
An analytical robustness certificate was derived, providing a sufficient bound on the maximum missed thrust duration under which the LQR model in the LL remains appropriate.
The certificate depends only on quantities that can be precomputed along the reference trajectory enabling rapid assessment of robustness over large ensembles of possible missed thrust realizations.
This yields ex ante guarantees on recoverability without reliance on Monte Carlo simulation or conservative margining practices, and integrates naturally with direct transcription tools widely used in modern trajectory design.
This analytical capability can potentially enable systematic placement of coast arcs, allocation of thrust segments, and characterization of sensitivity to control outages within existing trajectory design pipelines.
The numerical study demonstrated that the theoretical bounds preserve relative ordering between the safe radius and the maximum allowable missed thrust duration across a broad ensemble of scenarios, with dominant sensitivity arising from the lower control-effectiveness bound. 
While these bounds provide information about the appropriateness of the LL linear model, they do not characterize the feasibility of the closed-loop recovery phase. 
To address this distinction, a finite-horizon controllability energy analysis was performed as a post-hoc diagnostic on the recovery trajectories.
By evaluating the minimum control energy required to drive the deviation between the reference and the realization solutions back to the reference over the available recovery horizon, and comparing it to the energy available under actuator bounds, this analysis provides a quantitative measure of energetic recoverability within the linear model. 
The results show that successful recoveries frequently possess substantial energetic margin, even when the realized missed thrust duration exceeds the bound provided by the certificate. 
This confirms that the derived bound should be interpreted as a sufficient condition for first-order model appropriateness rather than a necessary condition for closed-loop recoverability, and helps explain why LQR-based recovery can remain effective beyond the certified regime.
Several extensions naturally follow from the present work.
Rather than uniform worst-case bounds, future work may employ trajectory dependent, interval based, or state dependent curvature bounds potentially supported by validated reachability analysis or probabilistic confidence intervals to significantly reduce conservatism in the certificates.
Additional research directions could also look into generalizing the LL model beyond the LQR setting, developing stochastic framework which integrate the missed thrust likelihood with probabilistic guarantees, and extending the proposed approach to problems involving stronger nonlinearities, such as long-horizon low-thrust transfers, multi-body dynamical environments, or high-fidelity ephemeris-based mission models.
Together, these efforts aim to further enhance the rigor, applicability, and predictive accuracy of robustness analysis for low-thrust spaceflight.

%% file: sections/acknowledgement.tex
\section*{Acknowledgments}
The simulations presented in this article were performed on computational resources managed and supported by Princeton Research Computing, a consortium of groups including the Princeton Institute for Computational Science and Engineering (PICSciE) and the Office of Information Technology's High Performance Computing Center and Visualization Laboratory at Princeton University.

%% file: sections/appendix.tex
\section{Appendix}
\label{sec:appendix}

% ===========================================================================
\subsection{Second-Order Taylor Expansion of the Vector Field}
\label{app:taylor_expansion}

For any fixed $t$, a Taylor expansion of the vector field with integral remainder gives
\begin{equation*}
    \begin{split}
    f \bigl(t, \xi^{\dagger} + e, u^{\dagger} + u_e\bigr) = f \bigl(t, \xi^{\dagger}, u^{\dagger}\bigr)
    & +\nabla_{ \xi}f \bigl(t, \xi^{\dagger}, u^{\dagger} \bigr) e + \nabla_{u}f \bigl(t, \xi^{\dagger}, u^{\dagger} \bigr) u_e \\
    & +r(t, e) + 2r(t, e, u_e) + r(t, u_e)\mathrm d\theta .
\end{split}
\end{equation*}
where,
\begin{equation*}
\begin{split}
    r(t, e) = &\int_{0}^{1}(1-\theta)\Bigl[\nabla_{ \xi\xi}^{2}f \bigl(t,\xi^{\dagger} +\theta e,u^{\dagger} +\theta u_e\bigr)[e,e] \Bigr]\mathrm d\theta \\
    r(t, e, u_e) = &\int_{0}^{1} 2\nabla_{ \xi u}^{2}f \bigl(t,\xi^{\dagger} +\theta e,u^{\dagger} +\theta u_e\bigr)[e,u_e] \Bigr]\mathrm d\theta \\
    r(t, u_e) = &\int_{0}^{1} \nabla_{uu}^{2}f \bigl(t,\xi^{\dagger} +\theta e,u^{\dagger} +\theta u_e\bigr)[u_e,u_e] \Bigr]\mathrm d\theta.
\end{split}
\end{equation*}
We assume that the dynamics are control-affine, which allows us to drop the higher-order control terms.
We can bound the norm of the remainder term $\|r(t, e)\|$ as follows:
\begin{align}\label{eq:remainder_bound}
  \bigl\| r(t,e) \bigr\|
  &= \left\|
       \int_{0}^{1} (1-\theta)
              \nabla_{ \xi}^{2}f
                  \bigl(t,\xi^{\dagger}+\theta e,0\bigr)
              [e,e]
          d\theta
     \right\| \nonumber \\
  &\le
     \int_{0}^{1} (1-\theta)
            \bigl\|
              \nabla_{ \xi}^{2}f
                  \bigl(t,\xi^{\dagger}+\theta e,0\bigr)
              [e,e]
            \bigr\|
     d\theta \qquad\text{(Triangle Inequality)} \nonumber \\
  &\le
     \int_{0}^{1} (1-\theta)
            \bigl\|
              \nabla_{ \xi}^{2}f
                  \bigl(t,\xi^{\dagger}+\theta e,0\bigr)
            \bigr\|
            \|e\|\|e\|
     d\theta \qquad\text{(Submultiplicativity)} \nonumber \\
  &\le
     H\|e\|^{2}
     \int_{0}^{1} (1-\theta)d\theta
     \qquad\text{(Assumption \ref{ass:second_order_system_bounds})} \nonumber \\
  &= \frac{H}{2}\|e\|^{2}.
\end{align}

% ===========================================================================
\subsection{Reverse Triangle Inequality}
\label{app:reverse_triangle_inequality}

For any vectors $x,y\in\mathbb R^{n}$ and any norm $\|\cdot\|$, the usual \emph{triangle} inequality gives $\|x\| = \| (x + y) - y\| \le \|x + y\| + \|y\|, \|y\| = \| (x + y) - x\| \le \|x + y\| + \|x\|.$
Combining the two inequalities yields
\begin{equation*}
    -\|x + y\| \le \|x\| - \|y\| \le \|x + y\|.
\end{equation*}
Hence,
\begin{equation} \label{eq:reverse_triangle_inequality}
    \bigl|\|x\| - \|y\|\bigr| \le \|x + y\|.
\end{equation}
which is often called the reverse triangle inequality.

%% file: references.bib
@article{mangasarian_fritz_1967,
    title = {The Fritz John necessary optimality conditions in the presence of equality and inequality constraints},
    journal = {Journal of Mathematical Analysis and Applications},
    volume = {17},
    number = {1},
    pages = {37-47},
    year = {1967},
    issn = {0022-247X},
    doi = {https://doi.org/10.1016/0022-247X(67)90163-1},
    url = {https://www.sciencedirect.com/science/article/pii/0022247X67901631},
    author = {O.L Mangasarian and S Fromovitz}
}

@article{tschauner_1965_rendezvous,
    title = {Rendezvous zu einem in elliptischer Bahn umlaufenden Ziel},
    journal = {Acta Astronautica},
    volume = {11},
    number = {2},
    pages = {104-109},
    year = {1965},
    author = {Tschauner, J and Hempel, P}
}

@software{jax_2018_github,
  author = {James Bradbury and Roy Frostig and Peter Hawkins and Matthew James Johnson and Chris Leary and Dougal Maclaurin and George Necula and Adam Paszke and Jake Vander{P}las and Skye Wanderman-{M}ilne and Qiao Zhang},
  title = {{JAX}: composable transformations of {P}ython+{N}um{P}y programs},
  url = {http://github.com/jax-ml/jax},
  version = {0.3.13},
  year = {2018},
}

@inproceedings{oh_analysis_2008,
    author    = {Oh, David and Landau, Damon and Randolph, Tom and Timmerman, Paul and Chase, James and Sims, John and Kowalkowski, Theresa},
    title     = {Analysis of System Margins on Deep Space Missions Using Solar Electric Propulsion},
    booktitle = {44th AIAA/ASME/SAE/ASEE Joint Propulsion Conference \& Exhibit},
    year      = {2008},
    note      = {AIAA Paper 2008-5286},
    doi       = {10.2514/6.2008-5286}
}

@conference{olympio_designing_2010,
    author    = {Olympio, Joris T.},
    title     = {Designing robust low-thrust interplanetary trajectories subject to one temporary engine failure},
    booktitle = {20th AIAA/AAS Space Flight Mechanics Meeting},
    year      = {2010},
    note      = {AIAA Paper 2019-0171}
}

@inproceedings{imken_modeling_2018,
    author    = {Imken, Thomas and Randolph, Travis and DiNicola, Michael and Nicholas, Austin},
    title     = {Modeling Spacecraft Safe Mode Events},
    booktitle = {Proceedings of the 2018 IEEE Aerospace Conference},
    year      = {2018},
    pages     = {1-13},
    publisher = {Institute of Electrical and Electronics Engineers},
    address   = {Piscataway, NJ},
    doi       = {10.1109/AERO.2018.8396383}
}

@conference{miller_low-thrust_2019,
    author    = {Miller, David and Linares, Richard},
    title     = {Low-Thrust Optimal Control via Reinforcement Learning},
    booktitle = {29th AIAA/AAS Space Flight Mechanics Meeting},
    year      = {2019},
    note      = {AAS Paper 19-560}
}

@conference{madni_missed_2020,
    author    = {Madni, Ashley A. and Hart, William and Imken, Travis and Oh, David Y. and Snyder, Steve},
    title     = {Missed Thrust Requirements for Psyche Mission},
    booktitle = {AIAA Propulsion and Energy 2020 Forum},
    year      = {2020},
    note      = {AIAA Paper 2020-3608}
}

@conference{mccarty_missed_2020,
    author    = {McCarty, Steven L. and Grebow, Daniel J.},
    title     = {Missed Thrust Analysis and Design For Low Thrust Cislunar Transfers},
    booktitle = {AAS/AIAA Astrodynamics Specialist Conference},
    year      = {2020},
    note      = {AAS Paper 20-535}
}

@article{alizadeh_sensitivity_2013,
    author    = {Alizadeh, Iman and Villac, Benjamin F.},
    title     = {Sensitivity Reduction and Lifetime Extension of Power-Limited Low-Thrust Trajectories},
    journal   = {Journal of Guidance, Control, and Dynamics},
    year      = {2013},
    volume    = {36},
    number    = {1},
    pages     = {218-228},
    doi       = {10.2514/1.56480}
}

@article{rubinsztejn_neural_2020,
    author    = {Rubinsztejn, Ari and Sood, Rohan and Laipert, Frank E.},
    title     = {Neural Network Optimal Control in Astrodynamics: Application to the Missed Thrust Problem},
    journal   = {Acta Astronautica},
    year      = {2020},
    volume    = {176},
    number    = {},
    pages     = {192-203},
    doi       = {10.1016/j.actaastro.2020.05.027}
}

@article{venigalla_low_2020,
    author    = {Venigalla, Chandrakanth and Englander, Jacob A. and Scheeres, Daniel J.},
    title     = {Multi-Objective Low-Thrust Trajectory Optimization with Robustness to Missed Thrust Events},
    journal   = {Journal of Guidance, Control, and Dynamics},
    year      = {2020},
    volume    = {45},
    number    = {7},
    pages     = {1255-1268},
    doi       = {10.2514/1.G006056}
}

@article{izzo_real-time_2021,
    author    = {Izzo, Dario and Öztürk, Ekin},
    title     = {Real-Time Guidance for Low-Thrust Transfers Using Deep Neural Networks},
    journal   = {Journal of Guidance, Control, and Dynamics},
    year      = {2021},
    volume    = {44},
    number    = {2},
    pages     = {315-327},
    doi       = {10.2514/1.G005254}
}

@article{zavoli_reinforcement_2021,
    author    = {Zavoli, Alessandro and Federici, Lorenzo},
    title     = {Reinforcement Learning for Robust Trajectory Design of Interplanetary Missions},
    journal   = {Journal of Guidance, Control, and Dynamics},
    year      = {2021},
    volume    = {44},
    number    = {8},
    pages     = {1440-1453},
    doi       = {10.2514/1.G005794}
}

@conference{laipert_monte_2018,
    author    = {Laipert, Frank E. and Imken, Travis},
    title     = {A Monte Carlo Approach to Measuring Trajectory Performance Subject to Missed Thrust},
    booktitle = {28th AIAA/AAS Space Flight Mechanics Meeting},
    year      = {2018},
    note      = {AIAA Paper 2018-0966}
}

@article{rayman_coupling_2007,
    author    = {Rayman, Marc D. and Fraschetti, Thomas C. and Raymond, Carol A. and Russell, Christopher T.},
    title     = {Coupling of System Resource Margins through the Use of Electric Propulsion: Implications in Preparing for the Dawn Mission to Ceres and Vesta},
    journal   = {Acta Astronautica},
    year      = {2007},
    volume    = {60},
    number    = {10},
    pages     = {930-938},
    doi       = {10.1016/j.actaastro.2006.11.012}
}

@article{laipert_automated_2015,
    author    = {Laipert, Frank E. and Longuski, James M.},
    title     = {Automated Missed-Thrust Propellant Margin Analysis for Low-Thrust Trajectories},
    journal   = {Journal of Spacecraft and Rockets},
    year      = {2015},
    volume    = {52},
    number    = {4},
    pages     = {1135-1143},
    doi       = {10.2514/1.A33264}
}

@conference{yi_nonlinear_2020,
    author    = {Yi, Zeji and Cao, Zhefeng and Theodorou, Evangelos and Chen, Yongxin},
    booktitle = {2020 American Control Conference (ACC)}, 
    title     = {Nonlinear Covariance Control via Differential Dynamic Programming}, 
    year      = {2020},
    pages     = {3571-3576},
    doi       = {10.23919/ACC45564.2020.9147531}
}

@article{okamoto_optimal_2018,
    author    = {Okamoto, Kazuhide and Goldshtein, Maxim and Tsiotras, Panagiotis},
    journal   = {IEEE Control Systems Letters}, 
    title     = {Optimal Covariance Control for Stochastic Systems Under Chance Constraints}, 
    year      = {2018},
    volume    = {2},
    number    = {2},
    pages     = {266-271},
    doi       = {10.1109/LCSYS.2018.2826038}
}

@article{ozaki_stochastic_2018,
    author    = {Ozaki, Naoya and Campagnola, Stefano and Funase, Ryu and Yam, Chit H.},
    title     = {Stochastic Differential Dynamic Programming with Unscented Transform for Low-Thrust Trajectory Design},
    journal   = {Journal of Guidance, Control, and Dynamics},
    year      = {2018},
    volume    = {41},
    number    = {2},
    pages     = {377-387},
    doi       = {10.2514/1.G002367}
}

@article{greco_direct_2020,
    author    = {Greco, Cristian and {Di Carlo}, Marilena and Vasile, Massimiliano and Epenoy, Richard},
    title     = {Direct multiple shooting transcription with polynomial algebra for optimal control problems under uncertainty},
    journal   = {Acta Astronautica},
    year      = {2020},
    volume    = {170},
    pages     = {224-234},
    doi       = {10.1016/j.actaastro.2019.12.010}
}

@article{ozaki_tube_2020,
    author    = {Ozaki, Naoya and Campagnola, Stefano and Funase, Ryu},
    title     = {Tube Stochastic Optimal Control for Nonlinear Constrained Trajectory Optimization Problems},
    journal   = {Journal of Guidance, Control, and Dynamics},
    year      = {2020},
    volume    = {43},
    number    = {4},
    pages     = {645-655},
    doi       = {10.2514/1.G004363}
}

@article{oguri_robust_2021,
    author    = {Oguri, Kenshiro and McMahon, Jay W.},
    title     = {Robust Spacecraft Guidance Around Small Bodies Under Uncertainty: Stochastic Optimal Control Approach},
    journal   = {Journal of Guidance, Control, and Dynamics},
    volume    = {44},
    number    = {7},
    pages     = {1295-1313},
    year      = {2021},
    doi       = {10.2514/1.G005426}
}

@article{greco_robust_2022,
    author    = {Greco, Cristian and Campagnola, Stefano and Vasile, Massimiliano},
    title     = {Robust Space Trajectory Design Using Belief Optimal Control},
    journal   = {Journal of Guidance, Control, and Dynamics},
    year      = {2022},
    volume    = {45},
    number    = {6},
    pages     = {1060-1077},
    doi       = {10.2514/1.G005704}
}

@article{oguri_stochastic_2022,
    author    = {Oguri, Kenshiro and McMahon, Jay W.},
    title     = {Stochastic Primer Vector for Robust Low-Thrust Trajectory Design Under Uncertainty},
    journal   = {Journal of Guidance, Control, and Dynamics},
    volume    = {45},
    number    = {1},
    pages     = {84-102},
    year      = {2022},
    doi       = {10.2514/1.G005970}
}

@inproceedings{papavassilopoulos_algorithms_1982,
    author    = {Papavassilopoulos, George P.},
    booktitle = {IEEE Conference on Decision and Control}, 
    title     = {Algorithms for static stackelberg games with linear costs and polyhedra constraints}, 
    year      = {1982},
    volume    = {},
    number    = {},
    pages     = {647-652},
    doi       = {10.1109/CDC.1982.268221}
}

@techreport{bi_exact_1989,
    author      = {Bi, Z and Calamai, Paul H. and Conn, A},
    title       = {An exact penalty function approach for the linear bilevel programming problem},
    institution = {Department of Systems Design Engineering, University of Waterloo, Waterloo, ON, Canada},
    number      = {167-0-310789},
    type        = {Technical Report},
    url         = {},
    year        = {1989},
    month       = {}
}

@techreport{bi_exact_1991,
    author      = {Bi, Z and Calamai, P and Conn, A},
    title       = {An exact penalty function approach for the linear bilevel programming problem},
    institution = {Department of Systems Design Engineering, University of Waterloo, Waterloo, ON, Canada},
    number      = {180-0-170591},
    type        = {Technical Report},
    url         = {},
    year        = {1991},
    month       = {}
}

@article{aiyoshi_hierarchical_1981,
    author    = {Aiyoshi, Eitaro and Shimizu, Kiyotaka},
    journal   = {IEEE Transactions on Systems, Man, and Cybernetics}, 
    title     = {Hierarchical Decentralized Systems and Its New Solution by a Barrier Method}, 
    year      = {1981},
    volume    = {11},
    number    = {6},
    pages     = {444-449},
    doi       = {10.1109/TSMC.1981.4308712}
}

@article{bialas_two-level_1982,
    author    = {Bialas, Wayne F. and Karwan, Mark H.},
    journal   = {IEEE Transactions on Automatic Control}, 
    title     = {On two-level optimization}, 
    year      = {1982},
    volume    = {27},
    number    = {1},
    pages     = {211-214},
    doi       = {10.1109/TAC.1982.1102880}
}

@article{candler_linear_1982,
    author    = {Candler, Wilfred and Townsley, Robert},
    title     = {A linear two-level programming problem},
    year      = {1982},
    journal   = {Computers \& Operations Research},
    volume    = {9},
    number    = {1},
    pages     = {59-76},
    doi       = {10.1016/0305-0548(82)90006-5},
}

@article{aiyoshi_solution_1984,
    author    = {Aiyoshi, Eitaro and Shimizu, Kiyotaka},
    journal   = {IEEE Transactions on Automatic Control}, 
    title     = {A solution method for the static constrained Stackelberg problem via penalty method}, 
    year      = {1984},
    volume    = {29},
    number    = {12},
    pages     = {1111-1114},
    doi       = {10.1109/TAC.1984.1103455}
}

@article{dempe_simple_1987,
    author    = {Dempe, Stephan},
    title     = {A simple algorithm for the-linear bilevel programming problem},
    journal   = {Optimization},
    volume    = {18},
    number    = {3},
    pages     = {373-385},
    year      = {1987},
    doi       = {10.1080/02331938708843247},
    URL       = {https://doi.org/10.1080/02331938708843247},
    eprint    = {https://doi.org/10.1080/02331938708843247}
}

@article{bard_branch_1990,
    author    = {Bard, Jonathan F. and Moore, James T.},
    title     = {A Branch and Bound Algorithm for the Bilevel Programming Problem},
    journal   = {SIAM Journal on Scientific and Statistical Computing},
    volume    = {11},
    number    = {2},
    pages     = {281-292},
    year      = {1990},
    doi       = {10.1137/0911017}
}

@article{kolstad_derivative_1990,
    author    = {Kolstad, Charles D. and Lasdon, Leon S.},
    title     = {Derivative evaluation and computational experience with large bilevel mathematical programs},
    journal   = {Journal of Optimization Theory and Applications},
    year      = {1990},
    volume    = {65},
    number    = {3},
    pages     = {485-499},
    doi       = {10.1007/BF00939562}
}

@article{moore_mixed_1990,
    author    = {Moore, James T. and Bard, Jonathan F.},
    title     = {The Mixed Integer Linear Bilevel Programming Problem},
    journal   = {Operations Research},
    volume    = {38},
    number    = {5},
    pages     = {911-921},
    year      = {1990},
    doi       = {10.1287/opre.38.5.911}
}

@article{outrata_numerical_1990,
    author    = {Outrata, Jiri V.},
    title     = {On the numerical solution of a class of Stackelberg problems},
    journal   = {Zeitschrift f{\"u}r Operations Research},
    year      = {1990},
    volume    = {34},
    number    = {4},
    pages     = {255-277},
    doi       = {10.1007/BF01416737}
}

@article{ishizuka_double_1992,
    author    = {Yo, Ishizuka and Aiyoshi, Eitaro},
    title     = {Double penalty method for bilevel optimization problems},
    journal   = {Annals of Operations Research},
    year      = {1992},
    volume    = {34},
    number    = {1},
    pages     = {73-88},
    doi       = {10.1007/BF02098173}
}

@article{tuy_global_1993,
    author    = {Tuy, Hoang and Migdalas, Athanasios and V{\"a}rbrand, Peter},
    title     = {A global optimization approach for the linear two-level program},
    journal   = {Journal of Global Optimization},
    year      = {1993},
    volume    = {3},
    number    = {1},
    pages     = {1-23},
    doi       = {10.1007/BF01100237}
}

@article{savard_steepest_1994,
    author    = {Savard, Gilles and Gauvin, Jacques},
    title     = {The steepest descent direction for the nonlinear bilevel programming problem},
    journal   = {Operations Research Letters},
    volume    = {15},
    number    = {5},
    pages     = {265-272},
    year      = {1994},
    issn      = {0167-6377},
    doi       = {https://doi.org/10.1016/0167-6377(94)90086-8}
}

@article{liu_solving_1995,
    title     = {Solving a bilevel linear program when the inner decision maker controls few variables},
    author    = {Liu, Yi-Hsin and Spencer, Thomas H.},
    year      = {1995},
    journal   = {European Journal of Operational Research},
    volume    = {81},
    number    = {3},
    pages     = {644-651},
    doi       = {10.1016/0377-2217(94)00005-W}
}

@article{vicente_discrete_1996,
    author    = {Vicente, Luis N. and Savard, Gilles and Judice, Joaquim J.},
    title     = {Discrete linear bilevel programming problem},
    journal   = {Journal of Optimization Theory and Applications},
    year      = {1996},
    volume    = {89},
    number    = {3},
    pages     = {597-614},
    doi       = {10.1007/BF02275351}
}

@article{liu_trust_1998,
    author    = {Liu, Guoshan and Han, Jiye and Wang, Shouyang},
    title     = {A trust region algorithm for bilevel programing problems},
    journal   = {Chinese Science Bulletin},
    year      = {1998},
    volume    = {43},
    number    = {10},
    pages     = {820-824},
    doi       = {10.1007/BF03182744}
}

@article{marcotte_trust_2001,
    author    = {Marcotte, Patrice and Savard, Gilles and Zhu, Daoli L.},
    title     = {A trust region algorithm for nonlinear bilevel programming},
    journal   = {Operations Research Letters},
    volume    = {29},
    number    = {4},
    pages     = {171-179},
    year      = {2001},
    doi       = {https://doi.org/10.1016/S0167-6377(01)00092-X}
}

@article{colson_trust-region_2005,
    author    = {Colson, Beno{\^i}t and Marcotte, Patrice and Savard, Gilles},
    title     = {A Trust-Region Method for Nonlinear Bilevel Programming: Algorithm and Computational Experience},
    journal   = {Computational Optimization and Applications},
    year      = {2005},
    volume    = {30},
    number    = {3},
    pages     = {211-227},
    doi       = {10.1007/s10589-005-4612-4}
}

@article{shi_extended_2006,
    author    = {Shi, Chengge and Lu, Jie and Zhang, Guangquan and Zhou, Hong},
    title     = {An extended branch and bound algorithm for linear bilevel programming},
    journal   = {Applied Mathematics and Computation},
    year      = {2006},
    volume    = {180},
    number    = {2},
    pages     = {529-537},
    doi       = {10.1016/j.amc.2005.12.039}
}

@article{ye_new_2010,
    author    = {Ye, Jane J. and Zhu, Daoli},
    title     = {New Necessary Optimality Conditions for Bilevel Programs by Combining the MPEC and Value Function Approaches},
    journal   = {SIAM Journal on Optimization},
    volume    = {20},
    number    = {4},
    pages     = {1885-1905},
    year      = {2010},
    doi       = {10.1137/080725088}
}

@article{xu_exact_2014,
    author    = {Xu, Pan and Wang, Lizhi},
    title     = {An exact algorithm for the bilevel mixed integer linear programming problem under three simplifying assumptions},
    journal   = {Computers \& Operations Research},
    year      = {2014},
    volume    = {41},
    number    = {},
    pages     = {309-318},
    doi       = {10.1016/j.cor.2013.07.016}
}

@article{liu_enhanced_2021,
    author    = {Liu, Shaonan and Wang, Mingzheng and Kong, Nan and Hu, Xiangpei},
    title     = {An enhanced branch-and-bound algorithm for bilevel integer linear programming},
    journal   = {European Journal of Operational Research},
    year      = {2021},
    volume    = {291},
    number    = {2},
    pages     = {661-679},
    doi       = {10.1016/j.ejor.2020.10.002}
}

@article{fischer_semismooth_2022,
    author    = {Fischer, Andreas and Zemkoho, Alain B. and Zhou, Shenglong},
    title     = {Semismooth Newton-type method for bilevel optimization: global convergence and extensive numerical experiments},
    journal   = {Optimization Methods and Software},
    volume    = {37},
    number    = {5},
    pages     = {1770-1804},
    year      = {2022},
    doi       = {10.1080/10556788.2021.1977810}
}

@article{ye_difference_2022,
    author    = {Ye, Jane J. and Yuan, Xiaoming and Zeng, Shangzhi and Zhang, Jin},
    title     = {Difference of convex algorithms for bilevel programs with applications in hyperparameter selection},
    year      = {2022},
    journal   = {Mathematical Programming},
    volume    = {198},
    number    = {2},
    pages     = {1583–1616},
    doi       = {10.1007/s10107-022-01888-3}
}

@article{jolaoso_fresh_2023,
    title     = {A fresh look at nonsmooth Levenberg–Marquardt methods with applications to bilevel optimization},
    author    = {Jolaoso, Lateef O. and Mehlitz, Patrick and Zemkoho, Alain B.},
    journal   = {Optimization},
    year      = {2023},
    volume    = {74},
    number    = {12},  
    pages     = {2745 - 2792},
    doi       = {10.1080/02331934.2024.2313688}
}

@book{dempe_2002_foundations,
    author    = {Dempe, Stephan},
    title     = {Foundations of Bilevel Programming},
    series    = {Nonconvex Optimization and Its Applications},
    volume    = {61},
    publisher = {Springer},
    address   = {New York, NY},
    year      = {2002},
    doi       = {10.1007/b101970},
    isbn      = {978-1-4020-0631-9},
    isbn      = {978-0-306-48045-4},
    url       = {https://doi.org/10.1007/b101970}
}

@book{dempe_2020_bilevel,
    editor    = {Dempe, Stephan and Zemkoho, Alain},
    title     = {Bilevel Optimization},
    subtitle  = {Advances and Next Challenges},
    series    = {Springer Optimization and Its Applications},
    volume    = {161},
    publisher = {Springer},
    address   = {Cham},
    year      = {2020},
    doi       = {10.1007/978-3-030-52119-6},
    isbn      = {978-3-030-52118-9},
    isbn      = {978-3-030-52119-6},
    url       = {https://doi.org/10.1007/978-3-030-52119-6}
}

@article{amlans_dynamical_2025j,
    author    = {Sinha, Amlan and Beeson, Ryne},
    title     = {Statistical Analysis of the Role of Invariant Manifolds on Robust Trajectories},
    journal   = {Journal of Guidance, Control, and Dynamics},
    volume    = {48},
    number    = {8},
    pages     = {1818-1839},
    year      = {2025},
    doi       = {10.2514/1.G008818},
    URL       = {https://doi.org/10.2514/1.G008818},
    eprint    = {https://doi.org/10.2514/1.G008818}
}

@article{amlans_initial_2025j,
    author    = {Sinha, Amlan and Beeson, Ryne},
    title     = {Initial Guess Generation for Low-Thrust Trajectory Design with Robustness to Missed-Thrust Events},
    journal   = {Journal of Guidance, Control, and Dynamics},
    volume    = {48},
    number    = {12},
    pages     = {2831-2848},
    year      = {2025},
    doi       = {10.2514/1.G009050},
    URL       = {https://doi.org/10.2514/1.G009050},
    eprint    = {https://doi.org/10.2514/1.G009050}
}
